\documentclass[10pt]{article}

\usepackage{amssymb,amsmath,amsthm}
\usepackage{amsfonts}
\usepackage{graphicx}
\usepackage[usenames]{color}
\usepackage{caption}
\usepackage{bm,mathrsfs}
\usepackage{url}
\usepackage{subfigure}
\usepackage{algorithm,algorithmic}
\usepackage{dsfont,verbatim}
\usepackage{epsfig}
\usepackage[margin=1.2in]{geometry}
\usepackage{enumerate}
\usepackage[table]{xcolor}
\usepackage{booktabs} % 提高表格质量
\usepackage[utf8]{inputenc}
\renewcommand{\footnotemark}{}

\usepackage{hyperref}
\hypersetup{
	colorlinks=true,   % ???????
	linkcolor=blue,    % ?????????
	citecolor=blue,    % ????
	filecolor=magenta, % ????
	urlcolor=cyan      % URL ??
}

\usepackage{indentfirst}%全文缩进

\def\div{\mbox{div }}

\def\w{\widetilde}

\def\u{\mathbf{u}}

\def\v{\mathbf{v}}
\def\w{\mathbf{w}}
\def\f{\mathbf{f}}

\def\V{\mathbf{V}}

\def\c{\mathbf{c}}
\def\e{\mathbf{e}}

\def\n{\mathbf{n}}

\newcommand\diff{\,\mathrm{d}}
\graphicspath{{./pics/}}
\allowdisplaybreaks
\newtheorem{theorem}{Theorem}[section]
\newtheorem{lemma}{Lemma}[section]

\newtheorem{remark}{Remark}[section]

\numberwithin{equation}{section}

\usepackage[numbers,sort&compress]{natbib}

\title{\bf 
 Fully discrete finite element approximation for the projection method to solve the Chemotaxis-Fluid System
	}
\author{
	Chenyang Li\textsuperscript{1}
	\thanks{\textsuperscript{1}School of Mathematical Sciences, East China Normal University, Shanghai 200241, China \texttt{(chenyangli1004@yeah.net)}}
		\and  Ping Lin\textsuperscript{2}
	\thanks{\textsuperscript{2} Division of Mathematics, University of Dundee, Dundee, DD1 4HN, UK \texttt{(P.Lin@dundee.ac.uk)} 
	}
	\and  Haibiao Zheng\textsuperscript{3,*}
	\thanks{\textsuperscript{3} \textbf{Corresponding author.} School of Mathematical Sciences, Ministry of Education Key Laboratory of Mathematics and Engineering Applications, Shanghai Key Laboratory of PMMP,  East China Normal University, Shanghai 200241, China. \texttt{(hbzheng@math.ecnu.edu.cn).}   }
}
\date{}

\begin{document}

\maketitle

\begin{abstract}
In this paper, we investigate a chemotaxis-fluid interaction model governed by the incompressible Navier–Stokes equations coupled with the classical Keller–Segel chemotaxis system. 
To numerically solve this coupled system, we develop a pressure-correction finite element method based on a projection framework.
The proposed scheme employs a backward Euler method for temporal discretization and a mixed finite element method for spatial discretization. Nonlinear terms are treated semi-implicitly to enhance computational stability and efficiency. 
We further establish a rigorous error estimates for the fully discrete scheme, demonstrating the convergence of the numerical method. A series of numerical experiments are conducted to validate the stability, accuracy, and effectiveness of the proposed method. The results confirm the scheme’s capability to capture the essential dynamical behaviors and characteristic features of the chemotaxis-fluid system.
\\[5pt]
\textbf{Keywords}:
Chemotaxis fluid coupling, Mixed finite element, projection method, error estimate
\end{abstract}

\pagestyle{myheadings}
\thispagestyle{plain}

%%%%%%%%%%%%%%%%%%%%%%%%%%%%%%%%%%%%%%%%%%%%%%%%%%%%%%%%%%%%%%%%%%%%%%%%%%%%%%%%%%%%%%%%%%%%%%%%%%%%%%%%%%%%%%%%%%%%%%%%%%%%%%%%%%%%%%%%%%%%%%%%%%%%%%%%%%%%%%%%%%%%%%%%%%%%%%%%%%%%%%%%%%%%%%%%%%%%%%%%%%%%%%%%%%%%%%%%%%%%%%%%%%%%
\section{Introduction}
%%%%%%%%%%%%%%%%%%%%%%%%%%%%%%%%%%%%%%%%%%%%%%%%%%%%%%%%%%%%%%%%%%%%%%%%%%%%%%%%%%%%%%%%%%%%%%%%%%%%%%%%%%%%%%%%%%%%%%%%%%%%%%%%%%%%%%%%%%%%%%%%%%%%%%%%%%%%%%%%%%%%%%%%%%%%%%%%%%%%%%%%%%%%%%%%%%%%%%%%%%%%%%%%%%%%%%%%%%%%%%%%%%%%
Chemotaxis refers to the active and directed movement of cells in response to chemical gradients in their surrounding microenvironment. This phenomenon plays a pivotal role in a wide range of biological processes, including embryonic development, angiogenesis (the formation of new blood vessels), immune surveillance and response, as well as cancer invasion and metastasis \cite{Ridley2003}. By enabling cells to detect and respond to spatial variations in chemical signals, chemotaxis constitutes a fundamental mechanism for cellular navigation and coordination in complex biological systems.

From a mathematical perspective, the Keller–Segel system and its variants \cite{arumugam2021, keller1971} serve as foundational models for describing chemotaxis. These models relate the evolution of cell density to chemical concentration under the assumption that interactions with other physical processes, such as fluid flow, are negligible. However, in many realistic biological scenarios, cell motion can generate fluid flows through buoyancy-driven effects arising from cell density gradients. Conversely, the fluid dynamics can significantly influence the transport and distribution of cells and chemoattractants \cite{dombrowski2004, tuval2005}.
To capture these bidirectional interactions, various extended models have been developed \cite{black2019, kang2020, kozono2016, tao2015, winkler2016, winkler2019}, which couple the classical Keller–Segel chemotaxis equations with the incompressible Navier–Stokes equations. These models account for a wide range of biological phenomena, including chemoattraction and chemorepulsion, nonlinear and singular sensitivities, signal production, and multi-species chemotactic behavior.
In this work, we consider a chemotaxis–Navier–Stokes system, which seeks to determine the velocity field $\mathbf{u}$ and pressure $p$ of an incompressible fluid, together with the cell density $\eta$ and chemical concentration $c$, governed by the following system:
\begin{align}\label{cfmodel}
\begin{cases}
	\eta_t - \alpha_\eta \Delta \eta + \beta \, \text{div} (\eta \nabla c) + \u \cdot \nabla \eta = 0, & \quad \text{in } \Omega \times (0, T], \\
	c_t - \alpha_c \Delta c + \u \cdot \nabla c + \gamma \eta c = 0, &\quad \text{in } \Omega \times (0, T], \\
	\u_t + (\u \cdot \nabla) \u - \nu \Delta u + \nabla p = \eta \nabla \phi, & \quad\text{in } \Omega \times (0, T], \\
\nabla \cdot \u = 0, & \quad\text{in } \Omega \times (0, T], \\
	\frac{\partial \eta}{\partial \n} = 0, \, \frac{\partial c}{\partial \n} = 0, \, \u = 0, &\quad \text{on } \partial \Omega \times (0, T], \\
	\eta(\cdot, 0) = \eta_0, \, c(\cdot, 0) = c_0, \, \u(\cdot, 0) = \u_0, &\quad \text{in } \Omega.
\end{cases}
\end{align}
%where $\Omega \subset \mathbb{R}^2$ is a bounded domain with Lipschitz continuous boundary $\partial \Omega$, $\n$ is the outward unit-normal vector. The unknowns $\u,p,\eta,c$ are the velocity and pressure of fluid, the cell density, the concentration of chemical substances respectively, the parameter $ \nu, a_\eta,a_c,\gamma$ denote the viscosity of the fluid, the diffusion of the cell, the diffusion of chemical substance, and  the consumption rate of the chemical signal. The parameter $\beta$ denotes the chemotactic sensitivity coefficient, quantifying the strength of a cell's response to spatial gradients in chemical concentration.

Here, $\Omega \subset \mathbb{R}^2$ denotes a bounded domain with Lipschitz continuous boundary $\partial \Omega$, and $\mathbf{n}$ represents the outward unit normal vector on $\partial \Omega$. The primary unknowns in the system are the fluid velocity field $\mathbf{u}$, the fluid pressure $p$, the cell density $\eta$, and the chemical concentration $c$.
The model involves several physical parameters:
\begin{itemize}
\item $\nu$ is the kinematic viscosity of the fluid,
\item $a_\eta$ and $a_c$ denote the diffusion coefficients for the cell density and the chemical concentration, respectively,
\item $\beta$ is the chemotactic sensitivity coefficient, characterizing the strength of the cells’ directed movement in response to spatial gradients in the chemical concentration.
\item $\gamma$ represents the rate at which cells consume the chemical signal,
\end{itemize}

These parameters collectively govern the coupled dynamics between chemotactic cell motion and fluid flow, capturing essential features of biological transport processes in complex environments.

The chemotaxis–fluid system, typically modeled by coupling the Keller–Segel equations with the incompressible Navier–Stokes equations, has been the subject of extensive theoretical investigation. Significant progress has been made concerning the existence, uniqueness, asymptotic behavior, and blow-up phenomena under supercritical mass conditions. These analytical properties have been rigorously studied in a series of works, such as \cite{jiang2015, winkler2016, winkler2019, lorz2010, winkler2014, zhang2015, lorz2012, winkler2020, bellomo2015}.
On the numerical side, various schemes have been proposed to simulate the chemotaxis–fluid system effectively and reliably. In particular, a critical consideration in the design of numerical methods is the preservation of key physical and structural properties of the system, such as mass conservation, positivity of the cell density and chemical concentration, and energy dissipation or stability. Several approaches that respect these properties have been developed and analyzed in the literature \cite{chertock2008, saito2007, strehl2013, liu2018, xiao2019, shen2020, zhao2020, huang2020, zhang2016, zhang20161, huang2021, guillen2019, jiang2022, wang2022}.
For instance, a conservative upwind finite element method for the classical chemotaxis equations that guarantees both positivity and mass conservation was proposed in \cite{satio2011}, where the error estimates were derived using the theory of analytic semigroups. 
Although a fully decoupled, linear, and positivity--preserving scheme was developed in another study for the chemotaxis--Stokes system by combining a non--incremental pressure-correction method with the flux--corrected transport (FCT) technique \cite{huang2021}, its neglect of the fluid convection term consequently limits applicability to more general flow regimes.
To address challenges related to nonlinear stabilization, \cite{feng2021} considered a regularized version of the chemotaxis–Stokes system. By transforming the algebraic-based algorithm into a Galerkin variational formulation incorporating nonlinear stabilization terms, the method effectively avoids the difficulty of deriving high regularity assumptions for the underlying spatial discretization.
Furthermore, by introducing the gradient of the chemical concentration as an auxiliary variable, mixed finite element methods have been designed for chemotaxis systems. These approaches enable the establishment of optimal error estimates and improved stability properties \cite{duarte2021, guillen2020, guillen20201, beltran2023}. In addition, a high-resolution hybrid finite-volume/finite-difference method based on vorticity formulation was developed in \cite{chertock2012} to investigate the nonlinear dynamic behavior of two-dimensional chemotaxis–fluid models.

In the chemotaxis fluid system, the inherent multi-physics coupling between cellular chemotaxis, chemical diffusion, and incompressible fluid dynamics poses significant numerical challenges. In particular, accurately preserving the divergence-free condition for the velocity field is delicate and crucial. To handle the nonlinearities introduced by advection and chemotactic drift terms, we adopt a semi-implicit time discretization, whereby the nonlinear terms are treated explicitly or lagged in time, while the linear terms are treated implicitly. This strategy allows us to avoid solving nonlinear systems at each time step, thus significantly reducing computational cost and complexity—an essential consideration when designing schemes for coupled nonlinear PDE systems.
To enforce the incompressibility constraint efficiently, we employ the classical projection method, originally introduced by Chorin \cite{chorin1968} and Temam \cite{teman1969}. The core idea of the projection method is to decouple the computation of velocity and pressure by splitting the Navier–Stokes equations into two sub-steps:
\begin{itemize}
\item First, an intermediate velocity field is computed without enforcing the divergence-free condition.

\item Then, a pressure-correction step is applied, which projects the velocity onto the subspace of solenoidal (divergence-free) vector fields via a Poisson equation for pressure.
\end{itemize}

This method avoids the saddle-point structure of the fully coupled velocity-pressure system, thus enabling more efficient and stable time-stepping schemes.
Projection-type methods have been widely employed and extended in the numerical analysis of incompressible flows and multi-physics problems \cite{shen1992, shen19921, shen1994, shen1996, guermond2006, achdou2000, frutos2019, cheng2023, ding2023, long2023, liyuan2024}. Their robustness and computational efficiency make them particularly attractive in the context of chemotaxis–fluid interactions.

%
%In the Chemotaxis-Fluid system, multi-physics coupling and preserving the divergence-free condition can be delicacy. The semi-implicit method is applied to treat the nonlinear term so that such that we only solve
%linear systems at each time step, which is an important issue in designing numerical schemes
%for the multi-physical nonlinear problems.  
%The projection method proposed by Chorin \cite{chorin1968} and  Temam \cite{teman1969}  is employed to to decouple the computation of velocity and pressure.
%The main idea of the projection method is firstly to compute a
%velocity field without taking into account the divergence-free condition, and then perform a pressure correction, which is a projection back to the subspace of solenoidal vector fields  \cite{liyuan2024}. This technique has been used to solve more models \cite{shen1992,shen19921,shen1994,shen1996,guermond2006,frutos2019,achdou2000,cheng2023,ding2023,long2023,liyuan2024}.

%%%%%%%%%%%%%%%%%%%%%%%%%%%%%%%
According to the previous discussion, on one hand, to
fill the aforementioned gap, and on the other hand, to further extend the applicability of pressure-correction projection methods to the study of chemotaxis–Navier–Stokes coupling.
%%%%%%%%%%%%%%%%%%%%%%%%%%%%%%%
To this end, we propose a first-order projection-based finite element scheme for the chemotaxis–fluid system and provide a rigorous convergence and error analysis. Specifically, we employ the Mini element for the velocity-pressure pair, which satisfies the inf-sup condition, ensuring stability in the incompressible flow approximation. For the cell density and chemical concentration, we use standard linear Lagrange finite elements. Nonlinear terms arising in the coupled system are treated via a semi-implicit time discretization, which allows us to avoid solving nonlinear systems at each time step, thereby reducing computational complexity while retaining stability and accuracy. To the best of our knowledge, there exists no prior work that systematically develops a finite element scheme based on the projection method for the fully coupled chemotaxis–Navier–Stokes system, along with a corresponding theoretical error estimate.

The structure of the paper is as follows: Section \ref{notation} introduces the functional spaces, notation, and several fundamental inequalities that will be used throughout the analysis. Section \ref{scheme} presents the projection finite element scheme and states the main convergence theorem. Section \ref{erroranalysis} is devoted to the detailed derivation of the error estimates, culminating in the proof of the main theorem. Section \ref{numericalresult} provides a series of numerical experiments that validate the theoretical findings and illustrate the robustness and accuracy of the method. Section \ref{conclusion} concludes the paper with a summary of results and a discussion of possible future research directions. Throughout this paper, we use the symbol $C$ to denote a general positive constant which can be different at different places and is independent of the time step size $\tau$ and the mesh size $h$.

%
%
%We propose a first-order projection finite element scheme for the
%Chemotaxis-Fluid system and give the corresponding complete convergent analysis. In the finite element discretization,  the velocity field and pressure field are approximated by Mini element, the cell density and the concentration of substance are approximated by the linear  Lagrangian element. 
%Nonlinear terms are handled by the semi-implicit method which significantly simplifies the computational process while
%maintaining the accuracy and stability of the solution.  
%Up to our knowledge, no
%works available in the literature in which finite element methods for
%approximating the solutions of the full chemotaxis–Navier–Stokes system by projection method, including corresponding error estimates. 

%The structure of this paper is outlined as follows. In section \ref{notation}, we define some notations and some known inequalities. In section  \ref{scheme}, we give the projection finite element method for the Chemotaxis-Fluid system and present the main theorem in this paper. A complete error analysis is proposed to prove the main theorem in Section  \ref{erroranalysis}. Numerical results are presented to confirm our theoretical analysis in Section \ref{numericalresult}. The last section
%is devoted to concluding remarks.
%%%%%%%%%%%%%%%%%%%%%%%%%%%%%%%

%%%%%%%%%%%%%%%%%%%%%%%%%%%%%%%%%%%%%%%%%%%%%%%%%%%%%%%%%%%%%%%%%%%%%%%%%%%%%%%%%%%%%%%%%%%%%%%%%%%%%%%%%%%%%%%%%%%%%%%%%%%%%%%%%%%%%%%%%%%%%%%%%%%%%%%%%%%%%%%%%%%%%%%%%%%
\section{Notation and Preliminaries}\label{notation}
%%%%%%%%%%%%%%%%%%%%%%%%%%%%%%%%%%%%%%%%%%%%%%%%%%%%%%%%%%%%%%%%%%%%%%%%%%%%%%%%%%%%%%%%%%%%%%%%%%%%%%%%%%%%%%%%%%%%%%%%%%%%%%%%%%%%%%%%%%%%%%%%%%%%%%%%%%%%%%%%%%%%%%%%%%%%%%%%%%%%%%%%%%%%%%%%%%%%%%%%%%%%%%%%%%%%%%%%%%%%%%%%%%%%
For $k\in N^+$ and $1\leq p\leq +\infty$, we denote $L^p(\Omega)$ and $W^{k, p}(\Omega)$ as the classical Lebesgue space and Sobolev space, respectively. The norms of these spaces are denoted by
\begin{align*}
	||u||_{L^p(\Omega)}&=\left(\int_{\Omega}|u(\mathbf{x})|^p\diff \mathbf{x}\right)^\frac{1}{p},\\
	||u||_{W^{k,p}(\Omega)}&=\left(\sum\limits_{|j|\leq k}||D^ju||_{L^p(\Omega)}^p\right)^\frac{1}{p}.
\end{align*}
within this context, $W^{k, 2}(\Omega)$ is also known as the Hilbert space and can be expressed as $H^k(\Omega)$.  $||\cdot||_{L^\infty}$ represents the norm of the space  $L^\infty(\Omega)$ which is defined as
\begin{equation*}
	||u||_{L^\infty(\Omega)}=ess\sup\limits_{\mathbf{x}\in \Omega}|u(\mathbf{x})|.
\end{equation*}

For simplicity, we denote the inner products of both
$L^2(\Omega)$ and $\textbf{L}^2(\Omega)$
by $(\cdot,\cdot)$, and use $\langle \cdot,\cdot \rangle$ to denote the dual product of $H^{-1}(\Omega) \times H^1_0(\Omega)$. namely,
\begin{align*}
	\begin{split}
		&(u,v)=\int_\varOmega u(x)v(x) d x \quad \forall  \, u,v\in L^2(\Omega),\\
		&({\bf u},{\bf v})=\int_\Omega {\bf u}(x)\cdot {\bf v}(x) dx \quad \forall \, {\bf u},{\bf v}\in {\bf L}^2(\Omega) .\\
	\end{split}
\end{align*}
%
%For simplicity, the $L^{p}(\Omega)$  norm is indexed by $\|\cdot\|_{L^{p}}$  with the special cases of $L^{2}(\Omega)$ and $L^{\infty}(\Omega)$ norms being written as $\|\cdot\|$ and $\|\cdot\|_{\infty}$.  For $ m\in N$, we denote the norm associated with the Sobolev space $W^{m,p}(\Omega)$ by $\|\cdot\|_{W^{m,p}}$, with the special case $W^{m,2}(\Omega)$ being written as $H^m(\Omega)$ with the norm $\|\cdot\|_m$.

Furthermore, we define the following frequently
utilized mathematical frameworks (d=2):
\begin{align*}
	 \textbf{V} =&H_0^1(\Omega)^d = \left\{ \v \in L^2(\Omega)^d : \nabla \v \in L^2(\Omega)^{d \times d} \text{ and } \v = 0 \text{ on } \partial\Omega \right\},\\ \textbf{V}_0=& \{ \v \in V, \nabla \cdot \v =0 \quad \text{in }\Omega  \},\\
	 \textbf{M} =& L^2_0 (\Omega) = \{  q\in L^2_0(\Omega), \int_{\Omega} q dx =0  \},\\
	 X=& H_0^1(\Omega) = \left\{ r \in L^2(\Omega) : \nabla r \in L^2(\Omega) \text{ and } r = 0 \text{ on } \partial\Omega \right\},\\
	   X_0=&\{ \eta \in H^1_0(\Omega), \int_{\Omega} \eta dx =0 \}.
	 \end{align*}

The transport terms have the diﬃculty that the corresponding discrete forms do not preserve the alternance property as in the continuous case, and we need to introduce some
skew-symmetric trilinear forms to complete error estimate  by
\begin{align}
	\begin{split}
		B(\textbf{u},\textbf{v},\textbf{w}) &= \int_\Omega (\u \cdot \nabla) \textbf{v} \cdot\textbf{w} dx + \frac{1}{2} \int_{\Omega} (\nabla \cdot \u) \v \cdot \w dx \\
		&=   \frac{1}{2}\int_\Omega (\u \cdot \nabla) \textbf{v} \cdot\textbf{w} dx -  \frac{1}{2} \int_{\Omega} (\u \cdot \nabla) \w \cdot\v dx  \quad \forall ~\u,\v,\w\in\textbf{V},\\
		b(\u,c,r) &=  \int_\Omega (\u \cdot \nabla c)r dx+ \frac{1}{2} \int_{\Omega} (\nabla \cdot \u) crdx, \\
		&=\frac{1}{2}\int_\Omega (\u \cdot \nabla c) r dx -  \frac{1}{2} \int_{\Omega} (\u \cdot \nabla r) c dx \quad \forall ~\u \in \V, \,  \forall \,c,r \in X.
	\end{split}
\end{align}
which has the following properties \cite{heyinnian2005,heyinnian2022}
\begin{align}
	& B(\textbf{u},\textbf{v},\textbf{v})=0, \quad b(\u,r,r)=0,\label{biobdf-2}\\
	&B(\textbf{u},\textbf{v},\textbf{w}) = - B(\textbf{u},\textbf{w},\textbf{v}),\quad b(\textbf{u},c,r) = - b(\textbf{u},r,c),\label{biobdf-17}\\
	&B(\textbf{u},\textbf{v},\textbf{w}) \leq C \| \nabla \textbf{u} \|_{L^2} \| \nabla \textbf{v} \| _{L^2} \| \nabla \textbf{w} \|_{L^2},\quad b(\textbf{u},c,r) \leq C \| \nabla \textbf{u} \|_{L^2} \| \nabla c \| _{L^2} \| \nabla r \|_{L^2},\\
	&B(\textbf{u},\textbf{v},\textbf{w})  \leq C \| \textbf{u} \| ^{\frac{1}{2}}_{L^2} \| \nabla \textbf{u}\| ^{\frac{1}{2}}_{L^2} \| \nabla \textbf{v} \| _{L^2} \| \nabla \textbf{w} \|_{L^2}, \quad  b(\textbf{u},c,r)  \leq C \| \textbf{u} \| ^{\frac{1}{2}}_{L^2} \| \nabla \textbf{u}\| ^{\frac{1}{2}}_{L^2} \| \nabla c \| _{L^2} \| \nabla r \|_{L^2}.
\end{align}

If $\nabla \cdot \u=0$, there holds $B(\textbf{u},\textbf{v},\textbf{w}) = ((\textbf{u} \cdot \nabla) \textbf{v},\textbf{w})= \int_\Omega (\u \cdot \nabla) \textbf{v}\cdot\textbf{w} dx $ and $b(\u,c,r) = (\textbf{u} \cdot \nabla c,r) = \int_\Omega (\u \cdot \nabla c)r dx$.

It is well known that system (\ref{cfmodel}) conserves in time the total mass of bacteria \cite{satio2011}, that is,
$$
 \int_{\Omega} \eta(x,t) dx =  \int_{\Omega} \eta_0 dx .
$$

Set $m_0 = \frac{1}{| \Omega|} \int_{\Omega} \eta_0 dx >0$ and take  $\tilde{\eta} = \eta-m_0$ as  the new variable, then the original problem  (\ref{cfmodel}) can be written by the following forms.
\begin{align}\label{recfmodel}
	\begin{cases}
		\tilde{\eta}_t- \alpha_\eta \Delta \tilde{\eta} + \beta \, \text{div} ((\tilde{\eta}+m_0) \nabla c) + \u \cdot \nabla \tilde{\eta} = 0, & \quad \text{in } \Omega \times (0, T], \\
		c_t - \alpha_c \Delta c + \u \cdot \nabla c + \gamma (\tilde{\eta}+m_0) c = 0, &\quad \text{in } \Omega \times (0, T], \\
		\u_t + (\u \cdot \nabla) \u - \nu \Delta \u + \nabla p = (\tilde{\eta}+m_0) \nabla \phi, & \quad\text{in } \Omega \times (0, T], \\
		\nabla \cdot \u = 0, & \quad\text{in } \Omega \times (0, T], \\
		\frac{\partial \tilde{\eta}}{\partial \n} = 0, \, \frac{\partial c}{\partial \n} = 0, \, \u = 0, &\quad \text{on } \partial \Omega \times (0, T], \\
		\tilde{\eta}(\cdot, 0) =\tilde{\eta}_0, \, c(\cdot, 0) = c_0, \, \u(\cdot, 0) = \u_0, &\quad \text{in } \Omega.
	\end{cases}
\end{align}

We present the main error estimate of $ ( \u,\eta,c)$ under the following regularity assumptions.

\textbf{Assumption A1}: Assume the exact solutions $ ( \u,p,\eta,c)$  satisfy
\begin{align}\label{cf-regularity}
	\begin{cases}
	\| \phi \|_{H^2} \leq C,\\
	\u \in L^\infty (0,T; W^{2,4}), \quad  \u_t \in L^{2}(0,T;H^2), \quad \u_{tt} \in L^2(0,T;L^2(\Omega)),\\
	p \in L^{\infty}(0,T;H^1), \quad p_t \in L^2(0,T; H^1),\\
	\tilde{\eta} \in L^{\infty}({0,T; W^{2,4}}), \quad \tilde{\eta}_t \in L^{2}(0,T;L^2 \cap H^2), \quad \tilde{\eta}_{tt} \in L^2(0,T;L^2),\\
	c\in L^{\infty} (0,T;W^{2,4}), \quad   c_t \in L^{2}(0,T;H^2), \quad c_{tt }\in L^2(0,T;L^2).
			\end{cases}
\end{align}

The following discrete Gronwall inequality established in \cite{evance2022,heywood1990} will be used frequently.
\begin{lemma} (Discrete Gronwall's inequality ) Let $a_k , b_k$ and $y_k$ be the nonnegative numbers such that \label{biobdf-11}
	\begin{align}\label{growninequality-discrete}
		a_n+ \tau \sum \limits^n \limits_{k=0}  b_k\leq \tau \sum \limits^n \limits_{k=0} \gamma _k a_k + B \quad \text{for} \,\, n \geq 1,
	\end{align}
	Suppose $\tau \gamma _k \leq 1$ and set $\sigma_k = (1-\tau \gamma_k) ^{-1}$. Then there holds
	\begin{align}\label{grown2}
		a_n + \tau \sum \limits^n \limits_{k=0} b_k \leq exp(\tau \sum \limits^n \limits_{k=0} \gamma_k \sigma_k) B \quad  \text{for} \,\, n\geq 1.
	\end{align}
	\begin{remark}
		If the sum on the right-hand side of (\ref{growninequality-discrete}) extends only up to $n-1$, then the estimate (\ref{grown2}) still holds for all $k \geq 1$ with $\sigma_k=1$.
	\end{remark}
	
\end{lemma}

The following Sobolev inequalities in 2D will be used frequently. \cite{evance2022}
\begin{align}\label{cf-inequality}
	\begin{cases}
		&W^{2,4}(\Omega)  \hookrightarrow W^{1,\infty}(\Omega),\\ 
		&H^2(\Omega) \hookrightarrow W^{1,q}(\Omega), \quad 2 \leq q < \infty,\\
		&H^2(\Omega) \hookrightarrow L^{\infty}(\Omega).
	\end{cases}
\end{align}

%%%%%%%%%%%%%%%%%%%%%%%%%%%%%%%%%%%%%%%%%%%%%%%%%%%%%%%%%%%%%%%%%%%%%%%%%%%%%%%%%%%%%%%%%%%%%%%%%%%%%%%%%%%%%%%%%%%%%%%%%%%%%%%%%%%%%%%%%%%%%%%%%%%%%%%%%%%%%%%%%%%%%%%%%%%%%%%%%%%%%%%%%%%%%%%%%%%%%%%%%%%%%%%%%%%%%%%%%%%%%%%%%%%%
\section{Pressure-correction projection finite element methods and  theoretical results}\label{scheme}
%%%%%%%%%%%%%%%%%%%%%%%%%%%%%%%%%%%%%%%%%%%%%%%%%%%%%%%%%%%%%%%%%%%%%%%%%%%%%%%%%%%%%%%%%%%%%%%%%%%%%%%%%%%%%%%%%%%%%%%%%%%%%%%%%%%%%%%%%%%%%%%%%%%%%%%%%%%%%%%%%%%%%%%%%%%%%%%%%%%%%%%%%%%%%%%%%%%%%%%%%%%%%%%%%%%%%%%%%%%%%%%%%%%%
In this section, we present the fully discrete finite element method and state the main results for the Chemotaxis-Fluid system. 
Let \(\mathcal{J}_h\) denote a quasi-uniform triangulation of the domain \(\Omega \subset \mathbb{R}^2\) into triangular elements \(K_j\), \(j=1,2,\dots,M\), with mesh size 
$
h = \max_{1 \leq j \leq M} \{\mathrm{diam}(K_j)\}.
$We define the finite element spaces for the velocity field, pressure field, cell density, and chemical concentration as follows:
\begin{align*}
	\mathbf{V}_h &= \{ \mathbf{v}_h \in \mathbf{V}: \mathbf{v}_h|_K \in (P_1(K) + b(K))^2, \quad \forall K \in \mathcal{J}_h \}, \\
	M_h &= \{ q_h \in M: q_h|_K \in P_1(K), \quad \forall K \in \mathcal{J}_h \}, \\
	X_h &= \{ \psi_h \in X: \psi_h|_K \in P_1(K), \quad \forall K \in \mathcal{J}_h \}, \\
	X_{0h} &= \{ w_h \in X_0: w_h|_K \in P_1(K), \quad \forall K \in \mathcal{J}_h \},
\end{align*}
where \(P_1(K)\) denotes the space of polynomials of degree at most one on the element \(K\), and \(b(K)\) is the standard bubble function, which is strictly positive in the interior of \(K\), takes the value one at the barycenter of \(K\), and vanishes on the boundary \(\partial K\) \cite{boffi2013,girault1986}. The pair \((\mathbf{V}_h, M_h)\) corresponds to the Mini element for velocity and pressure, while the cell density \(\eta\) and chemical concentration \(c\) are approximated by the standard linear Lagrangian finite elements.

The discrete divergence-free velocity space is defined by
\begin{align}\label{divergence-free}
	\mathbf{V}_{0h} := \{ \mathbf{v}_h \in \mathbf{V}_h : (\nabla \cdot \mathbf{v}_h, q_h) = 0, \quad \forall q_h \in M_h \}.
\end{align}

The inverse inequality will be used frequently \cite{boffi2013}.
\begin{align}
	\| \u_h \|_{W^{m,q}} \leq C h^{  l-m + n( \frac{1}{q} -\frac{1}{p}  )} \| \u_h\| _{W^{l,p}}, \quad \forall~ \u_h \in \V_h, \label{biobdf-16}\\
	\| c_h \|_{W^{m,q}} \leq C h^{  l-m + n( \frac{1}{q} -\frac{1}{p}  )} \| c_h\| _{W^{l,p}}, \quad \forall~ c_h \in X_h\, or\, X_{0h},\label{biobdf-21}
\end{align}

For temporal discretization, consider a uniform partition of the time interval \([0,T]\) given by \( t_n = n \tau \), \( n=0,1,\dots,N \), where \(\tau = T/N\) is the time step size. For any time-dependent function \(v(x,t)\), denote \(v^n := v(x, t_n)\) and define the backward difference quotient
\[
D_\tau v^{n+1} := \frac{v^{n+1} - v^n}{\tau}.
\]

We define the iteration initial value $\u^0_h=I_{1h} \u^0, \, p^0_h = I_{2h} p^0, \, \eta^0_h = I_{3h}\eta_0,\,  \c^0_h = I_{4h}c^0$, where $(I_{1h}, I_{2h},I_{3h},I_{4h})$ are the interpolation operators into $(\V_h \times M_h \times X_{0h}\times X_h)$, respectively, then we have the following approximation error:
\begin{align}
	\| \u^0 - I_{1h}\u^0 \|_{L^2} + h \| \nabla (\u^0- I_{1h}\u^0)\|_{L^2} \leq C h^2 \| \u^0\|_{H^2},\label{cf-uinterpolation}\\
	\| p^0 - I_{2h}p^0 \|_{L^2}\leq C h \| p^0\|_{H^1},\\
	\| \tilde{\eta}^0 - I_{3h}\tilde{\eta}^0\|_{L^2} +h\| \nabla ( \tilde{\eta}^0 - I_{3h}\tilde{\eta}^0)\|_{L^2}\leq C h^2 \| \tilde{\eta}^0\|_{H^2},\label{cf-etainterpolation}\\
	\| c^0 - I_{4h}c^0 \|_{L^2}+h\| \nabla ( c^0 - I_{4h}c^0)\|_{L^2} \leq C h^2 \| c^0\|_{H^2}.\label{cf-cinterpolation}
\end{align}

Starting from given initial approximations \((\mathbf{u}_h^0, p_h^0, \tilde{\eta}_h^0, c_h^0)\), the fully discrete pressure-correction projection method for the Chemotaxis-Fluid system \eqref{recfmodel} reads as follows:

\textbf{The first-order pressure-correction finite element method:}

\textbf{Step 1: } For given $(\u^{n}_h, p^n_h, \tilde{\eta}^n_h)$, we find $\u^{n+1}_h\in\V_h$ by 
\begin{align}\label{cf-scheme1}
	( \frac{\tilde{\u}_h^{n+1} -\u_h^n  }{\tau},\v_h   ) + \mu (\nabla \tilde{\u}^{n+1}_h,\v_h)+ B(\u^n_h, \tilde{\u}^{n+1}_h,\v_h) + (\nabla p^n_h, \v_h) \\
	= ( ( \tilde{\eta}^n_h+m_0) \nabla \phi^{n+1},\v_h  ), \quad \forall \v_h \in \V_h.\notag
\end{align}

\textbf{Step 2:}
Find $p^{n+1}_h \in M_h$ by 
\begin{align}\label{cf-scheme2}
	(\nabla p^{n+1}_h,\nabla q_h) =(\nabla p^n_h,\nabla q_h) - \frac{1}{2\tau} (\nabla \tilde{\u}^{n+1}_h,q_h), \quad \forall q_h \in M_h.
\end{align}

\textbf{Step 3:} 
Then we update $\u^{n+1}_h$ by 

\begin{align}\label{cf-scheme3}
	\u^{n+1}_h =\tilde{\u}^{n+1}_h-2 \tau ( \nabla p^{n+1}_h - \nabla p^n_h  )
\end{align}

\textbf{Step 4:} For given $(\u^n_h, \tilde{\eta}^n_h)$, we can get $\tilde{\eta}^{n+1}_h$ by

\begin{align}\label{cf-scheme4}
	(\frac{\tilde{\eta}^{n+1}_h - \tilde{\eta}^n_h}{\tau} , w_h) + a_\eta ( \nabla \tilde{\eta}^{n+1}_h, \nabla w) - \beta ( (\tilde{\eta}^n_h+m_0) \nabla c^{n+1}_h, \nabla w_h   )\\
	 +
	b(\u^n_h,  \tilde{\eta}^{n+1}_h , w_h)=0, \quad \forall w_h \in X_{0h}. \notag
\end{align}

\textbf{Step 5 :} Then we update $\eta^{n+1}_h$ by 

\begin{align}\label{cf-scheme5}
	\eta^{n+1}_h = \tilde{\eta}^{n+1}_h +m_0.
\end{align}

\textbf{Step 6:} For given $(\u^n_h,\tilde{\eta}^n_h)$, $c^{n+1}_h$ can be derived by

\begin{align}\label{cf-scheme6}
	(  \frac{c^{n+1}_h - c^n_h}{\tau},\psi_h   )+ a_c (\nabla c^{n+1}_h,\nabla \psi_h)+b(\u^n_h, c^{n+1}_h,\psi_h) \\
	+ \gamma ( (\tilde{\eta}^n+m_0) c^{n+1}_h,\psi_h )=0, \quad \forall \psi_h \in X_h.\notag
\end{align}

\begin{remark}
	From \textbf{Step 5}, it follows that the finite element approximation $\eta_h^{n+1}$ preserves the global mass conservation property, i.e., 
	\begin{align}
		\int_{\Omega} \eta_h^{n+1} \, dx = m_0 |\Omega|.
	\end{align}
	
	This conservation property is fundamental in ensuring the physical fidelity of the numerical solution over time.
\end{remark}

 Additionally, by combing the second-order backward differentiation formula (BDF2) and projection method, we also construct a fully decoupled,  linear and mass-conversation  scheme for the system (\ref{recfmodel}). Denote $
\hat{v}^{n}=2 v^n-v^{n-1}, \quad \hat{\tilde{\eta}}^n=2 \tilde{\eta}^n-\tilde{\eta}^{n-1}.
$ Given initial approximations \((\mathbf{u}_h^0, p_h^0, \eta_h^0, c_h^0)\), \((\mathbf{u}_h^1, p_h^1, \eta_h^1, c_h^1)\) can be derived by the first-order scheme (\ref{cf-scheme1})--(\ref{cf-scheme6}), having computed \((\mathbf{u}_h^{n-1}, p_h^{n-1}, \eta_h^{n-1}, c_h^{n-1})\) and \((\mathbf{u}_h^n, p_h^n, \eta_h^n, c_h^n)\),  the second-order fully discrete pressure-correction projection method for the Chemotaxis-Fluid system \eqref{recfmodel} reads as follows:

\textbf{The second-order pressure-correction finite element method:}

\textbf{Step 1: } For given $(\hat{\u}^{n}_h, p^n_h, \hat{\tilde{\eta}}^n_h)$, we find $\u^{n+1}_h$ by 
\begin{align}\label{cf-scheme11}
	(\frac{2 \tilde{\u}_h^{n+1}-3 \u^n_h+\u^{n-1}_h}{2\tau},\v_h   ) + \mu (\nabla \tilde{\u}^{n+1}_h,\v_h)+ B(\hat{\u}^n_h, \tilde{\u}^{n+1}_h,\v_h) + (\nabla p^n_h, \v_h)\\
	= ( ( \hat{\tilde{\eta}}^n_h+m_0) \nabla \phi^{n+1},\v_h  ), \quad \forall \v_h \in \V_h.\notag
\end{align}

\textbf{Step 2:}
Find $p^{n+1}_h \in M_h$ by 
\begin{align}\label{cf-scheme22}
	(\nabla (p^{n+1}_h- p^n_h),\nabla q_h)= - \frac{3}{2\tau} (\nabla \tilde{\u}^{n+1}_h,q_h), \quad \forall q_h \in M_h.
\end{align}

\textbf{Step 3:} 
Then we update $\u^{n+1}_h$ by 

\begin{align}\label{cf-scheme33}
	( \frac{3\u^{n+1}_h -3\tilde{\u}^{n+1}_h}{2\tau} , \v_h ) +( \nabla p^{n+1}_h - \nabla p^n_h ,\v_h )=0.
\end{align}

\textbf{Step 4:} For given $(\hat{\u}^n_h, \hat{\tilde{\eta}}^n_h)$, we can get $\tilde{\eta}^{n+1}_h$ by

\begin{align}\label{cf-scheme44}
	(\frac{2 \tilde{\eta}_h^{n+1}-3 \tilde{\eta}^n_h+\tilde{\eta}^{n-1}_h}{2\tau} , w_h) + a_\eta ( \nabla \tilde{\eta}^{n+1}_h, \nabla w) - \beta ( (\hat{\tilde{\eta}}^{n}_h+m_0) \nabla c^{n+1}_h, \nabla w_h   ) \\
	+
	b(\hat{\u}^n_h,  \tilde{\eta}^{n+1}_h , w_h)=0, \quad \forall w_h \in X_{0h}. \notag
\end{align}

\textbf{Step 5 :} Then we update $\eta^{n+1}_h$ by 

\begin{align}\label{cf-scheme55}
	\eta^{n+1}_h = \tilde{\eta}^{n+1}_h +m_0.
\end{align}

\textbf{Step 6:} For given $(\hat{\u}^n_h,\hat{\tilde{\eta}}^n_h)$, $c^{n+1}_h$ can be derived by

\begin{align}\label{cf-scheme66}
	(  \frac{3 c^{n+1}_h - 4c^n_h+ c^{n-1}_h}{2\tau},\psi_h   )+ a_c (\nabla c^{n+1}_h,\nabla \psi_h)+b(\hat{\u}^n_h, c^{n+1}_h,\psi_h) \\
	+ \gamma ( (\hat{\tilde{\eta}}^n+m_0) c^{n+1}_h,\psi_h )=0, \quad \psi_h \in X_h.\notag
\end{align}

%
%\begin{remark}
%	By taking the divergence of (\ref{cf-scheme33}), we has
%	\begin{align}
%		-\Delta p_h^{n+1}= - \frac{3}{2\tau} \nabla \cdot \tilde{\u}_h^{n+1} - \Delta p_h^n,
%	\end{align}
%	with the Neumann boundary condition $\partial _\n (p^{n+1}-p^n)=0$. Thus we can derive $\u^{n+1}$ by 
%	\begin{align}
%		\u^{n+1}_h= \tilde{\u}^{n+1}_h- \frac{2\tau}{3} \nabla (p_h^{n+1} -p_h^n).
%	\end{align}
%\end{remark}

As the analytical framework for the second-order BDF2 scheme is essentially the same as that for the first-order back Euler scheme, we therefore concentrate on the study of the first-order projection method in this work.

We introduce error function as follows:
\begin{align}
	\u^n-\tilde{\u}^n_h&=\u^n-R_h \u^n +R_h \u^n -\tilde{\u}^n_h := \theta^n_u +\tilde{e}_{\u}^n,\\
		\u^n-\u^n_h&=\u^n-R_h \u^n +R_h \u^n -\u^n_h := \theta^n_u +e_{\u}^n,\\
			p^n-p^n_h&=p^n-K_h p^n +K_h p^n -p^n_h := \theta^n_p +e_{p}^n,\\
			\tilde{\eta}^n-\tilde{\eta}^n_h&=\tilde{\eta}^n-T_h  \tilde{\eta}^n_h +T_h  \tilde{\eta}^n_h -\tilde{\eta}^n_h := \theta^n_{\eta} +\tilde{e}_{\eta}^n,\\	c^n-c^n_h&=c^n-\Pi_h c^n +\Pi_h c^n -c^n_h := \theta^n_c +e_{c}^n,
\end{align}
where $(R_h,K_h) : \V \times M \rightarrow \V_h \times M_h$ are the Stokes projection operators defined by 
\begin{align}
	( \nabla (\u-R_h \u ), \nabla \v_h)-(\nabla \cdot \v_h, p-K_h p)=0, \quad \forall (\v_h, q_h) \in (\V_h \times M_h).
\end{align}
and the Ritz projection operators defined by 
\begin{align}
	( \nabla (\tilde{\eta} -T_h \tilde{\eta}),  \nabla w_h)&=0, \quad \forall  w_h \in X_{0h},\\
	( \nabla (c -\Pi_h c),  \nabla \psi_h)&=0, \quad \forall  \psi_h \in X_{h}.
\end{align}

By the regularity assumption (\ref{cf-regularity}), we can have the following approximation  properties:
\begin{align}
	\| \u^n - R_h \u^n \|_{L^2}+h \| p^n-K_h p^n\|_{L^2} &\leq C h^2(\| \u^n\|_{H^2}  +\| p^n\|_{H^1}),\label{cf-uprojection}\\
	\| \tilde{\eta}^n-T_h \tilde{\eta}^n\|_{L^2} &\leq C h^2 \| \tilde{\eta}^n\|_{H^2},\label{cf-etaprojection}\\
\| c^n-\Pi_h c^n \|_{L^2}+	h \| \nabla (c^n-\Pi_h c^n) \|_{L^2} &\leq C h^3 \| c^n\|_{H^3}. \label{cf-cprojection}
\end{align}

Next, we will give the main theorem in this paper.
\begin{theorem}\label{cf-therror}
	Let $(\u^i,p^i, \eta^i,c^i)$ and $(\u^i_h,p^i_h, \eta^i_h,c^i_h)$ are the solutions of the continuous model (\ref{cfmodel}) and the finite element discrete scheme (\ref{cf-scheme1})-(\ref{cf-scheme6}), under the regularity assumption (\ref{cf-regularity}) and the time step conditon $\tau \leq Ch^2$, there exists some $C >0$ independent of $\tau$ and $h$ such that
	\begin{align}\label{cf-errorresult}
		\underset{0\leq i \leq N}{\max} (\| \u^i-\u^i_h\|_{L^2}^2 + \| \eta^i-\eta^i_h\|_{L^2}^2 +\| c^i-c^i_h\|_{L^2}^2 )\leq C (\tau^2+h^4).
	\end{align}
\end{theorem}
%%%%%%%%%%%%%%%%%%%%%%%%%%%%%%%%%%%%%%%%%%%%%%%%%%%%%%%%%%%%%%%%%%%%%%%%%%%%%%%%%%%%%%%%%%%%%%%%%%%%%%%%%%%%%%%%%%%%%%%%%%%%%%%%%%%%%%%%%%%%%%%%%%%%%%%%%%%%%%%%%%%%%%%%%%%%%%%%%%%%%%%%%%%%%%%%%%%%%%%%%%%%%%%%%%%%%%%%%%%%%%%%%%%%
\section{Error estimate for the pressure-correction projection finite element methods}\label{erroranalysis}
In this section, we will derive the convergent analysis for the fully discrete projection method. Firstly, we need to derive the error equation,
for $0\leq n \leq N-1$, the exact solutions $(\u^{n+1},p^{n+1}, \tilde{\eta}^{n+1},c^{n+1})$ to the system (\ref{recfmodel}) satisfies
\begin{align}
	D_{\tau} \u^{n+1}- \mu \Delta \u^{n+1} + \u^n\cdot \nabla \u^{n+1} + \nabla p^n &= (\tilde{\eta}^n +m_0) \nabla \phi^{n+1}+ R_{\u}^{n+1},\label{cf-exact1}\\
	D_{\tau} \tilde{\eta}^{n+1} -a_\eta \Delta \tilde{\eta}^{n+1} + \u^n\cdot \nabla \tilde{\eta}^{n+1}-\beta \div( (\tilde{\eta}^{n+1}+m_0)\nabla c^{n+1}  ) &=R^{n+1}_c,\label{cf-exact2}\\
	D_{\tau}c^{n+1}-a_c \Delta c^{n+1} + \u^n\cdot \nabla c^{n+1} + \gamma (\tilde{\eta}^n+m_0) c^{n+1} &= R^{n+1}_c,\label{cf-exact3}
\end{align}
where the trunction error $R^{n+1}_{\u}, R^{n+1}_{\tilde{\eta}},R^{n+1}_c$ are given by
\begin{align}
	R^{n+1}_{\u} &= D_{\tau} \u^{n+1} -\u_t(t^{n+1}) + \u^n \cdot \nabla \u^{n+1}- \u^{n+1} \cdot \nabla \u^{n+1} + \nabla p^n-\nabla p^{n+1} \notag\\
	&+ (\tilde{\eta}^{n+1}+m_0)\nabla \phi^{n+1} - (\tilde{\eta}^{n}+m_0)\nabla \phi^{n+1},\notag\\
	R^{n+1}_{\eta }  &= D_{\tau} \tilde{\eta}^{n+1} -\tilde{\eta}_t(t^{n+1})+ \u^n \cdot \nabla \tilde{\eta}^{n+1} -  \u^{n+1} \cdot \nabla \tilde{\eta}^{n+1},\notag \\
	R^{n+1}_c &=  D_{\tau} c^{n+1} -c_t(t^{n+1})+ \u^n \cdot \nabla c^{n+1}-\u^{n+1} \cdot \nabla c^{n+1}\notag\\
	&+ \gamma(  \tilde{\eta}^n+m_0)c^{n+1}-\gamma(  \tilde{\eta}^{n+1}+m_0)c^{n+1}.\notag
\end{align}
Next, we give the error estimate  for $R^{n+1}_{\u}, R^{n+1}_{\tilde{\eta}},R^{n+1}_c$ by the following lemma.
\begin{lemma}\label{cf-trunction-error}
Under the regularity assumption (\ref{cf-regularity}), there holds
\begin{align}\label{cf-12}
	\tau \sum_{n=0}^{N-1} ( \| R^{n+1}_{\u}\|_{L^2}^2 +\| R^{n+1}_{\tilde{\eta}}\|_{L^2}^2 + \| R^{n+1}_c\|_{L^2}^2  ) \leq C \tau ^2. 
\end{align}
\begin{proof}
	In terms of Taylor formula
	\begin{align}
	D_{\tau} \u^{n+1} - \u_t(t^{n+1}) = - \frac{1}{\tau } \int^{t^{n+1}}_{t^{n}} (t-t^n)\u_{tt}dt,
 		\end{align}
%we have
%\begin{align}
%<R^{n+1}_\u,\v>&= - <\frac{1}{\tau}  \int^{t^{n+1}}_{t^n} (t-t^n) \u_{tt}dt,\v   >- b( \int^{t^{n+1}}_{t^n}  \u_t dt, \u^{n+1},\v  ) \notag\\
%&+(\int^{t^{n+1}}_{t^n} p_t dt, \div \v ) + ( \int^{t^{n+1}}_{t^n} \tilde{\eta}_tdt \nabla \phi^{n+1},\v  ), \quad \forall \v \, \in \V,
%\end{align} 

From the regularity assumptions (\ref{cf-regularity}), we can have 
\begin{align}\label{cf-1}
	\| R^{n+1}_{\u} \|_{L^2} 
	\leq& \frac{C}{\tau} \int^{t^{n+1}}_{t^n} (t-t^n) \| \u_{tt}\|_{L^2}dt + C \| \nabla \u^{n+1} \|_{L^\infty} \int^{t^{n+1}}_{t^n} \| \u_t\|_{L^2}dt\notag\\
	&+C \int^{t^{n+1}}_{t^n} \|p_t \|_{L^2}dt + C \| \phi^{n+1} \|_{W^{1,\infty}} \int^{t^{n+1}}_{t^n} \| \tilde{\eta}_t\|_{L^2}dt  \\
	\leq& C \tau^{\frac{1}{2}} (     \int^{t^{n+1}}_{t^n} \| \u_{tt}\| _{L^2}^2dt  +  \int^{t^{n+1}}_{t^n}\| \u_t\|_{L^2}^2dt + \int^{t^{n+1}}_{t^n} \| p_t \|_{L^2}^2dt + \int^{t^{n+1}}_{t^n}\| \tilde{\eta}_t\|_{L^2}^2 dt)^{\frac{1}{2}}.\notag
\end{align}

For $R^{n+1}_{\tilde{\eta}}$, we can get
\begin{align}\label{cf-2}
	\|  R^{n+1}_{\tilde{\eta}}\|_{L^2} 
	\leq 
	&
	\frac{1}{\tau} \int^{t^{n+1}}_{t^n} (t-t^n)\|\tilde{\eta}_{tt}\|_{L^2}dt
	+ \int^{t^{n+1}}_{t^n} \| \u_{t}\|_{L^2}dt \| \nabla \tilde{\eta}^{n+1}\|_{L^{\infty}}\notag\\
	\leq & C \tau^{\frac{1}{2}} (  \int^{t^{n+1}}_{t^n}\| \tilde{\eta}_{tt}\|_{L^2}^2 dt
	+ \int^{t^{n+1}}_{t^n}\| \u_t\|_{L^2}^2 dt )^{\frac{1}{2}}.
\end{align}

Similarly, 
\begin{align}\label{cf-3}
		\|  R^{n+1}_{c}\| _{L^2}\leq &\frac{1}{\tau} \int^{t^{n+1}}_{t^n} (t-t^n)\|c_{tt}\|_{L^2}dt+ \int^{t^{n+1}}_{t^n} \|\u_{t}\|_{L^2} dt \|  \nabla c^{n+1}\|_{L^{\infty}} +C \int^{t^{n+1}}_{t^n}\| \tilde{\eta}_t\| _{L^2}dt \| c^{n+1}\|_{L^{\infty}} \notag\\
		\leq & C \tau^{\frac{1}{2}} (  \int^{t^{n+1}}_{t^n}\| c_{tt}\|_{L^2}^2 dt+ \int^{t^{n+1}}_{t^n}\| \u_t\|_{L^2}^2 dt+ \int^{t^{n+1}}_{t^n}\| \tilde{\eta}_t\|_{L^2}^2 dt )^{\frac{1}{2}}.
\end{align}

Thus,  from (\ref{cf-1}), (\ref{cf-2}), (\ref{cf-3}), we complete the proof of Lemma  \ref{cf-trunction-error}.
\end{proof}
\end{lemma}
Secondly, we will give the error equation, testing (\ref{cf-exact1}) - (\ref{cf-exact3}) by $ (\v_h,q_h,w_h,\psi_h) \in  \V_h \times M_h \times X_{0h} \times X_h$ and subtracting them from (\ref{cf-scheme1}), (\ref{cf-scheme4}), (\ref{cf-scheme6}), we get
\begin{align}
	&(\frac{\tilde{e}^{n+1}_\u - e^n_\u}{\tau},\v_h) + \mu (\nabla \tilde{e}^{n+1}_\u, \nabla \v_h) - (\nabla \cdot \v_h,e^n_p)\notag\\
	=&-(\frac{\theta^{n+1}_\u - \theta^n_\u}{\tau},\v_h)+(R^{n+1}_\u,\v_h)-B(e_{\u}^n,\u^{n+1},\v_h)-B(\theta^n_\u,\u^{n+1},\v_h)\label{cf-errorequation1}\\
	&-B(\u^n_h,\tilde{e}^{n+1}_\u,\v_h) -B(\u^n_h,\theta^{n+1}_\u,\v_h)+((\tilde{e}^n_\eta+ \tilde{\theta}^n_\eta) \nabla \phi^{n+1} ,\v_h), \quad \forall \v_h \in \V_h.\notag\\%%%%%%%%%%%%%%%%%%%%%%%%%%%%
&	(\frac{\tilde{e}^{n+1}_\eta- \tilde{e}^n_\eta}{\tau},w_h)+a_\eta ( \nabla \tilde{e}^{n+1}_\eta,\nabla w_h)\notag\\
	=&-(\frac{\theta^{n+1}_\eta - \theta^n_\eta}{\tau},w_h)+ (R^{n+1}_\eta,w_h)
	-b(e^n_\u,\tilde{\eta}^{n+1},w_h)-b(\theta^n_\u,\tilde{\eta}^{n+1},w_h) \label{cf-errorequation2}\\
	&-b(\u^n_h,\tilde{e}^{n+1}_\eta,w_h) 
	-b(\u^n_h, \theta_\eta^{n+1},w_h)
	+\beta(\tilde{e}^n_\eta \nabla c^{n+1},\nabla w_h)+ \beta(\theta_\eta^n \nabla c^{n+1},\nabla w_h) \notag\\
	&+\beta(  (\tilde{\eta}^n+m_0) \nabla (e^{n+1}_c +\theta^{n+1}_c) , \nabla w_h) + \beta(  (\tilde{e}^n_\eta+\theta^n_\eta) \nabla (e^{n+1}_c +\theta^{n+1}_c), \nabla w_h)\notag\\
		&+\beta (  (\tilde{\eta}^{n+1}-\tilde{\eta}^n) \nabla c^{n+1}, \nabla w_h  ),\quad  \forall w_h \in X_{0h}.\notag\\%%%%%%%%%%%%%%%%%%%%%%%%%%%%%%%
	&	(\frac{e^{n+1}_c- e^n_c}{\tau},w_h)+a_c ( \nabla e^{n+1}_c,\nabla \psi_h)\notag\\
	=&-(\frac{\theta^{n+1}_c - \theta^n_c}{\tau},\psi_h) + (R^{n+1}_c,\psi_h) -b(e^n_\u,c^{n+1},\psi_h)-b(\theta^n_\u, c^{n+1},\psi_h)\label{errorequation3}\\
	&-b(\u^n_h,e^{n+1}_c,\psi_h)-b(\u^n_h,\theta^{n+1}_c,\psi_h)-\gamma(\tilde{e}^n_\eta c^{n+1},\psi_h)-\gamma(\theta^n_\eta c^{n+1},\psi_h)\notag\\
	&-\gamma( (\tilde{\eta}^n+m_0) (e^{n+1}_c+\theta^{n+1}_c),\psi_h )
	-\gamma( (  \tilde{e}^n_\eta+\theta^n_\eta) (e^{n+1}_c+\theta^{n+1}_c),\psi_h), \quad \forall \psi_h\in X_h\notag.
 \end{align}

From (\ref{cf-scheme3}), we have 
\begin{align}
	\u^{n+1}_h -R^{n+1}_h \u^{n+1} +R^{n+1}_h \u^{n+1} -\tilde{\u}^{n+1}_h+ 2 \tau (  \nabla (p^{n+1}_h- K_hp^{n+1})   -    \nabla (p^{n}_h- K_hp^{n} )) = - 2 \tau ( \nabla K_h ( p^{n+1}-p^{n})).
\end{align}

Furthermore
\begin{align}\label{cf-errorequation4}
	e^{n+1}_\u-\tilde{e}^{n+1}_\u + 2\tau ( \nabla e^{n+1}_p- \nabla e^n_p) =  2 \tau ( \nabla K_h ( p^{n+1}-p^{n})).
\end{align}

\begin{lemma}\label{cf-therror2}
Under the regularity assumption (\ref{cf-regularity}) and the time step conditon $\tau \leq Ch^2$, there exists some $C_1 >0$ independent of $\tau$ and $h$ such that
	\begin{align}\label{cf-errorresult2}
	\| e^{i}_\u \|^2_{L^2}+\| \tilde{e}^{i}_\eta \|^{2}_{L^2}+\| e^{i}_c \|^{2}_{L^2} \leq C_1 (\tau^2+h^4), \quad \forall \, 0\leq i \leq N.
		\end{align}
\end{lemma}
\begin{proof}
It is straightforward to observe that the proof of Theorem \ref{cf-therror} follows directly from Lemma \ref{cf-therror2} together with the projection error estimates (\ref{cf-uprojection})–(\ref{cf-cprojection}). Therefore, it remains to establish Lemma \ref{cf-therror2}, which will be proved by means of mathematical induction.
So we firstly prove (\ref{cf-errorresult2}) is valid for $i=1$, taking $n=0$ in the error equations (\ref{cf-errorequation1}) - (\ref{cf-errorequation4}), we have 
\begin{align}
	&(\frac{\tilde{e}^{1}_\u - e^0_\u}{\tau},\v_h) + \mu (\nabla \tilde{e}^{1}_\u, \nabla \v_h) - (\nabla \cdot \v_h,e^0_p)\notag\\
	=&-(\frac{\theta^{1}_\u - \theta^0_\u}{\tau},\v_h)+(R^{1}_\u,\v_h)-B(e_{\u}^0,\u^{1},\v_h)-B(\theta^0_\u,\u^{1},\v_h)\label{cf-uerror00}\\
	&-B(\u^0_h,\tilde{e}^{1}_\u,\v_h) -B(\u^0_h,\theta^{1}_\u,\v_h)+((\tilde{e}^0_\eta+ \theta^0_\eta) \nabla \phi^{1} ,\v_h), \quad \forall \v_h \in \V_h .\notag\\%%%%%%%%%%%%%%%%%%%%%%%%%%%%
	&(\frac{\tilde{e}^{1}_\eta- \tilde{e}^0_\eta}{\tau},w_h)+a_\eta ( \nabla \tilde{e}^{1}_\eta,\nabla w_h)\notag\\
	=&-	(\frac{\theta^{1}_\eta - \theta^0_\eta}{\tau},w_h)+ (R^{1}_\eta,w_h)-b(e^0_\u,\tilde{\eta}^{1},w_h)-b(\theta^0_\u,\tilde{\eta}^{1},w_h) \label{cf-etaerror00}\\
	&-b(\u^0_h,\tilde{e}^{1}_\eta,w_h) -b(\u^0_h, \theta_\eta^{1},w_h)+\beta(\tilde{e}^0_\eta \nabla c^{1},\nabla w_h)+ \beta(\theta_\eta^0 \nabla c^{1},\nabla w_h) \notag\\
	&+\beta(  (\tilde{\eta}^0+m_0) \nabla (e^{1}_c +\theta^{1}_c) ,\nabla  w_h) + \beta(  (\tilde{e}^0_\eta+\theta^0_\eta) \nabla (e^{1}_c +\theta^{1}_c), \nabla w_h)\notag\\
	&+\beta (  (\tilde{\eta}^{1}-\tilde{\eta}^0) \nabla c^{1}, \nabla w_h  )\quad  \forall w_h \in X_{0h}.\notag\\%%%%%%%%%%%%%%%%%%%%%%%%%%%%%%%
	&(\frac{e^{1}_c- e^0_c}{\tau},w_h)+a_c ( \nabla e^{1}_c,\nabla \psi_h)\notag\\
	=&-	(\frac{\theta^{1}_c - \theta^0_c}{\tau},\psi_h) + (R^{1}_c,\psi_h) -b(e^0_\u,c^{1},\psi_h)-b(\theta^0_\u,c^{1},\psi_h)\label{cf-cerror00}\\
	&-b(\u^0_h,e^{1}_c,\psi_h)-b(\u^0_h,\theta^{1}_c,\psi_h)-\gamma(\tilde{e}^0_\eta c^{1},\psi_h)-\gamma(\theta^0_\eta c^{1},\psi_h)\notag\\
	&-\gamma( (\tilde{\eta}^0+m_0) (e^{1}_c+\theta^{1}_c),\psi_h )
	-\gamma( (  \tilde{e}^0_\eta+\theta^0_\eta) (e^{1}_c+\theta^{1}_c),\psi_h), \quad \forall \psi_h\in X_h\notag.
\end{align}

%\subsection{Convergent analysis}
%%%%%%%%%%%%%%%%%%%%%%%%%%%%%%%%%%%%%%%%%%%%%%%%%%%%%%%%%%%%
%%%%%%%%%%%%%%%%%%%%%
%\subsection{Estimate of $ e^{1}_\u$}
%%%%%%%%%%%%%%%%%%%
%%%%%%%%%%%%%%%%%%%%%%%%%%%%%%%%%%%%%%%%%%%%%%%%%%%%%%%%%%%%
Setting $ \v_h=\tau \tilde{e}^1_\u$ in (\ref{cf-uerror00}) and noticing that $ e^0_\u=e^0_p=0$, we have
\begin{align}
	\| \tilde{e}^{1}_\u\|_{L^2}^2+\mu \tau \| \nabla \tilde{e}^{1}_\u\|_{L^2}^2
	=&-\tau(\frac{\theta^{1}_\u - \theta^0_\u}{\tau},\v_h)+ \tau  (R^{1}_\u,\tilde{e}^{1}_\u)- \tau  B(\theta_{\u}^0,\u^{1},\tilde{e}^{1}_\u) \label{cf-1uerror}\\
	%- \tau B(e^0_\u,\tilde{\u}^{1},\tilde{e}^{1}_\u)\\
	&- \tau B(\u^0_h,\tilde{e}^{1}_\u,\tilde{e}^{1}_\u) - \tau B(\u^0_h,\theta^{1}_\u,\tilde{e}^{1}_\u)
	+\tau( \theta^0_\eta \nabla \phi^{1} ,\tilde{e}^{1}_\u).\notag
\end{align}

Using the regularity assumption (\ref{cf-regularity}), projection error (\ref{cf-uinterpolation}),  H\"{o}lder inequality and Young inequality, we have
\begin{align}
	-\tau(\frac{\theta^{1}_\u - \theta^0_\u}{\tau},\v_h)
	\leq & \nu_1 \| \tilde{e}^1_\u\|^2_{L^2}+C (\| \theta^1_\u\|^2_{L^2}+\| \theta^0_\u\|^2_{L^2})\\
	\leq &\nu_1 \| \tilde{e}^1_\u\|^2_{L^2}+C h^4.\notag\\
	\tau(\theta^0_\eta \nabla \phi^{1} ,\tilde{e}^{1}_\u) \leq &C \| \theta^0_\eta\|_{L^2}\| \nabla \phi^{1}\|_{L^3} \|\tilde{e}^{1}_\u\|_{L^6}\\
	\leq &C h^4  + \nu_4  \tau \|  \nabla \tilde{e}^1_\u\|^2_{L^2}.\notag
\end{align}

Similarly, by using (\ref{cf-12}), one has
\begin{align}
 \tau    (R^{1}_\u,\tilde{e}^{1}_\u)
 \leq & \nu_4 \tau  \| \nabla  \tilde{e}^1_\u\|^2_{L^2}+ C \tau  \|R^{1}_\u\|^2_{L^2}  \\
 \leq &  \nu_4 \tau \| \nabla  \tilde{e}^1_\u\|^2_{L^2}+C\tau^2. \notag
 \end{align}
 
Using the regularity assumption (\ref{cf-regularity}), (\ref{cf-inequality}), (\ref{cf-uinterpolation}), we can derive 
\begin{align}
	 \tau  B(\theta_{\u}^0,\u^{1},\tilde{e}^{1}_\u)
	 \leq &C\tau \| \theta^0_\u\|_{L^2} \| \nabla \u^1\|_{L^3} \| \tilde{e}^1_\u\|_{L^6}  \notag\\
	 \leq & C   \tau h^4+ \nu_4 \tau \| \nabla \tilde{e}^1_\u\|^2_{L^2}. \notag
\end{align}

In terms of (\ref{biobdf-2}), (\ref{biobdf-17}) and inverse inequalities (\ref{biobdf-16}), we can get 
\begin{align}
	&|-\tau B(\u^0_h,\tilde{e}^{1}_\u,\tilde{e}^{1}_\u) -\tau B(\u^0_h,\theta^{1}_\u,\tilde{e}^{1}_\u)|\notag\\
	\leq& \tau B(\e^0_\u+\theta^0_\u,\theta^{1}_\u,\tilde{e}^{1}_\u)+ \tau B(\u^0,\tilde{e}^{1}_\u,\theta^{1}_\u)\notag\\
	\leq & \tau \| \theta^0_\u\|_{L^3} \| \nabla  \theta^{1}_\u \|_{L^2}\| \tilde{e}^{1}_\u\|_{L^6} 
	+ \tau \| \u^0\|_{L^{\infty}} \| \nabla \tilde{e}^{1}_\u\|_{L^2} \| \theta^{1}_\u\|_{L^2} \\
	\leq & C \tau h^{-\frac{3}{2}} \| \theta^0_\u\|_{L^2}  \| \theta^{1}_\u \|_{L^2} \| \nabla  \tilde{e}^{1}_\u\|_{L^2}  + C  \tau h^4 + \nu_4 \tau \| \nabla  \tilde{e}^{1}_\u\|_{L^2}\notag\\
	\leq & \nu_4 \tau  \| \nabla  \tilde{e}^{1}_\u\|^2_{L^2}+C  \tau h^4+C \tau h^5.\notag
\end{align}

Substituting the above inequalities into (\ref{cf-1uerror}), for sufficiently small $\nu_1, \nu_4$, and sufficiently small $ \tau , h$ such that $C\tau \leq C_1$ and $C \tau h \leq C_1$, we can derive
\begin{align}\label{cf-1uerrorresult}
		\| \tilde{e}^{1}_\u\|_{L^2}^2+\mu \tau \| \nabla \tilde{e}^{1}_\u\|_{L^2}^2 \leq C_1 (\tau^2+h^4).
\end{align}

Let $n=0$ in (\ref{cf-errorequation4}), we have 
\begin{align}\label{cf-1errorequation4}
	e^{1}_\u-\tilde{e}^{1}_\u + 2\tau  \nabla e^{1}_p =  2 \tau ( \nabla K_h ( p^{1}-p^{0})).
\end{align}

Testing (\ref{cf-1errorequation4}) by $2 \e^1_\u$ and noticing that (\ref{divergence-free}), we can derive 
\begin{align}
	\| \e^1_\u \|^2_{L^2} - \| \tilde{e}^1_\u\|^2_{L^2} + \| \tilde{e}^1_\u - \tilde{e}^1_\u\|^2_{L^2}=0.
\end{align}

Associating with (\ref{cf-1uerrorresult}) yield
\begin{align}\label{cf-uerrorreulst}
	\| \e^1_\u \|^2_{L^2}+\mu \tau \| \nabla \tilde{e}^{1}_\u\|_{L^2}^2 \leq C_1 (\tau^2+h^4).
\end{align}
%%%%%%%%%%%%%%%%%%%%
%%%%%%%%%%%%%%%%%%%%%%%%%%%%%%%%%%%%%%%%%%%%%%%%%%%%%%%%%%%%
%\subsection{Estimate of $e^{1}_\eta$}
%%%%%%%%%%%%%%%%%%%%
%%%%%%%%%%%%%%%%%%%%%%%%%%%%%%%%%%%%%%%%%%%%%%%%%%%%%%%%%%%%
Setting $w_h=\tau \tilde{e}^1_\eta$ in (\ref{cf-etaerror00}) and noticing that $ e^0_\eta=e^0_\u=0$, we have
\begin{align}\label{cf-1etaerrorequation}
	%%%%%%%%%%%%%%%%%%%%%%%
	\| \tilde{e}^{1}_\eta\|_{L^2}^2+a_\eta \tau \| \nabla \tilde{e}^{1}_\eta\|_{L^2}^2
	=&-\tau(\frac{\theta^{1}_\eta- \theta^0_\eta}{\tau},\tilde{e}^1_\eta)
	+\tau(R^{1}_\eta,\tilde{e}^{1}_\eta)
	%-\tau b(\tilde{e}^0_\u,\tilde{\eta}^{1},\tilde{e}^{1}_\eta)
	-\tau b(\theta^0_\u,\tilde{\eta}^{1},\tilde{e}^{1}_\eta) \\
	&-\tau b(\u^0_h,\tilde{e}^{1}_\eta,\tilde{e}^{1}_\eta) 
	-\tau b(\u^0_h, \theta_\eta^{1},\tilde{e}^{1}_\eta)
	%+\tau \beta(\tilde{e}^0_\eta \nabla c^{1},\tilde{e}^{1}_\eta)
	+\tau \beta(\theta_\eta^0 \nabla c^{1},\nabla\tilde{e}^{1}_\eta) \notag\\
	&+  \tau\beta(  (\tilde{\eta}^0+m_0) \nabla (e^{1}_c +\theta^{1}_c) , \nabla\tilde{e}^{1}_\eta) 
	+  \tau \beta(  \theta^0_\eta \nabla (e^{1}_c +\theta^{1}_c),  \nabla \tilde{e}^{1}_\eta)\notag\\
	&+2 \tau \beta (  (\tilde{\eta}^{1}-\tilde{\eta}^0) \nabla c^{1},  \nabla \tilde{e}^1_\eta ).\notag
\end{align}

Using the regularity assumption (\ref{cf-regularity}), (\ref{cf-etainterpolation}),  H\"{o}lder inequality and Young inequality, we have
\begin{align}
	-\tau(\frac{\theta^{1}_\eta- \theta^0_\eta}{\tau},\tilde{e}^1_\eta) 
	\leq&\nu_2 \| \tilde{e}^1_\eta\|^2_{L^2}+ C (\| \theta^1_\eta\|^2_{L^2}+\| \theta^0_\eta\|^2_{L^2})\notag\\
	\leq &\nu_2 \| \tilde{e}^1_\eta\|^2_{L^2}+C h^4.\notag
\end{align}
 
Similarly, by using (\ref{cf-12}), one has
\begin{align}
 \tau(R^{1}_\eta,\tilde{e}^{1}_\eta)
	\leq & \nu_2 \| \tilde{e}^1_\eta\|^2_{L^2}+ C \tau  \|R^{1}_\eta\|^2_{L^2} \notag\\
	\leq &\nu_2 \| \tilde{e}^1_\eta\|^2_{L^2}+C\tau^2. \notag
\end{align}

Using the regularity assumption (\ref{cf-regularity}), (\ref{cf-inequality}),  (\ref{cf-etainterpolation}), we can derive 
\begin{align}
	\tau b(\theta^0_\u,\tilde{\eta}^{1},\tilde{e}^{1}_\eta)
	\leq &C\tau \| \theta^0_\u\|_{L^2} \| \nabla \tilde{\eta}^1\|_{L^3} \|\| \tilde{e}^1_\eta\|_{L^6} \notag\\
	\leq & C  \tau h^4+ \nu_5 \tau \| \nabla \tilde{e}^1_\eta\|^2_{L^2}. \notag
\end{align}

In terms of (\ref{biobdf-2}), (\ref{biobdf-17}) and inverse inequalities (\ref{biobdf-21}), we can get 
\begin{align}
	-\tau b(\u^0_h,\tilde{e}^{1}_\eta,\tilde{e}^{1}_\eta) - \tau b(\u^0_h,\theta^{1}_\eta,\tilde{e}^{1}_\eta)
	\leq& |\tau b(\theta^0_\u,\theta^{1}_\eta,\tilde{e}^{1}_\eta)+ \tau b(\u^0,\tilde{e}^{1}_\eta, \theta^{1}_\eta)|\\
	\leq & \tau \|\theta^0_\u\|_{L^3} \| \nabla  \theta^{1}_\eta \|_{L^2}\| \tilde{e}^{1}_\eta\|_{L^6} 
	+ \tau \| \u^0\|_{L^{\infty}} \| \nabla \tilde{e}^{1}_\eta\|_{L^2} \| \theta^{1}_\eta\|_{L^2} \notag\\
	\leq & C \tau h^{-\frac{3}{2}} \|\theta^0_\u\|_{L^2}  \| \theta^{1}_\eta \|_{L^2} \| \nabla  \tilde{e}^{1}_\eta\|_{L^2}  + C \tau h^4 + \nu_5 \tau \| \nabla  \tilde{e}^{1}_\eta\|_{L^2}\notag\\
	\leq & \nu_5\tau \| \nabla  \tilde{e}^{1}_\eta\|^2_{L^2}+C \tau h^4+C\tau h^5.\notag
\end{align}

By using the regularity assumption (\ref{cf-regularity}), (\ref{cf-inequality}), H\"{o}lder inequality and Young inequality, we have
\begin{align}
	\tau \beta(\theta_\eta^0 \nabla c^{1}, \nabla \tilde{e}^{1}_\eta)\leq
	&C\tau \| \theta^0_\eta\|_{L^2} \| \nabla c^{1}\|_{L^\infty} \|  \nabla \tilde{e}^{1}_\eta\|_{L^2}\\
	\leq & C \tau h^4+\nu_5 \tau \| \nabla  \tilde{e}^{1}_\eta\|^2_{L^2}.\notag
\end{align}

Applying (\ref{cf-cprojection}), H\"{o}lder inequality and Young inequality, we can derive
\begin{align}
 \tau\beta(  (\tilde{\eta}^0+m_0) \nabla (e^{1}_c +\theta^{1}_c) , \nabla \tilde{e}^{1}_\eta) \leq
	& 2 \tau \| \tilde{\eta}^0+m_0\|_{L^{\infty}} (\| \nabla e^{1}_c\|_{L^2} + \| \nabla \theta^{1}_c\|_{L^2} ) \| \nabla \tilde{e}^{1}_\eta\|_{L^2}\\
	\leq &\nu_5 \tau \| \nabla  \tilde{e}^{1}_\eta\|^2_{L^2}+ C \tau \| \nabla e^{1}_c\|^2_{L^2} + C \tau h^4.\nonumber
\end{align}

Utilizing  the regularity assumption (\ref{cf-regularity}), (\ref{cf-cprojection}), inverse inequality (\ref{biobdf-21}), for sufficiently small $h$, we have 
\begin{align}
 \tau \beta(  \theta^0_\eta \nabla (e^{1}_c +\theta^{1}_c), \nabla  \tilde{e}^{1}_\eta)\leq
	&C \tau \| \theta^0_\eta\|_{L^\infty} (\| \nabla \theta^{1}_c\|_{L^2} + \| \nabla e^{1}_c\|_{L^2})\| \nabla \tilde{e}^{1}_\eta\|_{L^2}\\
	\leq & C\tau h
	(\| \nabla \theta^{1}_c\|_{L^2} + \| \nabla e^{1}_c\|_{L^2})
	\| \nabla \tilde{e}^{1}_\eta\|_{L^2}\notag\\
	\leq &C\tau h \| \nabla \theta^{1}_c\|_{L^2}  \| \nabla \tilde{e}^{1}_\eta\|_{L^2}+ C\tau h \| \nabla e^{1}_c\|_{L^2}  \| \nabla \tilde{e}^{1}_\eta\|_{L^2}\notag\\
	\leq &C\tau  \|  \theta^{1}_c\|_{L^2}  \| \nabla \tilde{e}^{1}_\eta\|_{L^2}+ C\tau  \| \nabla e^{1}_c\|_{L^2}  \| \nabla \tilde{e}^{1}_\eta\|_{L^2}\notag\\
	\leq & C \tau h^4  +\nu_5 \tau \|  \nabla \tilde{e}^{1}_\eta\|^2_{L^2}+C\tau \| \nabla e^{1}_c\|^2_{L^2} .\notag
\end{align}
%
%
%
%
%
%
%
%we have 
%\begin{align}
% \tau \beta(  \theta^0_\eta \nabla (e^{1}_c +\theta^{1}_c), \nabla  \tilde{e}^{1}_\eta)\leq
%	&C \tau \| \theta^0_\eta\|_{L^4} (\| \nabla \theta^{1}_c\|_{L^2} + \| \nabla e^{1}_c\|_{L^2})\| \tilde{e}^{1}_\eta\|_{L^4}\\
%	\leq & C\tau h^{-\frac{1}{2}}  \| \theta^0_\eta\|_{L^2} h^{-1}\| \theta^{1}_c\|_{L^2}\| \nabla \tilde{e}^{1}_\eta\|_{L^2} + C \tau h^{-\frac{1}{2}} \| \theta^0_\eta\|_{L^2} h^{-1}\|e^{1}_c\|_{L^2} \|\nabla \tilde{e}^{1}_\eta\|_{L^2}\notag\\
%	\leq & C \tau h  \| \theta^0_\eta\|^2_{L^2} +\nu_5 \tau\|\nabla \tilde{e}^{1}_\eta\|^2_{L^2}+ C\tau h^{-3} \| \theta^0_\eta\|^2_{L^2}\| e^{1}_c\|^2_{L^2}.\notag\\
%	\leq & C \tau h^5 +\nu_5 \tau\| \nabla \tilde{e}^{1}_\eta\|^2_{L^2}+C\tau h \| e^{1}_c\|^2_{L^2}.\notag
%\end{align}

For the last term, we have
\begin{align}
2 \tau \beta (  (\tilde{\eta}^{1}-\tilde{\eta}^0) \nabla c^{1},  \nabla \tilde{e}^1_\eta )
	\leq &C \tau \| \tilde{\eta}^{1}-\tilde{\eta}^0\|_{L^2} \| \nabla c^{1}\|_{L^\infty} \| \nabla \tilde{e}^{1}_\eta \|_{L^2} \\
\leq & \nu_5 \tau\| \nabla \tilde{e}^{1}_\eta\|^2_{L^2}+ C\tau^2 \int_{t_{0}}^{t_{1}} \| \tilde{\eta}_t\|^2_{L^2}ds\notag\\
\leq & \nu_5 \tau\| \nabla \tilde{e}^{1}_\eta\|^2_{L^2}+ C\tau^2.\notag
\end{align}

Substituting the above estimates into (\ref{cf-1etaerrorequation}), for sufficiently small $\nu_2,\nu_5$, and sufficiently small $ \tau , h$ such that $C\tau \leq C_1$ and $C \tau h \leq C_1$, we have
\begin{align}\label{cf-etaerrorresult}
		\| \tilde{e}^{1}_\eta\|_{L^2}^2+a_\eta \tau \| \nabla \tilde{e}^{1}_\eta\|_{L^2}^2\leq C_1 (\tau^2+ h^4)+C\tau \| \nabla e^{1}_c\|^2_{L^2}.
\end{align}

 %%%%%%%%%%%%%%%%%%%%
 %%%%%%%%%%%%%%%%%%%%%%%%%%%%%%%%%%%%%%%%%%%%%%%%%%%%%%%%%%%%
%\subsection{Estimate of $e^{1}_c$}
%%%%%%%%%%%%%%%%%%%%
%%%%%%%%%%%%%%%%%%%%%%%%%%%%%%%%%%%%%%%%%%%%%%%%%%%%%%%%%%%%
Setting $\psi_h=\tau e^1_c$ in (\ref{cf-cerror00}) and noticing that $ e^0_c=e^0_\u=\tilde{e}^0_\eta=0$, we have 
\begin{align}\label{cf-1cerrorequation}
	%%%%%%%%%%%%%%%%%%%%
	\| e^{1}_c\|_{L^2}^2+a_c \tau \| \nabla e^{1}_c\|_{L^2}^2
	= &- \tau(\frac{\theta^{1}_c - \theta^0_c}{\tau},\psi_h)
	+ \tau (R^{1}_c,e^{1}_c)
	% - \tau b(e^0_\u,c^{1},e^{1}_c)
	-\tau b(\theta^0_\u,c^{1},e^{1}_c)\\
	&
	-\tau b(\u^0_h,e^{1}_c,e^{1}_c)
	- \tau b(\u^0_h,\theta^{1}_c,e^{1}_c)
	%-\tau \gamma(\tilde{e}^0_\eta c^{1},e^{1}_c)
	-\tau\gamma(\theta^0_\eta c^{1},e^{1}_c)\notag\\
	&- \tau\gamma( (\tilde{\eta}^0+m_0) (e^{1}_c+\theta^{1}_c),e^{1}_c )
	- \tau\gamma( \theta^0_\eta (e^{1}_c+\theta^{1}_c),e^{1}_c ).\notag
\end{align}

Using the regularity assumption (\ref{cf-regularity}), (\ref{cf-cinterpolation}),  H\"{o}lder inequality and Young inequality, we have
\begin{align}
-\tau(\frac{\theta^{1}_c - \theta^0_c}{\tau},\psi_h) 
	\leq&\nu_3 \| \tilde{e}^1_c\|^2_{L^2}+C (\| \theta^1_c\|^2_{L^2}+\| \theta^0_c\|^2_{L^2})\notag\\
	\leq &\nu_3 \| \tilde{e}^1_c\|^2_{L^2}+C h^4.\notag
\end{align}

Similarly, by using (\ref{cf-12}), one has
\begin{align}
	\tau (R^{1}_c,e^{1}_c)
	\leq & \nu_3 \| \tilde{e}^1_c\|^2_{L^2} + C \tau  \|R^{1}_c\|^2_{L^2}\notag\\
	\leq &\nu_3 \| \tilde{e}^1_c\|^2_{L^2}  +C\tau^2. \notag
\end{align}

Using the regularity assumption (\ref{cf-regularity}), (\ref{cf-inequality}), (\ref{cf-cinterpolation}) we can derive 
\begin{align}
	\tau b(\theta^0_\u,c^{1},e^{1}_c)
	\leq &C\tau \| \theta^0_\u\|_{L^2} \| \nabla c^1\|_{L^3} \|\| e^1_c\|_{L^6} \notag \notag\\
	\leq & C  \tau  h^4+ \nu_6 \tau \| \nabla e^1_c\|^2_{L^2}. \notag
\end{align}

In terms of (\ref{biobdf-2}), (\ref{biobdf-17}) and inverse inequalities (\ref{biobdf-16}), we can get 
\begin{align}
	-\tau b(\u^0_h,e^{1}_c,e^{1}_c)
	- \tau b(\u^0_h,\theta^{1}_c,e^{1}_c)\leq
	&|-\tau b(\u^0_h,e^{1}_c,e^{1}_c) - \tau b(\u^0_h,\theta^{1}_c,e^{1}_c)|\notag\\
	\leq& | \tau b(\theta^0_\u,\theta^{1}_c,e^{1}_c)+  \tau b(\u^0,e^{1}_c, \theta^{1}_c)|\notag\\
	\leq &\tau \| \theta^0_\u\|_{L^3} \| \nabla  \theta^{1}_c \|_{L^2}\| e^{1}_c\|_{L^6} 
	+ \tau \| \u^0\|_{L^{\infty}} \| \nabla e^{1}_c\|_{L^2} \| \theta^{1}_c\|_{L^2} \\
	\leq & C \tau h^{-\frac{3}{2}}\| \theta^0_\u\|_{L^2}  \| \theta^{1}_c \|_{L^2} \| \nabla  e^{1}_c\|_{L^2}  + C \tau h^4 + \nu_6 \tau  \| \nabla  e^{1}_c\|_{L^2}\notag\\
	\leq & \nu_6 \tau \| \nabla  e^{1}_c\|^2_{L^2}+C \tau h^4+C \tau h^5.\notag
\end{align}

By using H\"{o}lder inequality and Young inequality, we have
\begin{align}
	\tau\gamma(\tilde{\theta}^0_\eta c^{1},e^{1}_c)\leq&
	C\tau   \| \theta^0_\eta\|_{L^2}\| c^{1}\|_{L^{3}} \| e^{1}_c\|_{L^6} \\
	\leq &C \tau h^4 +\nu_6 \tau \| \nabla e^{1}_c\|_{L^2}^2.\notag\\
- \tau\gamma( (\tilde{\eta}^0+m_0) (e^{1}_c+\theta^{1}_c),e^{1}_c ) \leq&
	|- \tau\gamma( (\tilde{\eta}^0+m_0) (e^{1}_c+\theta^{1}_c),e^{1}_c )|\\
	\leq & C\tau \| \tilde{\eta}^0+m_0\|_{L^{3}}  \| e^{1}_c+\theta^{1}_c\|_{L^2} \|e^{1}_c\|_{L^6}. \notag\\
	\leq & C \tau h^4 + C \tau \| e^{1}_c\|^2_{L^2}+\nu_6 \tau\| \nabla  e^{1}_c\|^2_{L^2}.\notag
\end{align}

Utilizing  the regularity assumption (\ref{cf-regularity}), inverse inequality (\ref{biobdf-21}), we have 
\begin{align}
- \tau\gamma( \theta^0_\eta (e^{1}_c+\theta^{1}_c),e^{1}_c ) \leq&
	|- \tau\gamma(  \theta^0_\eta (e^{1}_c+\theta^{1}_c),e^{1}_c )|\\
	\leq &  C \tau \| \theta^0_\eta\|_{L^4}  \| e^{1}_c+\theta^{1}_c\|_{L^2} \| e^{1}_c\|_{L^4}\notag\\
	\leq& C \tau h^{-\frac{1}{2}} \| \theta^0_\eta\|_{L^2} \| e^{1}_c\|_{L^2} \| \nabla e^{1}_c\|_{L^2} 
	+ C \tau h^{-\frac{1}{2}} \| \theta^0_\eta\|_{L^2} \| \theta^{1}_c\|_{L^2} \| \nabla e^{1}_c\|_{L^2}  \notag\\
	\leq &\nu_6\tau \| \nabla  \tilde{e}^{1}_c\|_{L^2} +C\tau h^{-1}\| \theta^0_\eta\|_{L^2}^2 \|e^{1}_c\|^2_{L^2} + C\tau h^3 \|\theta^0_\eta\|_{L^2}^2.\notag\\
	\leq &\nu_6\tau \| \nabla  \tilde{e}^{1}_c\|_{L^2}+C \tau h^3 \|e^{1}_c\|^2_{L^2} +C \tau h^7. \notag
\end{align}

Substituting the above estimates into (\ref{cf-1cerrorequation}), for sufficiently small $\nu_3, \nu_6,\tau,h$, we can derive
\begin{align}
	\| e^{1}_c\|_{L^2}^2+a_c \tau \| \nabla e^{1}_c\|_{L^2}^2\leq C_1 (\tau^2+h^4) +C \tau h^3 \|e^{1}_c\|^2_{L^2} +C \tau  \|e^{1}_c\|^2_{L^2}.
\end{align}

Combining with  (\ref{cf-etaerrorresult}), for sufficiently small $ \tau, h $ such that $C \tau h^3$, $C\tau$ sufficiently small, we can derive

\begin{align}\label{cf-etacerror}
	\| e^{1}_c\|_{L^2}^2+\| \tilde{e}^{1}_\eta\|_{L^2}^2+a_\eta \tau \| \nabla \tilde{e}^{1}_\eta\|_{L^2}^2 +a_c \tau \| \nabla e^{1}_c\|_{L^2}^2  	 \leq C_1(\tau^2+h^4).
\end{align}

According to (\ref{cf-uerrorreulst}) and (\ref{cf-etacerror}), we have proved  (\ref{cf-errorresult2}) is valid for  $i=1$.
By assuming that (\ref{cf-errorresult2}) is valid for  $i=n$, i.e.
	\begin{align}\label{cf-infuctionerror}
	\| e^{n}_\u \|^2_{L^2}+\| \tilde{e}^{n}_\eta \|^{2}_{L^2}+\| e^{n}_c \|^{2}_{L^2} \leq C(\tau^2+h^4).
\end{align}
We will prove  (\ref{cf-errorresult2}) is valid for $i=n+1$.
%%%%%%%%%%%%%%%%%%%%
%%%%%%%%%%%%%%%%%%%%%%%%%%%%%%%%%%%%%%%%%%%%%%%%%%%%%%%%%%%%
%\subsection{Estimate of $ e^{n+1}_\u$}
%%%%%%%%%%%%%%%%%%%%
%%%%%%%%%%%%%%%%%%%%%%%%%%%%%%%%%%%%%%%%%%%%%%%%%%%%%%%%%%%%
Setting $\v_h=2\tau \tilde{e}_\u^{n+1}$ in (\ref{cf-errorequation1}) and applying $a(a-b)=\frac{1}{2} \big(a^2-b^2+(a-b)^2 \big) $, one has 
\begin{align}\label{cf-uerror}
&\| \tilde{e}^{n+1}_\u\|_{L^2}^2-\| e^{n}_\u\|_{L^2}^2+\| \tilde{e}^{n+1}_\u-e^{n}_\u\|_{L^2}^2 +2 \mu \tau \| \nabla \tilde{e}^{n+1}_\u\|_{L^2}^2+2\tau(\nabla e^n_p,\tilde{e}^{n+1}_\u)\notag\\
=&-2 \tau (\frac{\theta^{n+1}_\u - \theta^n_\u}{\tau},\tilde{e}^{n+1}_\u)+2 \tau (R^{n+1}_\u,\tilde{e}^{n+1}_\u)-2 \tau B(e^n_\u,\u^{n+1},\tilde{e}^{n+1}_\u)-2 \tau  B(\theta_{\u}^n,\u^{n+1},\tilde{e}^{n+1}_\u)\\
&-2 \tau B(\u^n_h,\tilde{e}^{n+1}_\u,\tilde{e}^{n+1}_\u) -2 \tau B(\u^n_h,\theta^{n+1}_\u,\tilde{e}^{n+1}_\u)
+2\tau((\tilde{e}^n_\eta+ \theta^n_\eta) \nabla \phi^{n+1} ,\tilde{e}^{n+1}_\u)	.\notag
\end{align}

By using projection error (\ref{cf-uprojection}), H\"{o}lder inequality and Young inequality, one has
\begin{align}
|2 \tau (\frac{\theta^{n+1}_\u - \theta^n_\u}{\tau},\tilde{e}^{n+1}_\u) |\leq&\epsilon_1  \tau \| \tilde{e}^{n+1}_\u\|_{L^2}^2 +C \tau \| \frac{\theta^{n+1}_\u - \theta^n_\u}{\tau} \|^2_{L^2}\notag\\
\leq & \epsilon_1  \tau \| \tilde{e}^{n+1}_\u\|_{L^2}^2 + C \tau h^4 \|  \frac{\u^{n+1}-\u^n}{\tau} \|_{H^2}^2 \notag\\
\leq &  \epsilon_1 \tau \| \nabla  \tilde{e}^{n+1}_\u\|^2_{L^2}+C\tau h^4 \|  \u_t(\xi_1)\|_{H^2}^2 \quad \xi_1\in(t^n,t^{n+1})\\
\leq & \epsilon_1 \tau \| \nabla  \tilde{e}^{n+1}_\u\|^2_{L^2}+C\tau h^4.\notag\\
|2 \tau  (R^{n+1}_\u,\tilde{e}^{n+1}_\u)|\leq & \tau \| R^{n+1}_\u\|_{L^2}^2+ \epsilon_1 \tau \| \nabla  \tilde{e}^{n+1}_\u\|_{L^2}.
\end{align}

Using the regularity assumption (\ref{cf-regularity}), (\ref{cf-inequality}), H\"{o}lder inequality and Young inequality, we have
\begin{align}
&	|2\tau((\tilde{e}^n_\eta+ \theta^n_\eta) \nabla \phi^{n+1} ,\tilde{e}^{n+1}_\u)|\notag\\
	\leq& C\tau \|\tilde{e}^n_\eta\|^2_{L^2} + \epsilon_1 \tau \| \nabla  \tilde{e}^{n+1}_\u\|_{L^2}+C \tau h^4.\\
&|-2 \tau  B(\theta_{\u}^n,\u^{n+1},\tilde{e}^{n+1}_\u)-2 \tau B(e^n_\u,\u^{n+1},\tilde{e}^{n+1}_\u)| \notag \\
\leq & 2 \tau \| \theta^n_\u\|_{L^2} \| \nabla \u^{n+1}\|_{L^3} \|\tilde{e}^{n+1}_\u \|_{L^6}+2 \tau \| e^n_\u \|_{L^2}  \| \nabla \u^{n+1}\|_{L^3} \|\tilde{e}^{n+1}_\u \|_{L^6} \\
\leq &  \epsilon_1  \tau \| \nabla \tilde{e}^{n+1}_\u \|_{L^2}^2 + C\tau  \| e^n_\u\|_{L^2}^2 +C\tau h^4.\notag
\end{align}

In terms of (\ref{biobdf-2}), (\ref{biobdf-17}), (\ref{cf-uprojection}) and inverse inequalities (\ref{biobdf-16}), we can get 
\begin{align}
	&|-2 \tau B(\u^n_h,\tilde{e}^{n+1}_\u,\tilde{e}^{n+1}_\u) -2 \tau B(\u^n_h,\theta^{n+1}_\u,\tilde{e}^{n+1}_\u)|\notag\\
	\leq& 2 \tau B(\e^n_\u+\theta^n_\u,\theta^{n+1}_\u,\tilde{e}^{n+1}_\u)+ 2 \tau B(\u^n,\tilde{e}^{n+1}_\u,\theta^{n+1}_\u)\notag\\
	\leq & 2\tau \| \e^n_\u+\theta^n_\u\|_{L^3} \| \nabla  \theta^{n+1}_\u \|_{L^2}\| \tilde{e}^{n+1}_\u\|_{L^6} 
	+ 2\tau \| \u^n\|_{L^{\infty}} \| \nabla \tilde{e}^{n+1}_\u\|_{L^2} \| \theta^{n+1}_\u\|_{L^2} \\
	\leq & C \tau \| \e^n_\u+\theta^n_\u\|_{H^1} h^{-1} \| \theta^{n+1}_\u \|_{L^2} \| \nabla  \tilde{e}^{n+1}_\u\|_{L^2}  + C \tau h^4 + \epsilon_1 \tau \| \nabla  \tilde{e}^{n+1}_\u\|_{L^2}\notag\\
	\leq & C \tau \| \e^n_\u+\theta^n_\u\|_{L^2} h^{-2} \| \theta^{n+1}_\u \|_{L^2} \| \nabla  \tilde{e}^{n+1}_\u\|_{L^2}  + C \tau h^4 + \epsilon_1 \tau \| \nabla  \tilde{e}^{n+1}_\u\|_{L^2}\notag\\
	\leq & \epsilon_1 \tau \| \nabla  \tilde{e}^{n+1}_\u\|^2_{L^2}+ C\tau \| \e^n_\u\|_{L^2}^2+C \tau h^4.\notag
\end{align}

Substituting the above estimates into (\ref{cf-uerror}), choosing $\epsilon_1$ sufficiently small, we have 
\begin{align}\label{cf-uerrorresult}
	&\| \tilde{e}^{n+1}_\u\|_{L^2}^2-\| e^{n}_\u\|_{L^2}^2+\| \tilde{e}^{n+1}_\u-e^{n}_\u\|_{L^2}^2 +2 \mu \tau \| \nabla \tilde{e}^{n+1}_\u\|_{L^2}^2\notag\\
	\leq & \tau \| R^{n+1}_\u\|_{L^2}^2+C\tau h^4+C\tau ( \| e^n_\u\|_{L^2}^2+ \|\tilde{e}^n_\eta\|^2_{L^2}).
\end{align}

Testing (\ref{cf-errorequation4}) by $e^{n+1}_\u$ and $\frac{1}{2} ( e^{n+1}_\u+ \tilde{e}^{n+1}_\u)$, we have 
\begin{align}
	\frac{1}{2} (   \| e^{n+1}_\u \|_{L^2}^2- \| \tilde{e}^{n+1}_\u\|_{L^2}^2+ \|  e^{n+1}_\u - \tilde{e}^{n+1}_\u\|_{L^2}^2)=&0,\label{cf-7}\\
	\frac{1}{2} (\| \e^{n+1}_u\|_{L^2}^2- \| \tilde{e}^{n+1}_\u\|_{L^2}^2  ) + \tau ( \nabla e^{n+1}_p- \nabla e^n_p,\tilde{e}^{n+1}_\u ) =&   \tau ( \nabla K_h ( p^{n+1}-p^{n}),\tilde{e}^{n+1}_\u).\label{cf-8}
\end{align}
where we use (\ref{divergence-free}).
% $(\nabla q_h ,\u^{n+1}_h) = -(\nabla \cdot \u^{n+1}_h, q_h)=0$.

By using H\"{o}lder inequality and Young inequality, one has
\begin{align}\label{cf-4}
	\tau (\nabla K_h(p^{n+1}-p^n),\tilde{e}^{n+1}_\u   ) \leq& C \tau \| K_h (p^{n+1}-p^n)\|_{L^2} \| \nabla \tilde{e}^{n+1}_\u\|_{L^2}\notag\\
	\leq & \epsilon_1 \tau | \nabla \tilde{e}^{n+1}_\u\|_{L^2}^2+ C \tau^2 \int_{t^{n}}^{t^{n+1}} \| p_t\|^2_{L^2} dt.
\end{align}

From (\ref{cf-errorequation4}), one has 
\begin{align}
\tilde{e}^{n+1}_\u =	e^{n+1}_\u + 2\tau ( \nabla e^{n+1}_p- \nabla e^n_p)  - 2 \tau ( \nabla K_h ( p^{n+1}-p^{n})).
\end{align}

Thus 
\begin{align}\label{cf-5}
&\tau ( \nabla e^{n+1}_p+ \nabla e^n_p,\tilde{e}^{n+1}_\u ) \\
=& 2 \tau^2( \| \nabla e^{n+1}_p \|^2_{L^2}-\| \nabla e^{n}_p \|^2_{L^2} ) - 2 \tau^2 (\nabla (e^{n+1}_p +e^{n}_p ) ,  \nabla K_h ( p^{n+1}-p^{n}) ) \notag
\end{align}
where
\begin{align}\label{cf-6}
	&2 \tau^2 (\nabla (e^{n+1}_p +e^{n}_p ) ,  \nabla K_h ( p^{n+1}-p^{n}) ) \notag \\
	\leq & 2\tau^2 \|\nabla (e^{n+1}_p+e^{n}_p)\|_{L^2}\| \nabla K_h ( p^{n+1}-p^{n})\|_{L^2}\\
	\leq &C \tau^{\frac{3}{2}} \|\nabla (e^{n+1}_p+e^{n}_p)\|_{L^2} \tau^{\frac{1}{2}}\| \nabla K_h ( p^{n+1}-p^{n})\|_{L^2} \notag\\
	\leq & C \tau^3 ( \|\nabla (e^{n+1}_p\|^2_{L^2} +\|\nabla e^{n}_p\|^2_{L^2}  ) + C\tau^2 \int_{t^{n}}^{t^{n+1}} \|  \nabla p_t\|^2_{L^2} dt.\notag
\end{align}

Summing up (\ref{cf-uerrorresult}), (\ref{cf-7}), (\ref{cf-8}) and considering (\ref{cf-4}), (\ref{cf-5}), (\ref{cf-6}), for sufficiently small $\epsilon_1$, we have 
\begin{align}\label{cf-9}
	&\| e^{n+1}_\u\|_{L^2}^2-\| e^{n}_\u\|_{L^2}^2+ 2 \mu \tau \| \nabla \tilde{e}^{n+1}_\u\|_{L^2}^2 + \tau^2( \| \nabla e^{n+1}_p \|^2_{L^2}-\| \nabla e^{n}_p \|^2_{L^2} ) \notag\\
	\leq &C \tau^3 ( \|\nabla e^{n+1}_p\|^2_{L^2}+\|\nabla e^{n}_p\|^2_{L^2}  ) +C\tau ( \| e^n_\u\|_{L^2}^2+ \|\tilde{e}^n_\eta\|^2_{L^2})\notag\\
	&+C \tau^2 \int_{t^{n}}^{t^{n+1}} \| p_t\|^2_{L^2} dt + C\tau^2 \int_{t^{n}}^{t^{n+1}} \|  \nabla p_t\|^2_{L^2} dt+ \tau \| R^{n+1}_\u\|_{L^2}^2+C\tau h^4.
\end{align}
%%%%%%%%%%%%%%%%%%%%
%%%%%%%%%%%%%%%%%%%%%%%%%%%%%%%%%%%%%%%%%%%%%%%%%%%%%%%%%%%%
%\subsection{Estimate of $ \tilde{e}^{n+1}_\eta$}
%%%%%%%%%%%%%%%%%%%%
%%%%%%%%%%%%%%%%%%%%%%%%%%%%%%%%%%%%%%%%%%%%%%%%%%%%%%%%%%%%
Setting $w_h= 2 \tau  \tilde{e}^{n+1}_\eta $ in (\ref{cf-errorequation2}), we have 
\begin{align}\label{cf-etaerror}
	&\| \tilde{e}^{n+1}_\eta\|_{L^2}^2-\| e^{n}_\eta\|_{L^2}^2+\| \tilde{e}^{n+1}_\eta-e^{n}_\eta\|_{L^2}^2 +2a_\eta \tau \| \nabla \tilde{e}^{n+1}_\eta\|_{L^2}^2\notag\\
	=&-2\tau(\frac{\theta^{n+1}_\eta - \theta^n_\eta}{\tau},\tilde{e}^{n+1}_\eta)
	+2\tau(R^{n+1}_\eta,\tilde{e}^{n+1}_\eta)
	-2\tau b(e^n_\u,\tilde{\eta}^{n+1},\tilde{e}^{n+1}_\eta)
	-2\tau b(\theta^n_\u,\tilde{\eta}^{n+1},\tilde{e}^{n+1}_\eta) \\
	&-2\tau b(\u^n_h,\tilde{e}^{n+1}_\eta,\tilde{e}^{n+1}_\eta)
	-2\tau b(\u^n_h, \theta_\eta^{n+1},\tilde{e}^{n+1}_\eta)
	+2\tau \beta(\tilde{e}^n_\eta \nabla c^{n+1},\nabla \tilde{e}^{n+1}_\eta)
	+2\tau \beta(\theta_\eta^n \nabla c^{n+1},\nabla \tilde{e}^{n+1}_\eta) \notag\\
	&+ 2 \tau\beta(  (\tilde{\eta}^n+m_0) \nabla (e^{n+1}_c +\theta^{n+1}_c) ,\nabla  \tilde{e}^{n+1}_\eta) 
	+ 2 \tau \beta(  (\tilde{e}^n_\eta+\theta^n_\eta) \nabla (e^{n+1}_c +\theta^{n+1}_c),  \nabla \tilde{e}^{n+1}_\eta)\notag\\
		&+2\tau\beta (  (\tilde{\eta}^{n+1}-\tilde{\eta}^n) \nabla c^{n+1}, \nabla \tilde{e}^{n+1}_\eta ).\notag\\
	:=& \sum_{i=1}^{11}2\tau (Y_i ,\tilde{e}^{n+1}_\eta).\notag
\end{align}

By using $ab \leq \frac{1}{2}(a^2+b^2)$, one has
\begin{align}
	2\tau (Y_1 ,\tilde{e}^{n+1}_\eta) \leq&|2 \tau (\frac{\theta^{n+1}_\eta - \theta^n_\eta}{\tau},\tilde{e}^{n+1}_\eta) |\notag\\
	\leq&  \tau \| \tilde{e}^{n+1}_\eta\|_{L^2}^2 + \tau h^4 \|  \frac{\eta^{n+1}-\eta^n}{\tau} \|_{H^2}^2 \notag\\
	\leq &  \tau \|  \tilde{e}^{n+1}_\eta \|^2_{L^2}+C\tau h^4 \|  \eta_t(\xi_2)\|_{H^2}^2\quad \xi_2\in(t^n,t^{n+1})\\
	\leq & \tau \|  \tilde{e}^{n+1}_\eta\|^2_{L^2}+C\tau h^4.\notag\\
		2\tau (Y_2 ,\tilde{e}^{n+1}_\eta)\leq&|2 \tau  (R^{n+1}_\eta,\tilde{e}^{n+1}_\eta)\notag\\
		\leq & \tau \| R^{n+1}_\eta\|_{L^2}^2+  \tau \| \tilde{e}^{n+1}_\eta\|_{L^2}.
\end{align}

Using the regularity assumption (\ref{cf-regularity}), (\ref{cf-inequality}), H\"{o}lder inequality and Young inequality, we have
\begin{align}
		2\tau (Y_3+Y_4 ,\tilde{e}^{n+1}_\eta)\leq
	&|-2 \tau  b(e^n_\u,\tilde{\eta}^{n+1},\tilde{e}^{n+1}_\eta)-2 \tau b(\theta_{\u}^n,\tilde{\eta}^{n+1},\tilde{e}^{n+1}_\eta)| \notag \\
	\leq & 2 \tau \| e^n_\u\|_{L^2}   \| \nabla \tilde{\eta}^{n+1}\|_{L^3} \|\tilde{e}^{n+1}_\eta \|_{L^6}
	+2 \tau \| \theta^n_\u\| _{L^2} \| \nabla \tilde{\eta}^{n+1}\|_{L^3} \|\tilde{e}^{n+1}_\eta \|_{L^6} \\
	\leq &  \epsilon_2  \tau \| \nabla \tilde{e}^{n+1}_\eta \|_{L^2}^2 + C\tau  \| e^n_\u\|_{L^2}^2 +C\tau h^4.\notag
\end{align}

In terms of (\ref{biobdf-2}), (\ref{biobdf-17}) and inverse inequalities (\ref{biobdf-16}), we can get 
\begin{align}
		2\tau (Y_5+Y_6 ,\tilde{e}^{n+1}_\eta)\leq
	&|-2 \tau b(\u^n_h,\tilde{e}^{n+1}_\eta,\tilde{e}^{n+1}_\eta) -2 \tau b(\u^n_h,\theta^{n+1}_\eta,\tilde{e}^{n+1}_\eta)|\notag\\
	\leq& |2 \tau b(\e^n_\u+\theta^n_\u,\theta^{n+1}_\eta,\tilde{e}^{n+1}_\eta)+ 2 \tau b(\u^n,\tilde{e}^{n+1}_\eta, \theta^{n+1}_\eta)|\notag\\
	\leq & 2\tau \| \e^n_\u+\theta^n_\u\|_{L^3} \| \nabla  \theta^{n+1}_\eta \|_{L^2}\| \tilde{e}^{n+1}_\eta\|_{L^6} 
	+ 2\tau \| \u^n\|_{L^{\infty}} \| \nabla \tilde{e}^{n+1}_\eta\|_{L^2} \| \theta^{n+1}_\eta\|_{L^2} \\
	\leq & C \tau \| \e^n_\u+\theta^n_\u\|_{H^1} h^{-1} \| \theta^{n+1}_\eta \|_{L^2} \| \nabla  \tilde{e}^{n+1}_\eta\|_{L^2}  + C \tau h^4 + \epsilon_2 \tau\| \nabla  \tilde{e}^{n+1}_\eta\|_{L^2}\notag\\
	\leq & C \tau \| \e^n_\u+\theta^n_\u\|_{L^2} h^{-2} \| \theta^{n+1}_\eta \|_{L^2} \| \nabla  \tilde{e}^{n+1}_\eta\|_{L^2}  + C \tau h^4 + \epsilon_2\tau \| \nabla  \tilde{e}^{n+1}_\eta\|_{L^2}\notag\\
	\leq & \epsilon_2 \tau \| \nabla  \tilde{e}^{n+1}_\eta\|^2_{L^2}+ C\tau \| \e^n_\u\|_{L^2}^2+C \tau h^4.\notag
\end{align}

By using the regularity assumption (\ref{cf-regularity}), (\ref{cf-inequality}), H\"{o}lder inequality and Young inequality,  we have
\begin{align}
	2\tau (Y_7+Y_8 ,\tilde{e}^{n+1}_\eta)\leq
	&C\tau \| \tilde{e}^n_\eta+\tilde{\theta}^n_\eta\|_{L^2} \| \nabla c^{n+1}\|_{L^\infty} \| \nabla \tilde{e}^{n+1}_\eta\|_{L^2}\\
	\leq & C \tau  \| \tilde{e}^n_\eta\|_{L^2}^2+C \tau h^4+\epsilon_2 \tau \| \nabla  \tilde{e}^{n+1}_\eta\|_{L^2}.\notag
\end{align}

Applying (\ref{cf-cprojection}), H\"{o}lder inequality and Young inequality, we can derive
\begin{align}
		2\tau (Y_9 ,\tilde{e}^{n+1}_\eta)\leq
		& 2 \tau \| \tilde{\eta}^n+m_0\|_{L^{\infty}} (\| \nabla e^{n+1}_c\|_{L^2} + \| \nabla \theta^{n+1}_c\|_{L^2} ) \|\nabla \tilde{e}^{n+1}_\eta\|_{L^2}\\
		\leq &  \epsilon_2  \tau \| \tilde{e}^{n+1}_\eta\|^2_{L^2}+ C\tau \| \nabla e^{n+1}_c\|^2_{L^2} + C \tau h^4.\nonumber
\end{align}

Utilizing  the regularity assumption (\ref{cf-regularity}), (\ref{cf-cprojection}), inverse inequality (\ref{biobdf-21}), (\ref{cf-infuctionerror}), $\tau\leq Ch^2$, we have 
\begin{align}
	2\tau (Y_{10} ,\tilde{e}^{n+1}_\eta)\leq
	&C \tau \| \tilde{e}_\eta^n+\theta^n_\eta\|_{L^\infty} (\| \nabla \theta^{n+1}_c\|_{L^2} + \| \nabla e^{n+1}_c\|_{L^2})\| \nabla \tilde{e}^{n+1}_\eta\|_{L^2}\\
	\leq & C\tau h^{-1} (\tau+h^2)
	(\| \nabla \theta^{n+1}_c\|_{L^2} + \| \nabla e^{n+1}_c\|_{L^2})
	\| \nabla \tilde{e}^{n+1}_\eta\|_{L^2}\notag\\
	\leq & C \tau h^4  +\epsilon_2 \tau \|  \nabla \tilde{e}^{n+1}_\eta\|^2_{L^2}+C\tau \| \nabla e^{n+1}_c\|^2_{L^2} .\notag
	\end{align}
	
For the last term, we have 
	\begin{align}
	2\tau (Y_{11} ,\tilde{e}^{n+1}_\eta)\leq
	&2 \tau\beta (  (\tilde{\eta}^{n+1}-\tilde{\eta}^n) \nabla c^{n+1}, \nabla \tilde{e}^{n+1}_\eta ) \notag \\
	\leq &C \tau \| \tilde{\eta}^{n+1}-\tilde{\eta}^n\|_{L^2} \| \nabla c^{n+1}\|_{L^\infty} \| \nabla \tilde{e}^{n+1}_\eta \|_{L^2}\notag \\
	\leq &  \epsilon_2 \tau \| \nabla \tilde{e}^{n+1}_\eta \|_{L^2} + C\tau^2 \int_{t_{n}}^{t_{n+1}} \| \tilde{\eta}_t\|^2_{L^2}ds.
\end{align}
	Substituting the above estimates into (\ref{cf-etaerror}), we can derive
\begin{align}\label{cf-10}
	&\| \tilde{e}^{n+1}_\eta\|_{L^2}^2-\| e^{n}_\eta\|_{L^2}^2+\| \tilde{e}^{n+1}_\eta-e^{n}_\eta\|_{L^2}^2 +2a_\eta \tau \| \nabla \tilde{e}^{n+1}_\eta\|_{L^2}^2\\
	\leq & C\tau h^4 +C \tau\| \tilde{e}^{n+1}_\eta\|^2_{L^2}  + \tau \| R^{n+1}_\eta\|_{L^2}^2 + \epsilon_2  \tau \| \nabla \tilde{e}^{n+1}_\eta \|_{L^2}^2 + C\tau ( \| e^n_\u\|_{L^2}^2+ \| \tilde{e}^n_\eta\|_{L^2}^2) \notag\\
	&+\epsilon_3 \tau \| \nabla e^{n+1}_c\|^2_{L^2} 
	+C\tau^2 \int_{t_{n}}^{t_{n+1}} \| \tilde{\eta}_t\|^2_{L^2}+C\tau \| \nabla e^{n+1}_c\|^2_{L^2} .\notag
\end{align}
%%%%%%%%%%%%%%%%%%%%
%%%%%%%%%%%%%%%%%%%%%%%%%%%%%%%%%%%%%%%%%%%%%%%%%%%%%%%%%%%%
%\subsection{Estimate of $e^{n+1}_c $}
%%%%%%%%%%%%%%%%%%%%
%%%%%%%%%%%%%%%%%%%%%%%%%%%%%%%%%%%%%%%%%%%%%%%%%%%%%%%%%%%%
Setting $\psi_h=2\tau e^{n+1}_c$ in (\ref{errorequation3}), we have
\begin{align}\label{cf-cerror}
	&\| e^{n+1}_c\|_{L^2}^2-\| e^{n}_c\|_{L^2}^2+\| e^{n+1}_c-e^{n}_c\|_{L^2}^2 +2a_c \tau \| \nabla e^{n+1}_c\|_{L^2}^2\\
	= & -2 \tau(\frac{\theta^{n+1}_c - \theta^n_c}{\tau},e^{n+1}_c) 
	+2 \tau (R^{n+1}_c,e^{n+1}_c) -2 \tau b(e^n_\u,c^{n+1},e^{n+1}_c)
	-2 \tau b(\theta^n_\u,c^{n+1},e^{n+1}_c)\notag\\
	&
	-2 \tau b(\u^n_h,e^{n+1}_c,e^{n+1}_c)
	-2 \tau b(\u^n_h,\theta^{n+1}_c,e^{n+1}_c)
	-2 \tau \gamma(\tilde{e}^n_\eta c^{n+1},e^{n+1}_c)
	-2 \tau\gamma(\theta^n_\eta c^{n+1},e^{n+1}_c)\notag\\
	&-2 \tau\gamma( (\tilde{\eta}^n+m_0) (e^{n+1}_c+\theta^{n+1}_c),e^{n+1}_c )
	-2 \tau\gamma( (  \tilde{e}^n_\eta+\theta^n_\eta) (e^{n+1}_c+\theta^{n+1}_c),e^{n+1}_c )\notag\\
	:=& \sum_{i=1}^{10} 2\tau (M_i,e^{n+1}_c).\notag
	\end{align}
	
By using $ab \leq \frac{1}{2}(a^2+b^2)$, one has
\begin{align}
	2\tau (M_1 ,e^{n+1}_c) \leq&|2 \tau (\frac{\theta^{n+1}_c - \theta^n_c}{\tau},e^{n+1}_c) |\notag\\
	\leq&  \tau \| e^{n+1}_c\|_{L^2}^2 + \tau h^4 \|  \frac{c^{n+1}-c^n}{\tau} \|_{H^2}^2 \notag\\
	\leq &  \tau \|  e^{n+1}_c\|^2_{L^2}+C\tau h^4 \|  c_t(\xi_3)\|_{H^2}^2 \quad \xi_3\in(t^n,t^{n+1})\\
	\leq & \tau \|  e^{n+1}_c\|^2_{L^2}+C\tau h^4.\notag\\
	2\tau (M_1 ,e^{n+1}_c)\leq&|2 \tau  (R^{n+1}_c,e^{n+1}_c)\notag\\
	\leq & \tau \| R^{n+1}_c\|_{L^2}^2+  \tau \| e^{n+1}_c\|_{L^2}.
\end{align}

Using the regularity assumption (\ref{cf-regularity}), (\ref{cf-inequality}), H\"{o}lder inequality and Young inequality, we have
\begin{align}
	2\tau (M_3+Y_4 ,e^{n+1}_c)\leq
	&|-2 \tau  b(e^n_\u,c^{n+1},e^{n+1}_c)-2 \tau b(\theta_{\u}^n,c^{n+1},e^{n+1}_c)| \notag \\
	\leq & 2 \tau \| e^n_\u\|_{L^2}  \| \nabla c^{n+1}\|_{L^3} \|e^{n+1}_c \|_{L^3}
	+2 \tau \| \theta^n_\u\|_{L^2}  \| \nabla c^{n+1}\|_{L^3} \|e^{n+1}_c \|_{L^6} \\
	\leq &  \epsilon_3  \tau \| \nabla e^{n+1}_c \|_{L^2}^2 + C\tau  \| e^n_\u\|_{L^2}^2 +C\tau h^4.\notag
\end{align}

In terms of (\ref{biobdf-2}), (\ref{biobdf-17}) and inverse inequalities (\ref{biobdf-21}), we can get 
\begin{align}
	2\tau (M_5+M_6 ,e^{n+1}_c)\leq
	&|-2 \tau b(\u^n_h,e^{n+1}_c,e^{n+1}_c) -2 \tau b(\u^n_h,\theta^{n+1}_c,e^{n+1}_c)|\notag\\
	\leq& |2 \tau b(\e^n_\u+\theta^n_\u,\theta^{n+1}_c,e^{n+1}_c)+ 2 \tau b(\u^n,e^{n+1}_c, \theta^{n+1}_c)|\notag\\
	\leq & 2\tau \| \e^n_\u+\theta^n_\u\|_{L^3} \| \nabla  \theta^{n+1}_c \|_{L^2}\| e^{n+1}_c\|_{L^6} 
	+ 2\tau \| \u^n\|_{L^{\infty}} \| \nabla e^{n+1}_c\|_{L^2} \| \theta^{n+1}_c\|_{L^2} \\
	\leq & C \tau \| \e^n_\u+\theta^n_\u\|_{H^1} h^{-1} \| \theta^{n+1}_c \|_{L^2} \| \nabla  e^{n+1}_c\|_{L^2}  + C \tau h^4 + \epsilon_3 \tau \| \nabla  e^{n+1}_c\|_{L^2}\notag\\
	\leq & C \tau \| \e^n_\u+\theta^n_\u\|_{L^2} h^{-2} \| \theta^{n+1}_c \|_{L^2} \| \nabla  e^{n+1}_c\|_{L^2}  + C \tau h^4 + \epsilon_3 \tau \| \nabla  e^{n+1}_c\|_{L^2}\notag\\
	\leq & \epsilon_3 \tau \| \nabla  \tilde{e}^{n+1}_c\|^2_{L^2}+ C\tau \| \e^n_\u\|_{L^2}^2+C \tau h^4.\notag
\end{align}

By using H\"{o}lder inequality and Young inequality, we have
\begin{align}
	2\tau (M_7+M_8 ,e^{n+1}_c)\leq&
	C\tau ( \| \tilde{e}^n_\eta \|_{L^2} + \| \theta^n_\eta\|_{L^2}) \| c^{n+1}\|_{L^{3}} \| e^{n+1}_c\|_{L^6} \\
	\leq &C \tau h^4 +C\tau \| \tilde{e}^n_\eta \|_{L^2}^2+\epsilon_3  \tau \| \nabla e^{n+1}_c\|_{L^2}^2.\notag\\
	2\tau (M_9 ,e^{n+1}_c)\leq&
|-2 \tau\gamma( (\tilde{\eta}^n+m_0) (e^{n+1}_c+\theta^{n+1}_c),e^{n+1}_c )|\\
\leq & C\tau \| \tilde{\eta}^n+m_0\|_{L^{3}}  \| e^{n+1}_c+\theta^{n+1}_c\|_{L^2} \|e^{n+1}_c\|_{L^6}. \notag\\
\leq & C \tau h^4 + C \tau \| e^{n+1}_c\|^2_{L^2}+\epsilon_3  \tau \| \nabla  e^{n+1}_c\|^2_{L^2}.\notag
\end{align}

Utilizing  the regularity assumption (\ref{cf-regularity}), inverse inequality (\ref{biobdf-21}), we have 
\begin{align}
	2\tau (M_{10} ,e^{n+1}_c)\leq&
	|2 \tau\gamma( (  \tilde{e}^n_\eta+\theta^n_\eta) (e^{n+1}_c+\theta^{n+1}_c),e^{n+1}_c )|\\
	\leq &  C \tau \| \tilde{e}^n_\eta+\theta^n_\eta\|_{L^4}  \| e^{n+1}_c+\theta^{n+1}_c\|_{L^2} \| e^{n+1}_c\|_{L^4}\notag\\
	\leq& C \tau h^{-\frac{1}{2}} \| \tilde{e}^n_\eta+\theta^n_\eta\|_{L^2} \| e^{n+1}_c\|_{L^2} \| \nabla e^{n+1}_c\|_{L^2} 
	+ C \tau h^{-\frac{1}{2}} \| \tilde{e}^n_\eta+\theta^n_\eta\|_{L^2} \| \theta^{n+1}_c\|_{L^2} \| \nabla e^{n+1}_c\|_{L^2}  \notag\\
	\leq &\epsilon_3\tau \| \nabla  \tilde{e}^{n+1}_c\|^2_{L^2} +C\tau h^{-1}\| \tilde{e}^n_\eta+\theta^n_\eta\|_{L^2}^2 \|e^{n+1}_c\|^2_{L^2} + C\tau h^3 \|\tilde{e}^n_\eta+\theta^n_\eta\|_{L^2}^2.\notag
\end{align}

Substituting the above estimates into (\ref{cf-cerror}), we can derive
\begin{align}\label{cf-11}
		&\| e^{n+1}_c\|_{L^2}^2-\| e^{n}_c\|_{L^2}^2+\| e^{n+1}_c-e^{n}_c\|_{L^2}^2 +2a_c \tau \| \nabla e^{n+1}_c\|_{L^2}^2\\
		\leq& C \tau \| e^{n+1}_c\|_{L^2}^2+ \epsilon_3 \tau \| \nabla e^{n+1}_c\|_{L^2}^2+ C\tau h^4 + C\tau (\| e^n_\u\|^2_{L^2} + \| \tilde{e}^n_\eta\|^2_{L^2})\notag\\
		&+C\tau h^{-1}\| \tilde{e}^n_\eta+e^n_\eta\|_{L^2}^2 \|e^{n+1}_c\|^2_{L^2} + C\tau h^3 \|\tilde{e}^n_\eta+\theta^n_\eta\|_{L^2}^2.\notag
\end{align}

Thus, from (\ref{cf-9}), (\ref{cf-10}), (\ref{cf-11}), for sufficiently small $\epsilon_1,\epsilon_2,\epsilon_3$, we have
\begin{align}
	&\| e^{n+1}_\u\|_{L^2}^2-\| e^{n}_\u\|_{L^2}^2+ 2 \mu \tau \| \nabla \tilde{e}^{n+1}_\u\|_{L^2}^2 + \tau^2( \| \nabla e^{n+1}_p \|^2_{L^2}-\| \nabla e^{n}_p \|^2_{L^2} ) \\
	&+\| \tilde{e}^{n+1}_\eta\|_{L^2}^2-\| e^{n}_\eta\|_{L^2}^2+\| \tilde{e}^{n+1}_\eta-e^{n}_\eta\|_{L^2}^2 +2a_\eta \tau \| \nabla \tilde{e}^{n+1}_\eta\|_{L^2}^2\nonumber\\
	&+\| e^{n+1}_c\|_{L^2}^2-\| e^{n}_c\|_{L^2}^2+\| e^{n+1}_c-e^{n}_c\|_{L^2}^2 +2a_c \tau \| \nabla e^{n+1}_c\|_{L^2}^2\notag\\
	\leq &C \tau^3 ( \|\nabla (e^{n+1}_p\|^2_{L^2}+\|\nabla e^{n}_p\|^2_{L^2}  ) + C\tau(  \| \tilde{e}^{n+1}_\eta\|^2_{L^2} +\tau \| e^{n+1}_c\|_{L^2}^2+\| e^n_\u\|_{L^2}^2+ \|\tilde{e}^n_\eta\|^2_{L^2} )\nonumber \\
	&+C \tau^2 \int_{t^{n}}^{t^{n+1}} \| p_t\|^2_{L^2} dt + C\tau^2 \int_{t^{n}}^{t^{n+1}} \|  \nabla p_t\|^2_{L^2} dt+ C\tau ( \| R^{n+1}_\u\|_{L^2}^2+ \| R^{n+1}_\eta\|_{L^2}^2+ \| R^{n+1}_c\|_{L^2}^2)+C\tau h^4\nonumber\\
	& +C\tau h^{-1}\| \tilde{e}^n_\eta+\theta^n_\eta\|_{L^2}^2 \|e^{n+1}_c\|^2_{L^2} + C\tau h^3 \|\tilde{e}^n_\eta+\theta^n_\eta\|_{L^2}^2+ C\tau^2 \int_{t_{n}}^{t_{n+1}} \| \tilde{\eta}_t\|^2_{L^2}ds.\nonumber
\end{align}

Summing up from 0 to $N-1$,
in terms of mathematical induction method (\ref{cf-infuctionerror}), regularity assumptions (\ref{cf-regularity}), (\ref{cf-12}), then under the time size condition $\tau \leq Ch^2$ and for sufficiently small $h$ such that 
\begin{align}
	&\| e^{N}_\u\|_{L^2}^2+\| e^{N}_\eta\|_{L^2}^2+ \| e^{N}_c\|_{L^2}^2+2 \mu \tau \sum_{n=0}^{N-1}(\| \nabla \tilde{e}^{n+1}_\u\|_{L^2}^2 +\| \nabla \tilde{e}^{n+1}_\eta\|_{L^2}^2 + \| \nabla e^{n+1}_c\|_{L^2}^2)+\tau^2\| \nabla e^{N}_p \|^2_{L^2}  \\
	\leq & C(\tau^2 +h^4) +C \tau^3 ( \|\nabla e^{n+1}_p\|^2_{L^2}+\|\nabla e^{n}_p\|^2_{L^2}  ) + C\tau(  \| \tilde{e}^{n+1}_\eta\|^2_{L^2} +\tau \| e^{n+1}_c\|_{L^2}^2+\| e^n_\u\|_{L^2}^2+ \|\tilde{e}^n_\eta\|^2_{L^2} )\nonumber
	\end{align}
	
Application of Gronwall's inequality in Lemma \ref{biobdf-11} yields (\ref{cf-errorresult2}). The proof of Lemma \ref{cf-therror2} is completed.
	\end{proof}

\subsection{The proof of Theorem \ref{cf-therror}}
\begin{proof}
	According the projection error (\ref{cf-uprojection}), (\ref{cf-etaprojection}), (\ref{cf-cprojection}) and Lemma \ref{cf-therror2}, we can derive 
	\begin{align}\label{cf-uiresult}
		\| \u^i-\u^i_h\|_{L^2} =\| \u^i-R_h \u^i + R_h \u^i -\u^i_h\|_{L^2} \leq \| \u^i-R_h \u^i \|_{L^2} + \| e^i_\u\|_{L^2} \leq C(\tau^2+h^4).
	\end{align}
	
Similarly 

	\begin{align}
	\| \eta^i-\eta^i_h\|_{L^2} = &	\| (\tilde{\eta}^i+m_0)-(\tilde{\eta}^i_h+m_0)\|_{L^2}\label{cf-etairesult}\\
	 =&	\| \tilde{\eta}^i-\tilde{\eta}^i_h\|_{L^2}\notag\\
	 =&\| \eta^i-T_h \eta^i + T_h \eta^i -\eta^i_h\|_{L^2}\notag\\
	  \leq &\| \eta^i-T_h \eta^i \|_{L^2} + \| e^i_\eta\|_{L^2} \notag\\
	  \leq &C(\tau^2+h^4),\notag\\
	\| c^i-c^i_h\|_{L^2} =&\| c^i-\Pi_h c^i + \Pi_h c^i -c^i_h\|_{L^2}\notag\\
	 \leq &\| c^i-\Pi_h c^i \|_{L^2} + \| e^i_c\|_{L^2}\notag\\
	  \leq &C(\tau^2+h^4).\label{cf-ciresult}
\end{align}

Thus, we complete the proof of Theorem \ref{cf-therror}.
\end{proof}

Now the convergence analysis by the projection finite element method is done. Using the error estimates (\ref{cf-uiresult}), (\ref{cf-etairesult}), (\ref{cf-ciresult}) and the regularity assumption (\ref{cf-regularity}), we are now in a position to derive the uniform bound results for the fully discrete scheme, as detailed below.
\begin{theorem}
Under the regularity assumption (\ref{cf-regularity}), for sufficiently small $\tau$ and $h$, the following unconditional stability results satisfy
\begin{align}
	\| \u^i_h\|_{L^2}^2 + 	\| \eta^i_h\|_{L^2}^2+	\| c^i_h\|_{L^2}^2  \leq C.
\end{align}
for $1 \leq i \leq N$, which is difficult to derive by using the energy analysis method due to the nonlinear term $\beta \div(\eta \nabla c)$ in (\ref{cfmodel}).
\end{theorem}
%%%%%%%%%%%%%%%%%%%%%%%%%%%%%%%%%%%%%%%%%%%%%%%%%%%%%%%%%%%%%%%%%%%%%%%%%%%%%%%%%%%%%%%%%%%%%%%%%%%%%%%%%%%%%%%%%%%%%%%%%%%%%%%%%%%%%%%%%%%%%%%%%%%%%%%%%%%%%%%%%%%%%%%%%%%%%%%%%%%%%%%%%%%%%%%%%%%%%%%%%%%%%%%%%%%%%%%%%%%%%%%%%%%%

%%%%%%%%%%%%%%%%%%%%%%%%%%%%%%%%%%%%%%%%%%%%%%%%%%%%%%%%%%%%%%%%%%%%%%%%%%%%%%%%%%%%%%%%%%%%%%%%%%%%%%%%%%%%%%%%%%%%%%%%%%%%%%%%%%%%%%%%%%%%%%%%%%%%%%%%%%%%%%%%%%%%%%%%%%%%%%%%%%%%%%%%%%%%%%%%%%%%%%%%%%%%%%%%%%%%%%%%%%%%%%%%%%%%
\section{Numerical experiences}\label{numericalresult}
%%%%%%%%%%%%%%%%%%%%%%%%%%%%%%%%%%%%%%%%%%%%%%%%%%%%%%%%%%%%%%%%%%%%%%%%%%%%%%%%%%%%%%%%%%%%%%%%%%%%%%%%%%%%%%%%%%%%%%%%%%%%%%%%%%%%%%%%%%%%%%%%%%%%%%%%%%%%%%%%%%%%%%%%%%%%%%%%%%%%%%%%%%%%%%%%%%%%%%%%%%%%%%%%%%%%%%%%%%%%%%%%%%%%
\subsection{Convergence test}
In this section, we will present the numerical results to verify our theoretical analysis. All programs are implemented using the free finite element software FreeFem++ \cite{Hecht2012NewDI}. For the sake of simplicity, we solve the following coupled system with the artificial functions $g_1 ,g_2$ and \(\mathbf{f}\):
\begin{align}
	\eta_t - \alpha_\eta \Delta \eta + \beta \, \text{div} (\eta \nabla c) + u \cdot \nabla \eta = g_1, &  \quad  x \in \Omega, \, t>0,\\
	c_t - \alpha_c \Delta c + u \cdot \nabla c + \gamma \eta c = g_2, &\quad  x \in \Omega, \, t>0,\\
	u_t + (u \cdot \nabla) u - \nu \Delta u + \nabla p-\eta \nabla \phi = \textbf{f}, &\quad  x \in \Omega, \, t>0,\\
	\text{div} \, u = 0, & \quad  x \in \Omega, \, t>0.
\end{align}
in $\Omega \times [0, T]$, where $\Omega$ is the unit  square:
$$ \Omega = \{ (x,y)\in \mathbb{R}^2 : 0< x <1, 0<y<1   \}.$$

We set the final time $T=1$, $a_\eta=a_c=\mu=1$. To select the approximate functions $g$ and $\f$, we determine the exact solution  $(\sigma,\u,p)$ as follows
\begin{align}\left\{\begin{aligned}
		&\u_1(x,y,t)= sin(\pi x) cos (\pi y) sin (t),\\
		&\u_2(x,y,t)= -cos(\pi x) sin (\pi y) sin (t),\\
		& p(x,y,t)= cos(\pi x) cos(\pi y) sin(t),\\
		& \eta(x,y,t) = cos(2 \pi x ) cos(\pi y) sin(t),\\
		&c(x,y,t)= cos( \pi x ) cos(2\pi y) sin(t).
	\end{aligned}\right.\end{align}
Denote
\begin{align*}
	\|r - r_h\|_{L^2} &= \|r(t_N) - r_h^N\|_{L^2}, \\
	\|\mathbf{v} - \mathbf{v}_h\|_{L^2} &= \|\mathbf{v}(t_N) - \mathbf{v}_h^N\|_{L^2}.
\end{align*}

The numerical results are obtained by taking various grid sizes 
$ h = \frac{1}{4}, \frac{1}{8}, \dots $, 
where the meshes are generated from uniform triangular discretizations. 
By fixing the time step $\tau = 1/1000$ and refining the spatial grid size $h$, 
the spatial convergence rates are evaluated based on the numerical results 
reported in Tables~\ref{label6}--\ref{label8} and illustrated in Figure~\ref{tau1000}. 
Furthermore, Figures~\ref{velocity}--\ref{concentration} display 
the computed numerical solutions of the velocity field $\mathbf{u}$, 
the pressure $p$, the cell density $\eta$, and the chemical concentration $c$ 
at different time instants 
$t = 0,\, 0.2,\, 0.4,\, 0.6,\, 0.8,\, 1.0$. 
To examine the temporal convergence rate, we fix the spatial mesh size at 
$h = 1/200$ and vary the time step size 
$\tau = \frac{1}{4}, \frac{1}{8}, \dots$.
The convergence rates of the velocity, pressure, cell density, and concentration 
in the $L^2$-norm are summarized in Table~\ref{label2} and Figure \ref{h200},  
which show good agreement with the theoretical predictions. 
In addition, the relative errors are reported in Table~\ref{label4}. 

In conclusion, all numerical experiments demonstrate that the proposed algorithm 
achieves the expected convergence behavior and effectively validates 
its accuracy and robustness.

\begin{figure}
	\centering
	\includegraphics[width=0.7\linewidth]{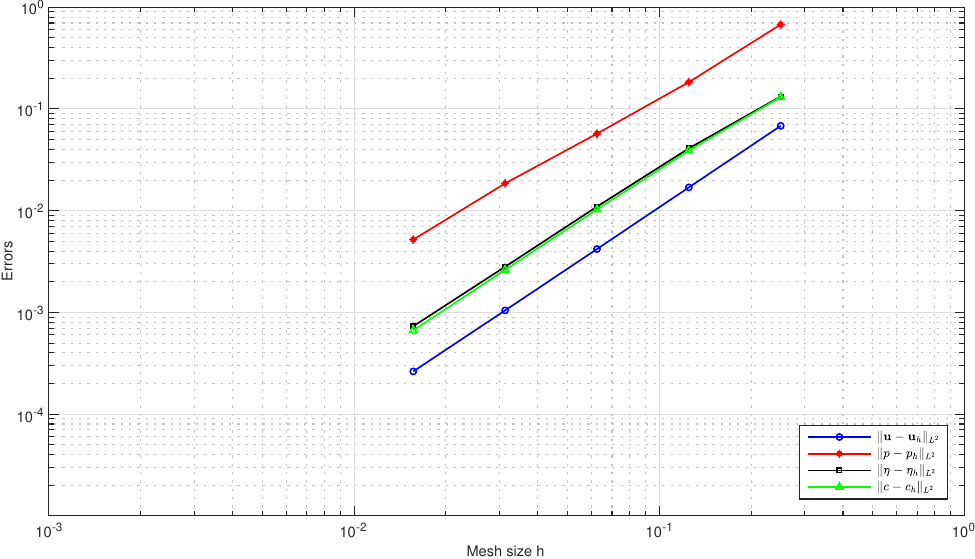}
	\caption{Convergence history of $(\u,p,\eta,c)$ for different $h$.}
	\label{tau1000}
\end{figure}

\begin{figure}
	\centering
	\includegraphics[width=0.7\linewidth]{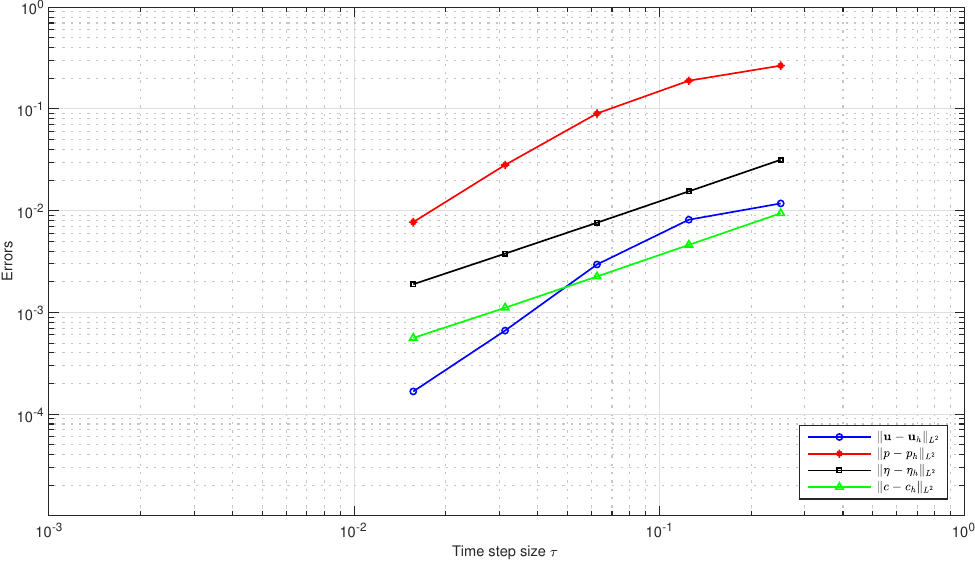}
	\caption{Convergence history of $(\u,p,\eta,c)$ for different $\tau$.}
	\label{h200}
\end{figure}

%% Table generated by Excel2LaTeX from sheet 'Sheet1'
%\begin{table}[htbp]  
%	\centering  
%	\caption{Stability results of the numerical solutions with $\tau=h^2$}   
% 	\setlength{\tabcolsep}{9mm}{ 
%		\rowcolors{1}{gray!20}{white} % 从第1行开始交替上色
%		\begin{tabular}{ccccc}       
%			%	\hline\hline   
%			$h$& $\| \textbf{u}\|_{L^2}$ & $\| p\|_{L^2}$   & $\| \eta\|_{L^2}$& $\| c\|_{L^2}$\\   
%			%	\hline
%			1/4   & 0.539696 & 0.426284 & 0.315041 & 0.316742 \\    1/8   & 0.580131 & 0.442062 & 0.387706 & 0.388616 \\    1/16  & 0.591262 & 0.426708 & 0.411871 & 0.412188 \\    1/32  & 0.594071 & 0.421613 & 0.418475 & 0.418561 \\    1/64  & 0.594775 & 0.420822 & 0.420167 & 0.420189 \\   
%			%	\hline\hline
%	\end{tabular}}
%	\label{label5}%
%\end{table}%

\begin{table}[htbp]  
	\centering  
	\caption{Numerical errors and convergence rates of $ \textbf{u},p,\eta,c$ in $L^2$-norm with $\tau=1/1000$}   
	\rowcolors{1}{gray!20}{white} % 从第1行开始交替上色
	\setlength{\tabcolsep}{2.6mm}{  
		\begin{tabular}{ccccccccc}         
			$h$& $\|\textbf{u}-\textbf{u}_h\|_{L^2}$ & Order  & $\| p-p_h\|_{L^2} $  & Order  & $\| \eta-\eta_h\|_{L^2} $ & Order& $\| c-c_h\|_{L^2} $& Order \\   
1/4   & 0.068218 &       & 0.676031 &       & 0.132643 &       & 0.130808 &  \\    1/8   & 0.016933 & 2.01  & 0.183462 & 1.88  & 0.040973 & 1.69  & 0.038975 & 1.75  \\    1/16  & 0.004196 & 2.01  & 0.057134 & 1.68  & 0.010961 & 1.90  & 0.010287 & 1.92  \\    1/32  & 0.001047 & 2.00  & 0.018544 & 1.62  & 0.00281 & 1.96  & 0.002619 & 1.97  \\    1/64  & 2.63E-04 & 1.99  & 5.19E-03 & 1.84  & 7.31E-04 & 1.94  & 0.000666 & 1.98  \end{tabular} }
	\label{label6}
\end{table}%

% Table generated by Excel2LaTeX from sheet 'Sheet1'\begin{table}[htbp]  \centering  \caption{Add caption}    \begin{tabular}{ccccccccc}          & uL2error & rate  & pL2   & rate  & etaL2 & rate  & cL2error & rate \\    1/4   & 0.068218 &       & 0.676031 &       & 0.132643 &       & 0.130808 &  \\    1/8   & 0.016933 & 2.01  & 0.183462 & 1.88  & 0.040973 & 1.69  & 0.038975 & 1.75  \\    1/16  & 0.004196 & 2.01  & 0.057134 & 1.68  & 0.010961 & 1.90  & 0.010287 & 1.92  \\    1/32  & 0.001047 & 2.00  & 0.018544 & 1.62  & 0.00281 & 1.96  & 0.002619 & 1.97  \\    1/64  & 2.63E-04 & 1.99  & 5.19E-03 & 1.84  & 7.31E-04 & 1.94  & 0.000666 & 1.98  \\    \end{tabular}%  \label{tab:addlabel}%\end{table}%

\begin{table}[htbp]  
	\centering  
	\caption{Numerical errors and convergence rates of $ \textbf{u},\eta,c$ in $H^1$-norm with $\tau=1/1000$}  
	\rowcolors{1}{gray!20}{white} % 从第1行开始交替上色
	\setlength{\tabcolsep}{5mm}{  
		\begin{tabular}{ccccccc}          
			$h$& $\|\textbf{u}-\textbf{u}_h\|_{H^1}$ & Order   & $\| \eta-\eta_h\|_{H^1} $ & Order& $\| c-c_h\|_{H^1} $& Order \\   
	 1/4   & 1.07199 &       & 1.55715 &       & 1.53617 &  \\    1/8   & 0.510077 & 1.07  & 0.837474 & 0.89  & 0.831965 & 0.88  \\    1/16  & 0.250376 & 1.03  & 0.427146 & 0.97  & 0.426257 & 0.96  \\    1/32  & 0.124114 & 1.01  & 0.214758 & 0.99  & 0.214623 & 0.99  \\    1/64  & 0.06177 & 1.01  & 0.107549 & 1.00  & 0.107519 & 1.00  \\ 
	\end{tabular} }
	\label{label7}
\end{table}%

% Table generated by Excel2LaTeX from sheet 'Sheet1'\begin{table}[htbp]  \centering  \caption{Add caption}    \begin{tabular}{ccccccc}          & uH1error & rate  & etaH1error & rate  & cH1error & rate \\    1/4   & 1.07199 &       & 1.55715 &       & 1.53617 &  \\    1/8   & 0.510077 & 1.07  & 0.837474 & 0.89  & 0.831965 & 0.88  \\    1/16  & 0.250376 & 1.03  & 0.427146 & 0.97  & 0.426257 & 0.96  \\    1/32  & 0.124114 & 1.01  & 0.214758 & \textbf{0.99 } & 0.214623 & \textbf{0.99 } \\    1/64  & 0.06177 & \textbf{1.01 } & 0.107549 & \textbf{1.00 } & 0.107519 & \textbf{1.00 } \\    \end{tabular}%  \label{tab:addlabel}%\end{table}%
\begin{table}[htbp]  
	\centering  
	\caption{Relative errors and convergence rates of $ \textbf{u},p,\eta,c$ with $\tau=1/1000$}   
	\rowcolors{1}{gray!20}{white} % 从第1行开始交替上色
	\setlength{\tabcolsep}{2.7mm}{   \begin{tabular}{ccccccccc}       
			$h$&$\frac{\| \u-\u_h\|_{L^2}}{\| \u\|_{L^2}}$   & Order &$\frac{\| p-p_h\|_{L^2}}{\| p\|_{L^2}}$ & Order &$\frac{\| \eta-\eta_h\|_{L^2}}{\| \eta\|_{L^2}}$ & Order&$\frac{\| c-c_h\|_{L^2}}{\| c\|_{L^2}}$  &Order\\    
1/4   & 0.127267 &       & 0.936792 &       & 0.421115 &       & 0.412925 &  \\    1/8   & 0.029192 & 2.12  & 0.579032 & 0.69  & 0.105682 & 1.99  & 0.100285 & 2.04  \\    1/16  & 0.007097 & 2.04  & 0.215922 & 1.42  & 0.026613 & 1.99  & 0.024957 & 2.01  \\    1/32  & 0.001762 & 2.01  & 0.073336 & 1.56  & 0.006715 & 1.99  & 0.006257 & 2.00  \\    1/64  & 0.000442 & 2.00  & 0.022099 & 1.73  & 0.00174 & 1.95  & 0.001584 & 1.98  
	\end{tabular}}
	\label{label8}
\end{table}%

% Table generated by Excel2LaTeX from sheet 'Sheet1'\begin{table}[htbp]  \centering  \caption{Add caption}    \begin{tabular}{ccccccccc}          & ure   & rate  & pre   & rate  & etare & rate  & cre   & rate \\    1/4   & 0.127267 &       & 0.936792 &       & 0.421115 &       & 0.412925 &  \\    1/8   & 0.029192 & 2.12  & 0.579032 & 0.69  & 0.105682 & 1.99  & 0.100285 & 2.04  \\    1/16  & 0.007097 & 2.04  & 0.215922 & 1.42  & 0.026613 & 1.99  & 0.024957 & 2.01  \\    1/32  & 0.001762 & 2.01  & 0.073336 & 1.56  & 0.006715 & 1.99  & 0.006257 & 2.00  \\    1/64  & 0.000442 & 2.00  & 0.022099 & 1.73  & 0.00174 & 1.95  & 0.001584 & 1.98  \\    \end{tabular}%  \label{tab:addlabel}%\end{table}%

%% Table generated by Excel2LaTeX from sheet 'Sheet1'
%\begin{table}[htbp]  
%	\centering  
%	\caption{Stability results of the numerical solutions with $\tau=h$}   
%	\setlength{\tabcolsep}{9mm}{  
%		\rowcolors{1}{gray!20}{white} % 从第1行开始交替上色
%		\begin{tabular}{ccccc}       
%			%	\hline\hline   
%			$\tau$& $\| \textbf{u}\|_{L^2}$ & $\| p\|_{L^2}$   & $\| \eta\|_{L^2}$& $\| c\|_0$\\   
%			%	\hline
%			1/4   & 0.540789 & 0.061565 & 0.315617 & 0.31677 \\   
%			1/8   & 0.579828 & 0.161348 & 0.387921 & 0.38844 \\    
%			1/16  & 0.591082 & 0.326345 & 0.411842 & 0.412053 \\    
%			1/32  & 0.593996 & 0.399496 & 0.418428 & 0.418485 \\    
%			1/64  & 0.594737 & 0.415733 & 0.420137 & 0.42015 \\    
%			1/128 & 0.594932 & 0.419485 & 0.420576 & 0.420579 \\    
%			%	\hline\hline
%	\end{tabular}}
%	\label{label1}%
%\end{table}%

\begin{table}[htbp]  
	\centering  
	\caption{Numerical errors and convergence rates of $ \textbf{u},p,\eta,c$ in $L^2$-norm with $h=1/200$}   
	\rowcolors{1}{gray!20}{white} % 从第1行开始交替上色
	\setlength{\tabcolsep}{2.6mm}{  
		\begin{tabular}{ccccccccc}         
			$\tau$ & $\|\textbf{u}-\textbf{u}_h\|_{L^2}$ & Order  & $\| p-p_h\|_{L^2} $  & Order  & $\| \eta-\eta_h\|_{L^2} $ & Order& $\| c-c_h\|_{L^2} $& Order \\   
		 1/4   & 0.011778 &       & 0.266525 &       & 0.031656 &       & 0.009429 &  \\    
1/8   & 0.008151 & 0.53  & 0.189623 & 0.49  & 0.015486 & 1.03  & 0.004609 & 1.03  \\   
1/16  & 0.002958 & 1.46  & 0.090238 & 1.07  & 0.007618 & 1.02  & 0.002254 & 1.03  \\    1/32  & 0.000663 & 2.16  & 0.028167 & 1.68  & 0.003782 & 1.01  & 0.00111 & 1.02  \\   
1/64  & 1.67E-04 & 1.99  & 7.72E-03 & 1.87  & 1.89E-03 & 1.00  & 0.000561 & 0.98   \end{tabular} }
	\label{label2}
\end{table}%

\begin{table}[htbp]  
	\centering  
	\caption{Relative errors and convergence rates of $ \textbf{u},p,\eta,c$ with $h=1/200$}   
	\rowcolors{1}{gray!20}{white} % 从第1行开始交替上色
	\setlength{\tabcolsep}{2.7mm}{   \begin{tabular}{ccccccccc}       
			$\tau$   &$\frac{\| \u-\u_h\|_{L^2}}{\| \u\|_{L^2}}$   & Order &$\frac{\| p-p_h\|_{L^2}}{\| p\|_{L^2}}$ & Order &$\frac{\| \eta-\eta_h\|_{L^2}}{\| \eta\|_{L^2}}$ & Order&$\frac{\| c-c_h\|_{L^2}}{\| c\|_{L^2}}$  &Order\\    
	1/4   & 0.019809 &       & 5.96817 &       & 0.075136 &       & 0.022442 &  \\    
1/8   & 0.013706 & 0.53  & 1.46102 & 2.03  & 0.036812 & 1.03  & 0.010964 & 1.03  \\    
1/16  & 0.004973 & 1.46  & 0.281443 & 2.38  & 0.018112 & 1.02  & 0.005359 & 1.03  \\   
1/32  & 0.001114 & 2.16  & 0.05449 & 2.37  & 0.008992 & 1.01  & 0.00264 & 1.02  \\   
1/64  & 0.00028 & 1.99  & 0.013026 & 2.06  & 0.004493 & 1.00  & 0.001334 & 0.99 
	\end{tabular}}
	\label{label4}
\end{table}%

%\begin{table}[htbp]  
%	\centering  
%	\caption{Add caption}    \begin{tabular}{ccccccccc}          
%		& ure   & rate  & pre   & rate  & etare & rate  & cre   & rate \\    
%		1/4   & 0.019809 &       & 5.96817 &       & 0.075136 &       & 0.022442 &  \\    
%		1/8   & 0.013706 & 0.53  & 1.46102 & 2.03  & 0.036812 & 1.03  & 0.010964 & 1.03  \\    
%		1/16  & 0.004973 & 1.46  & 0.281443 & 2.38  & 0.018112 & 1.02  & 0.005359 & 1.03  \\   
%		 1/32  & 0.001114 & 2.16  & 0.05449 & 2.37  & 0.008992 & 1.01  & 0.00264 & 1.02  \\   
%		 1/64  & 0.00028 & 1.99  & 0.013026 & 2.06  & 0.004493 & 1.00  & 0.001334 & 0.99  \\    \end{tabular}%  \label{tab:addlabel}%\end{table}%

\begin{figure}[htbp] % htbp???????????????  
	
	\centering % ???????  

	\begin{minipage}{0.32\textwidth} % ????????  
		\centering  
		\includegraphics[width=\textwidth]{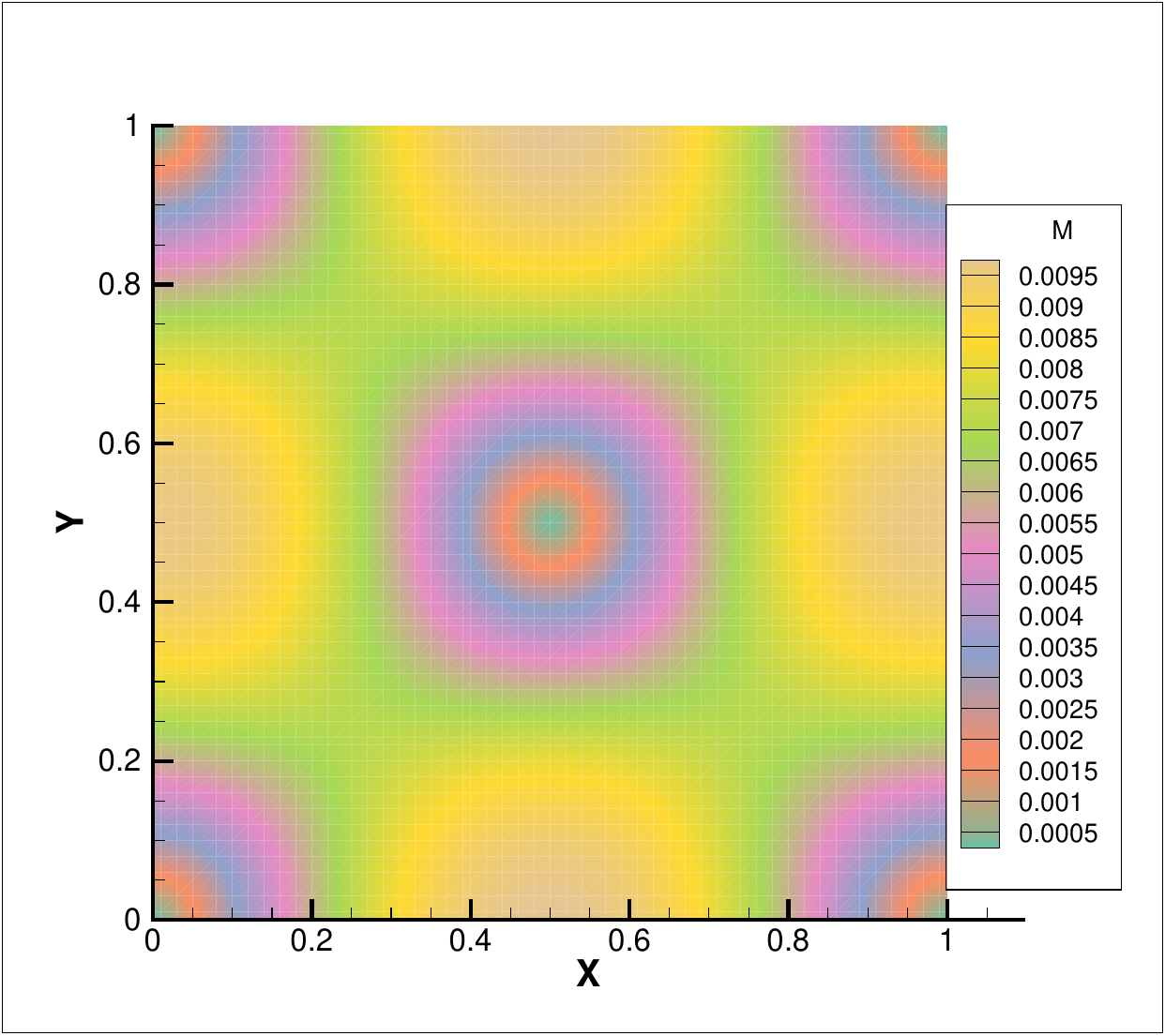}  
		% \caption{t=0}  
		%\label{fig:image1}  
	\end{minipage}  
	\begin{minipage}{0.32\textwidth} % ????????  
		\centering  
		\includegraphics[width=\textwidth]{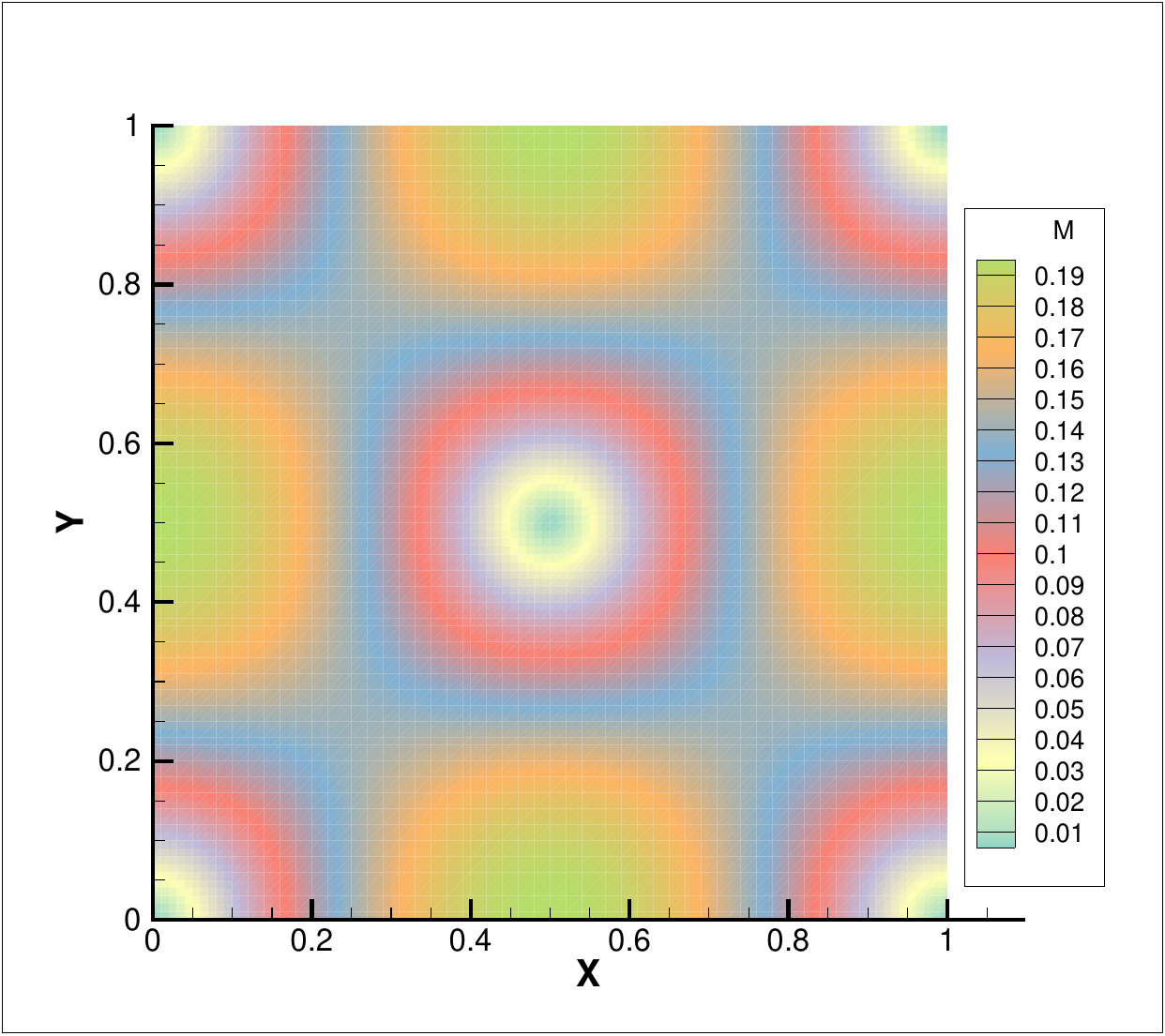} % ?????????
		%  \caption{3}  
		%  \label{fig:image3}  
	\end{minipage}  
	\begin{minipage}{0.32\textwidth}  
		\centering  
		\includegraphics[width=\textwidth]{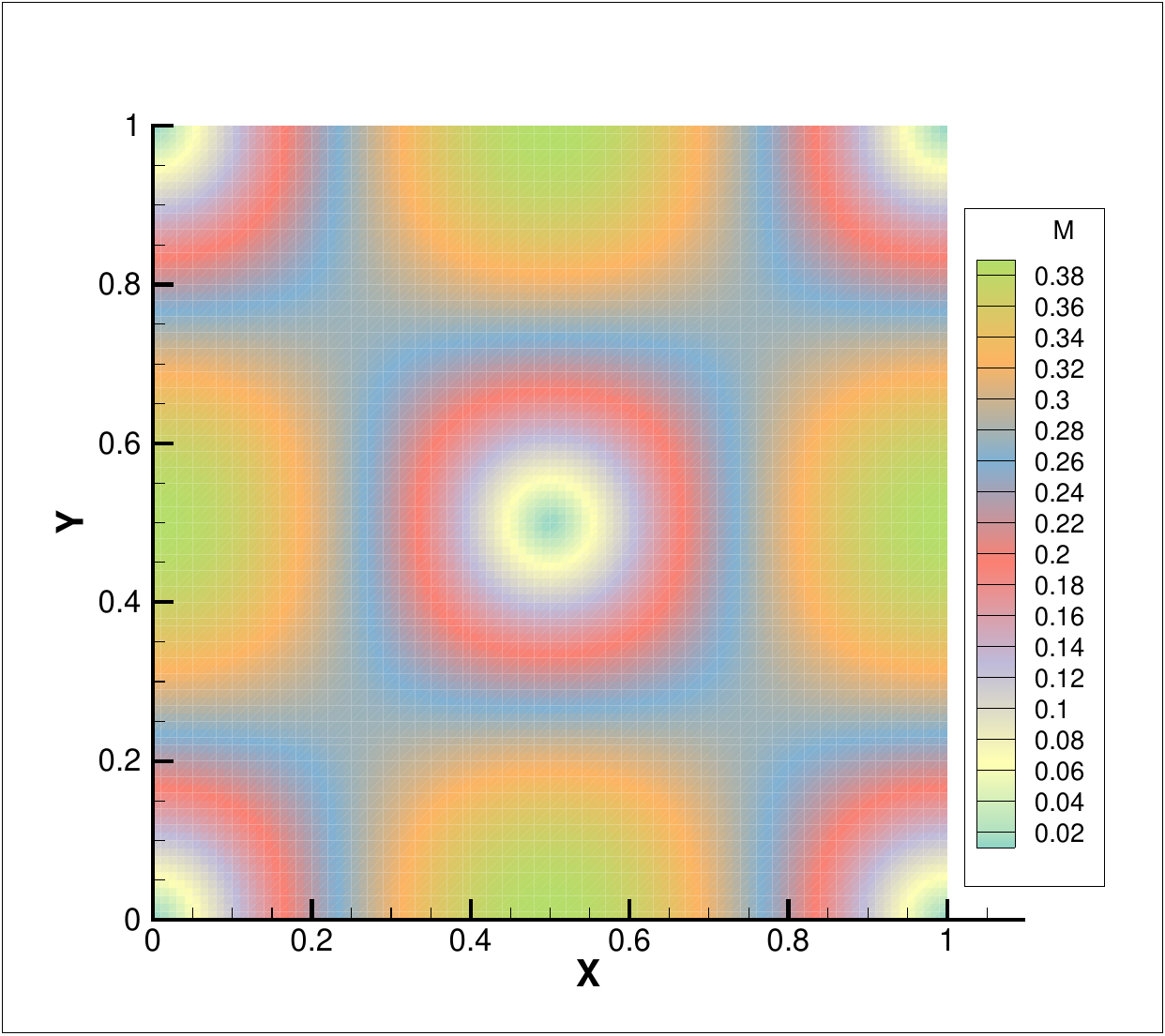}  
		%  \caption{6}  
		%  \label{fig:image6}  
	\end{minipage}  
	
	\begin{minipage}{0.32\textwidth} % ????????  
		\centering  
		\includegraphics[width=\textwidth]{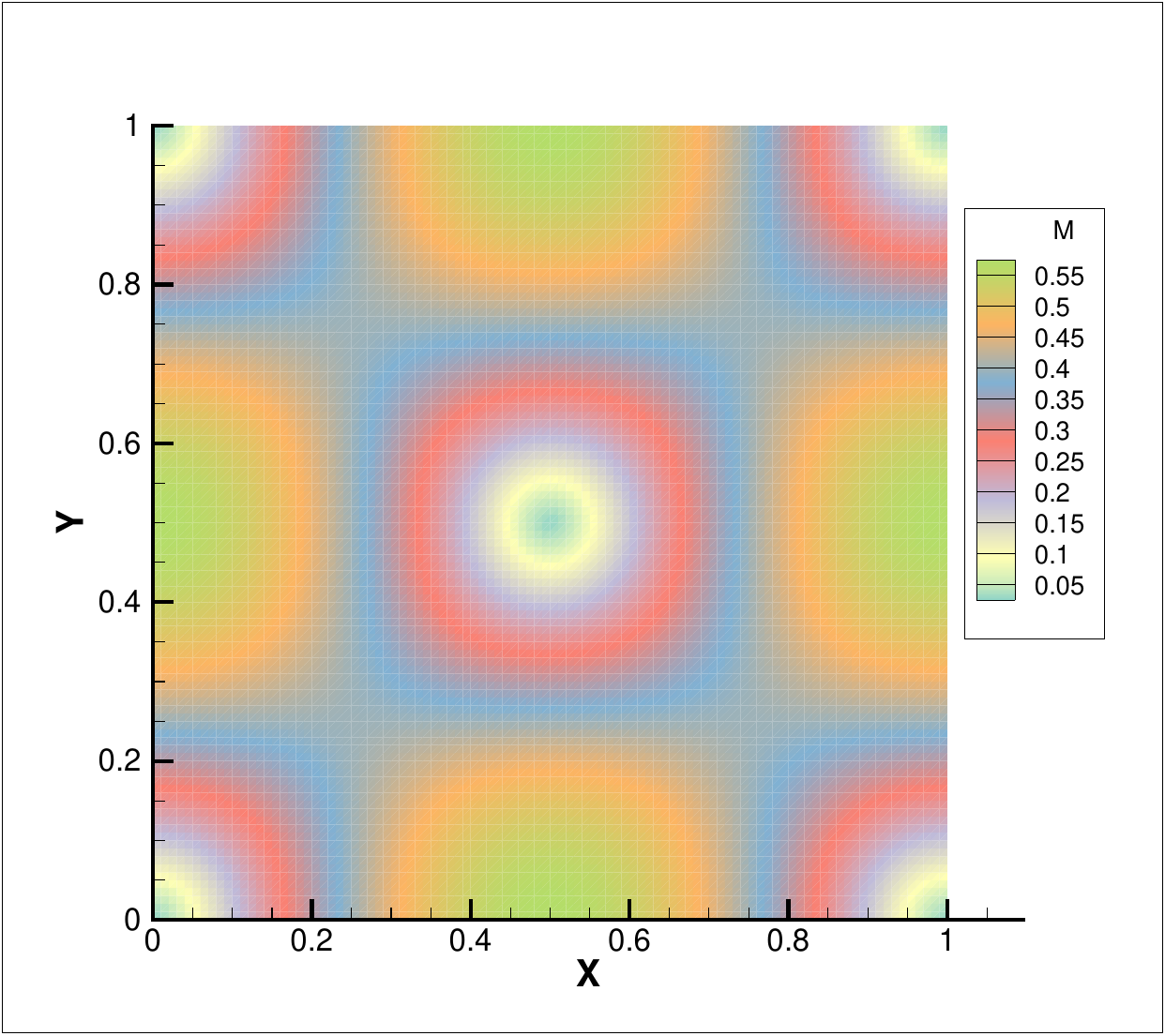}  
		% \caption{t=0}  
		%\label{fig:image1}  
	\end{minipage}  
	\begin{minipage}{0.32\textwidth} % ????????  
		\centering  
		\includegraphics[width=\textwidth]{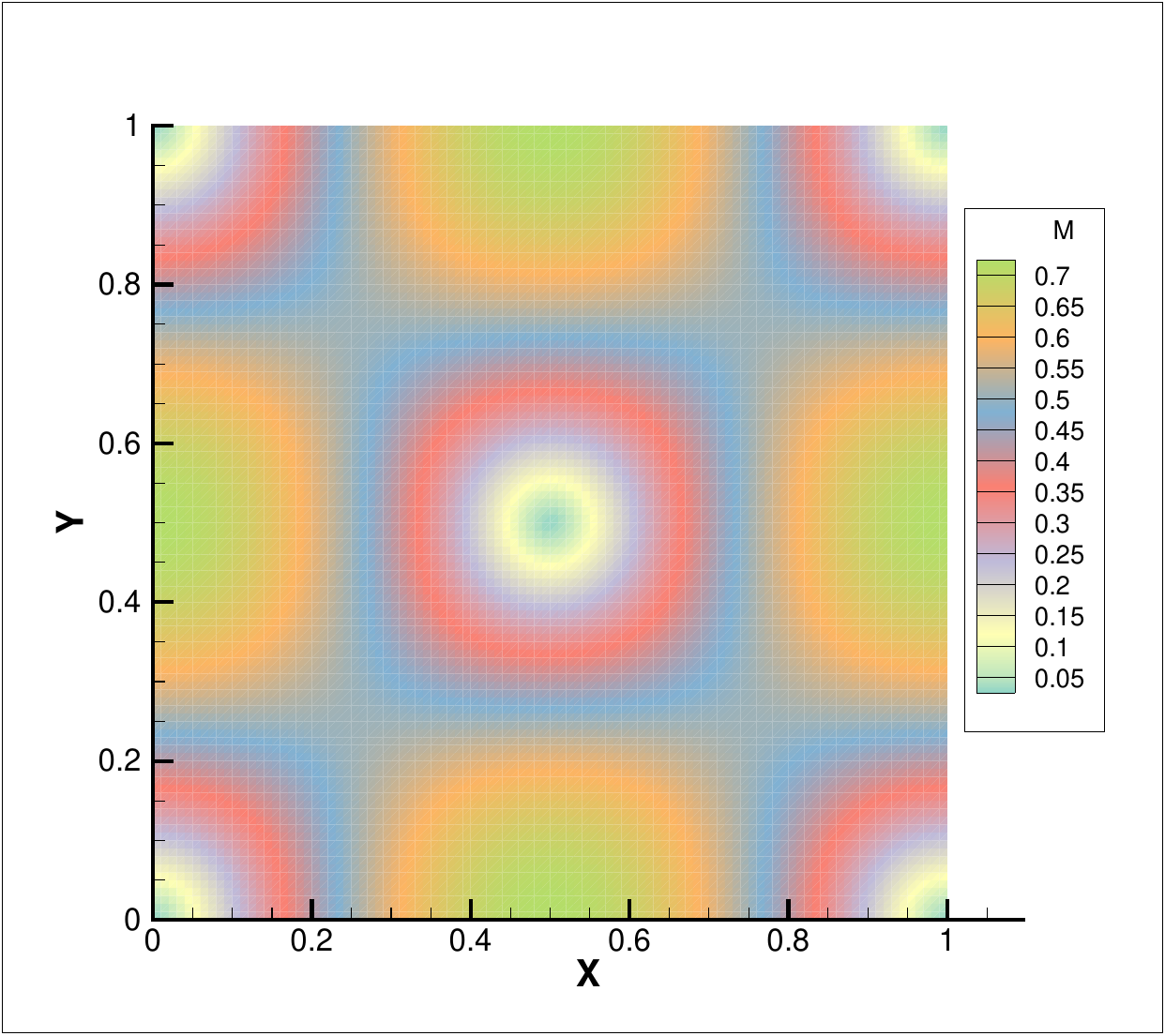} % ?????????
		%  \caption{3}  
		%  \label{fig:image3}  
	\end{minipage}  
	\begin{minipage}{0.32\textwidth}  
		\centering  
		\includegraphics[width=\textwidth]{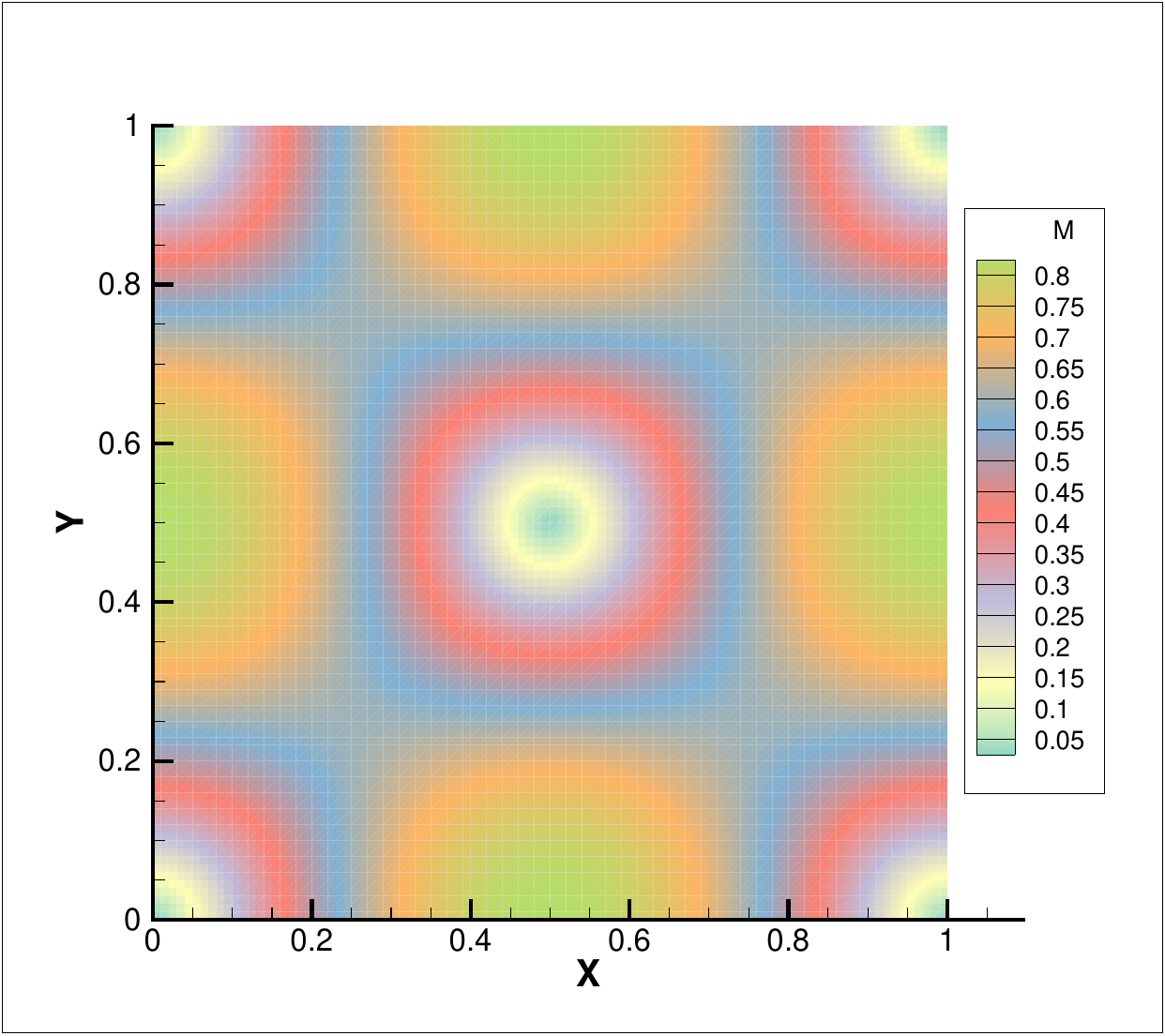}  
		%  \caption{6}  
		%  \label{fig:image6}  
	\end{minipage}  
	% \vspace{1.0em}
	\caption{Numerical solutions of velocity at times t = 0, 0.2, 0.4, 0.6, 0.8, 1.0.}  
	\label{velocity}  
\end{figure}

\begin{figure}[htbp] % htbp???????????????  
	
	\centering % ???????  

	\begin{minipage}{0.32\textwidth} % ????????  
		\centering  
		\includegraphics[width=\textwidth]{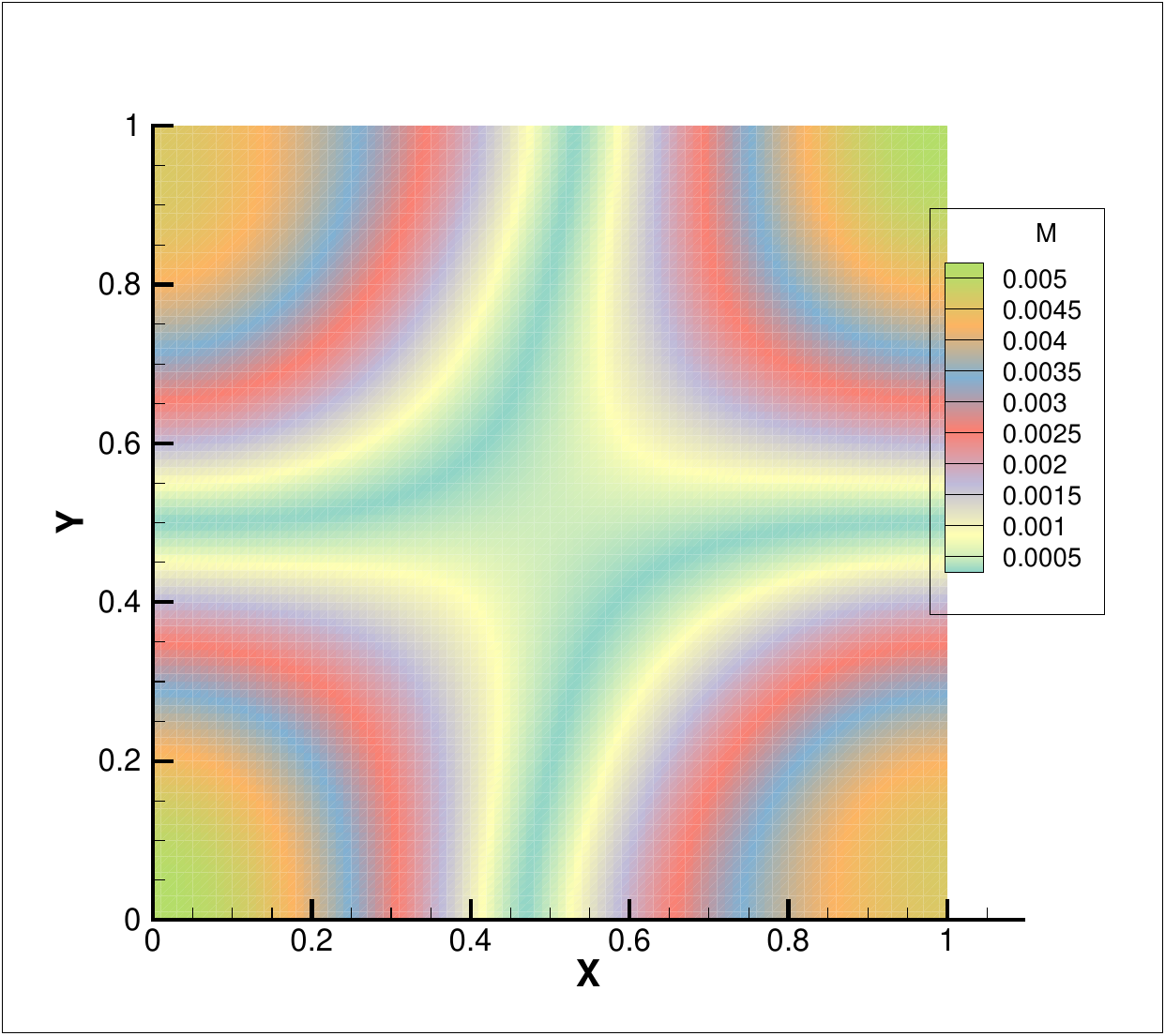}  
		% \caption{t=0}  
		%\label{fig:image1}  
	\end{minipage}  
	\begin{minipage}{0.32\textwidth} % ????????  
		\centering  
		\includegraphics[width=\textwidth]{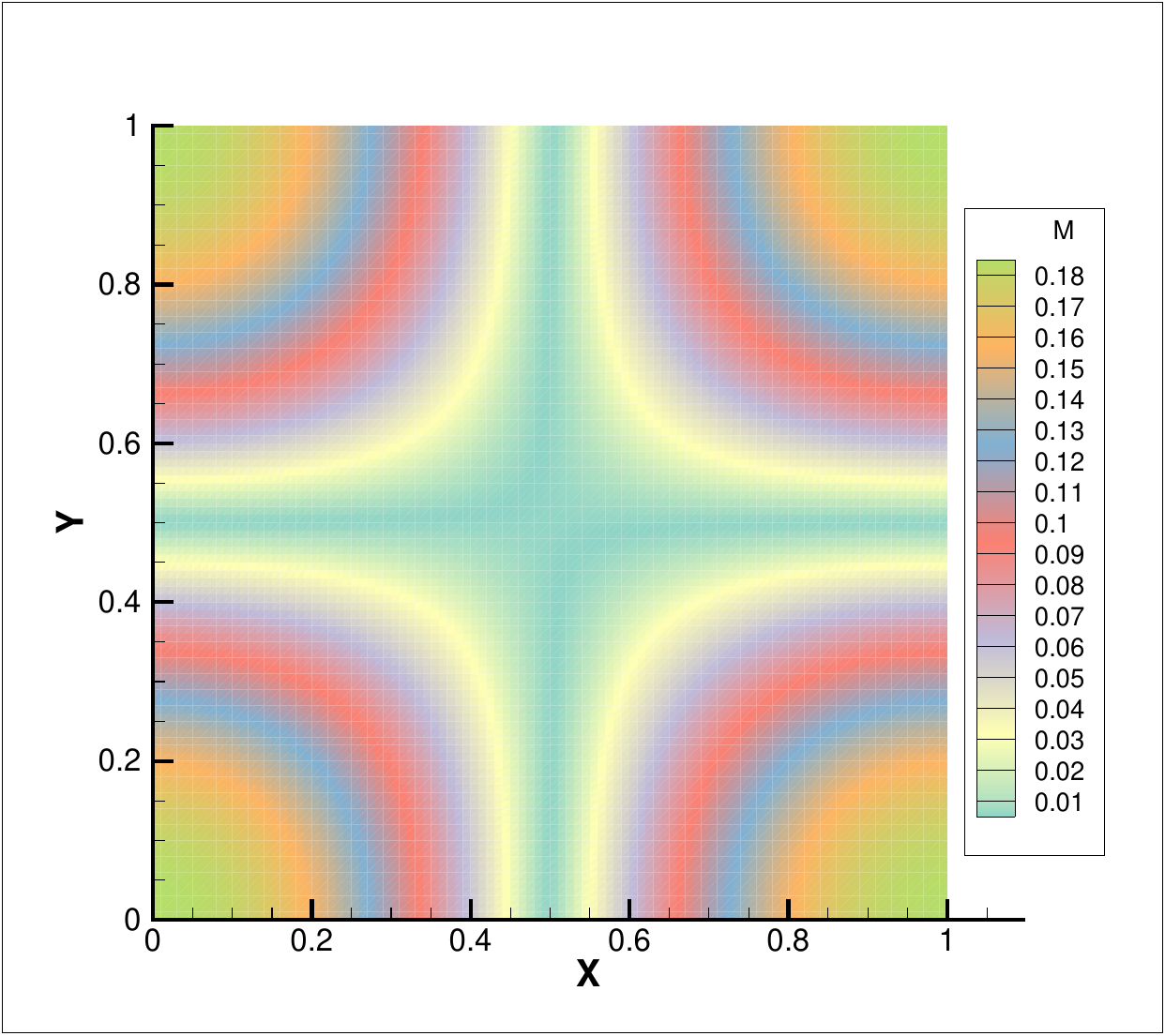} % ?????????
		%  \caption{3}  
		%  \label{fig:image3}  
	\end{minipage}  
	\begin{minipage}{0.32\textwidth}  
		\centering  
		\includegraphics[width=\textwidth]{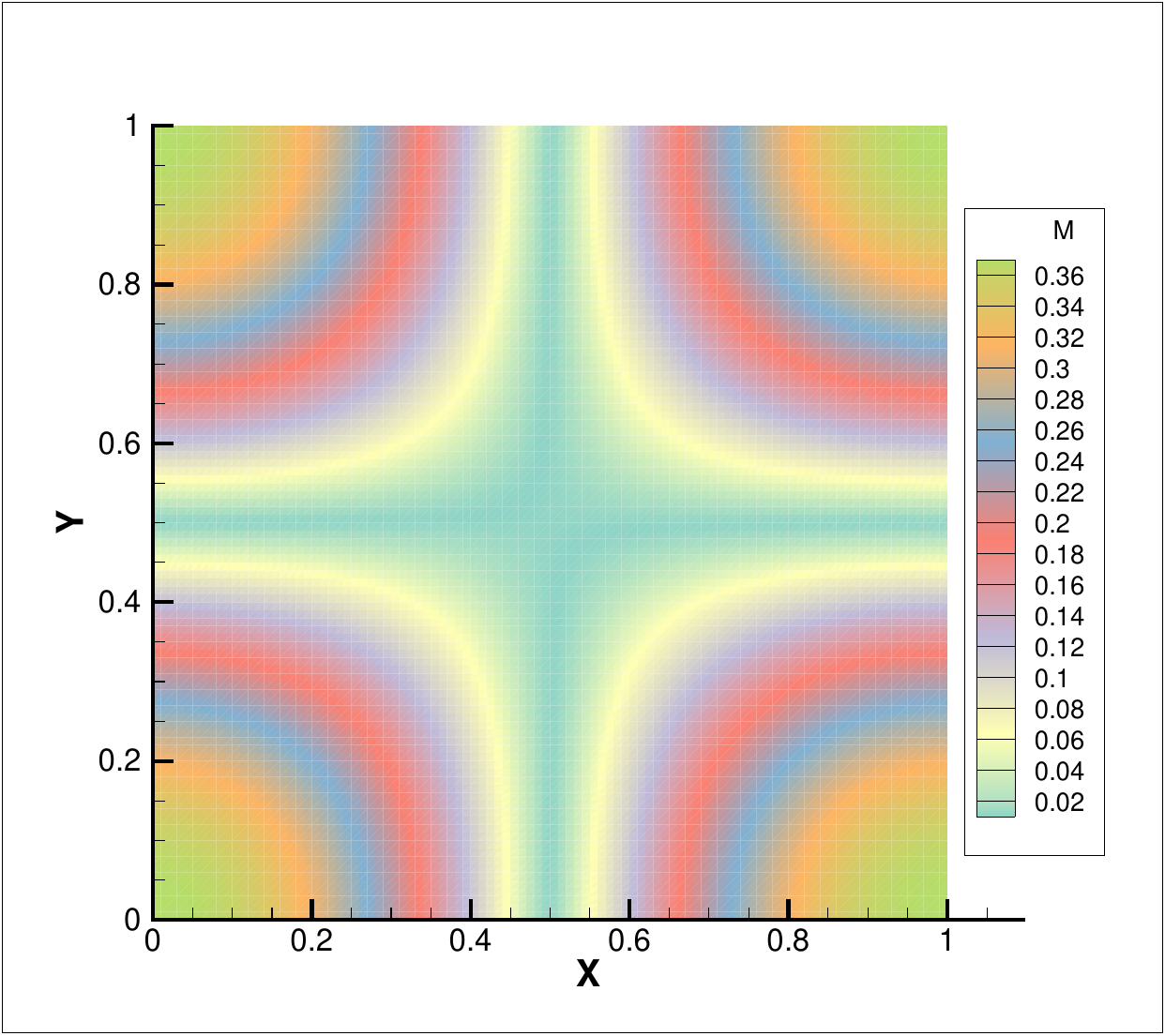}  
		%  \caption{6}  
		%  \label{fig:image6}  
	\end{minipage}  
	
	\begin{minipage}{0.32\textwidth} % ????????  
		\centering  
		\includegraphics[width=\textwidth]{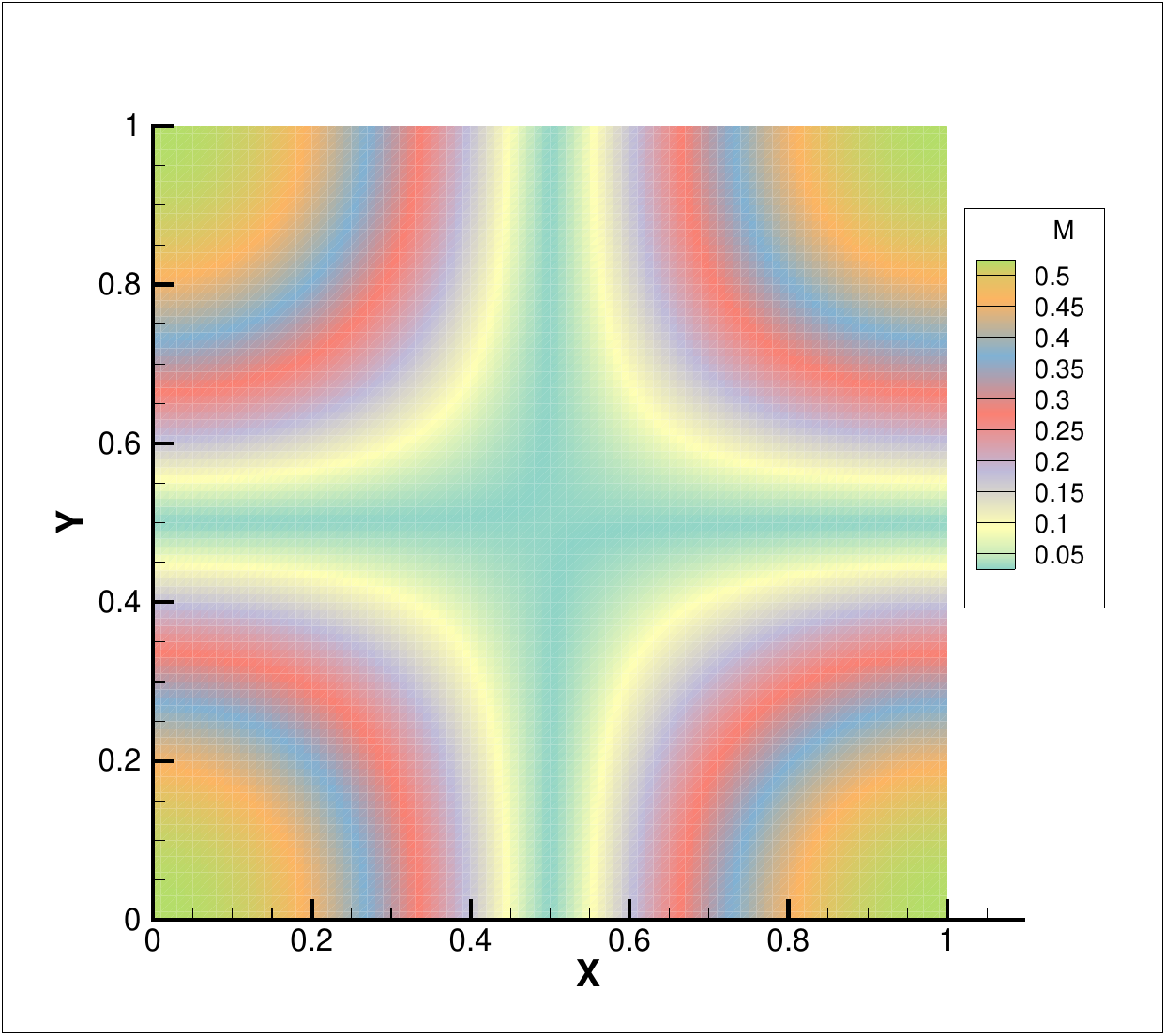}  
		% \caption{t=0}  
		%\label{fig:image1}  
	\end{minipage}  
	\begin{minipage}{0.32\textwidth} % ????????  
		\centering  
		\includegraphics[width=\textwidth]{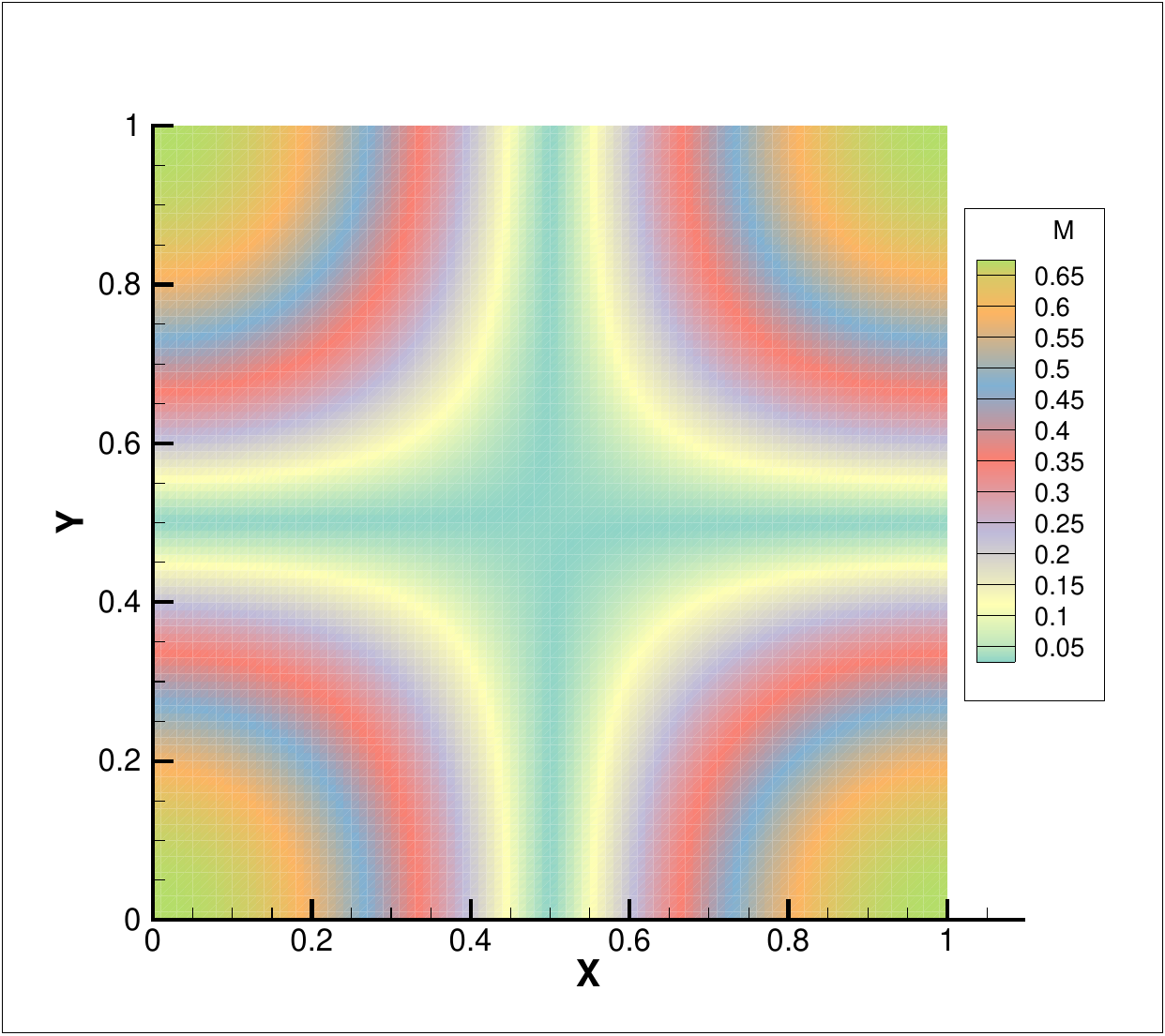} % ?????????
		%  \caption{3}  
		%  \label{fig:image3}  
	\end{minipage}  
	\begin{minipage}{0.32\textwidth}  
		\centering  
		\includegraphics[width=\textwidth]{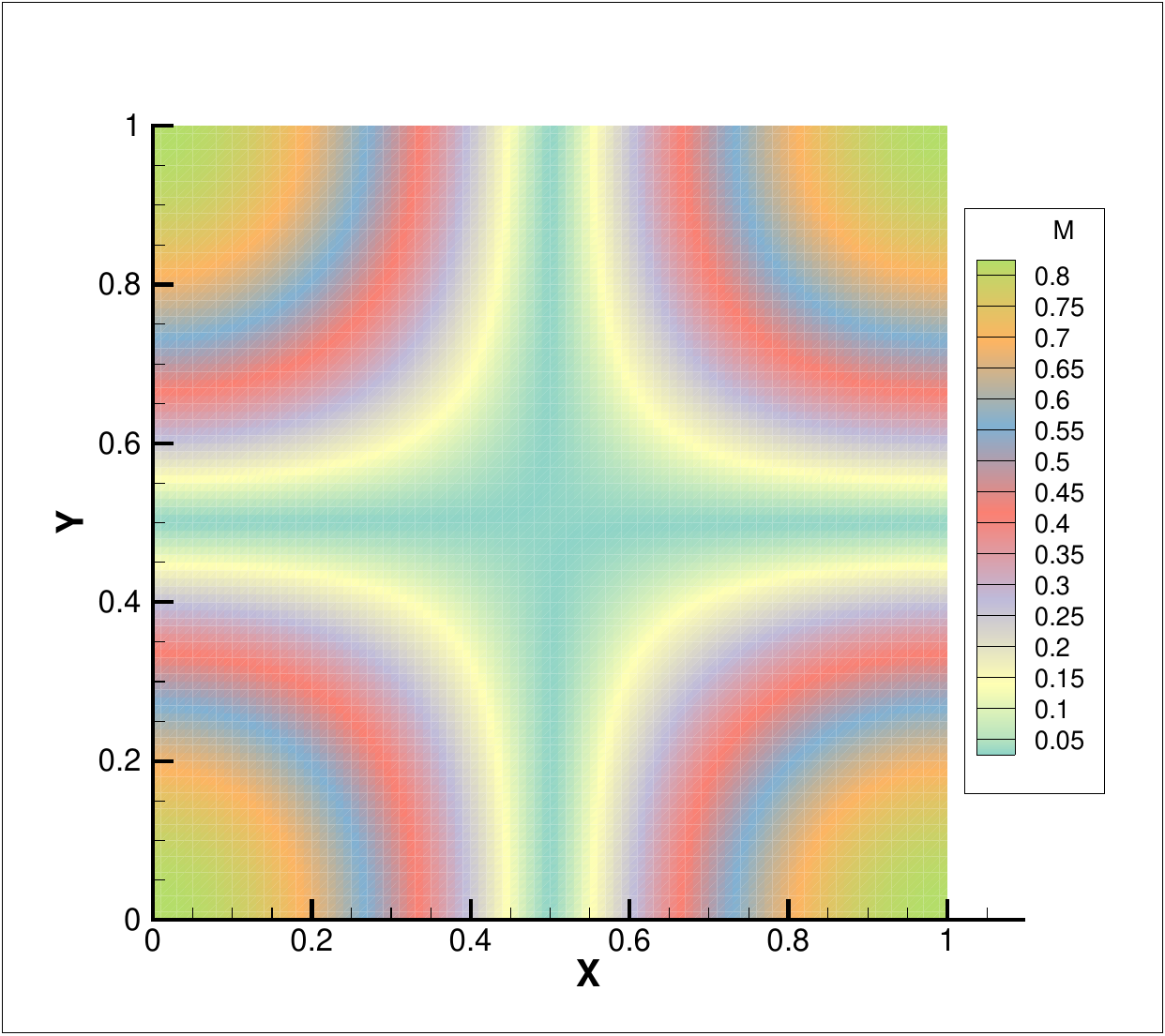}  
		%  \caption{6}  
		%  \label{fig:image6}  
	\end{minipage}  
	% \vspace{1.0em}
	\caption{Numerical solutions of pressure at times t = 0, 0.2, 0.4, 0.6, 0.8, 1.0.}  
	\label{pressure}  
\end{figure}

\begin{figure}
	\begin{minipage}{0.32\textwidth} % ????????  
	\centering  
	\includegraphics[width=\textwidth]{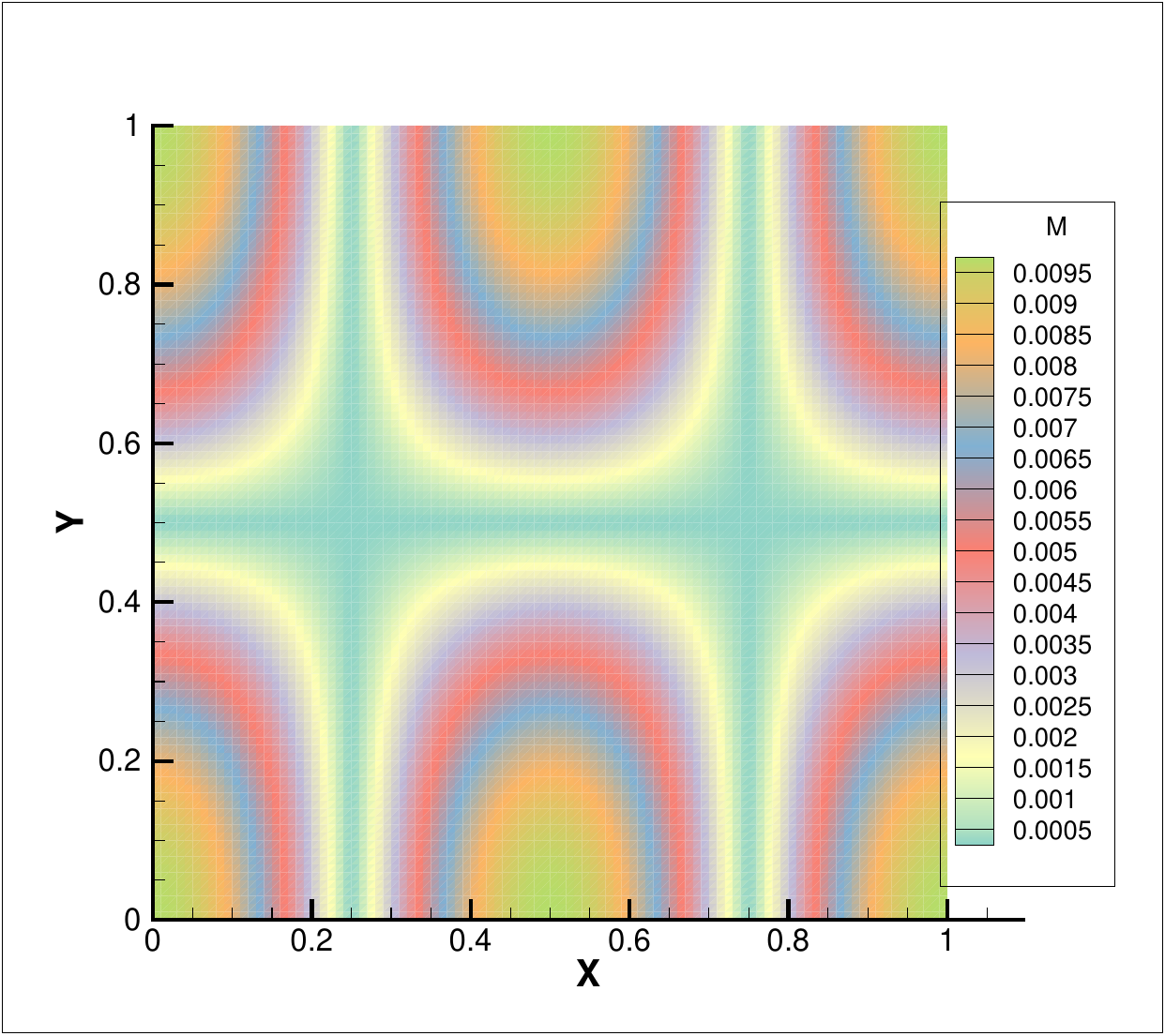}  
	% \caption{t=0}  
	%\label{fig:image1}  
\end{minipage}  
\begin{minipage}{0.32\textwidth} % ????????  
	\centering  
	\includegraphics[width=\textwidth]{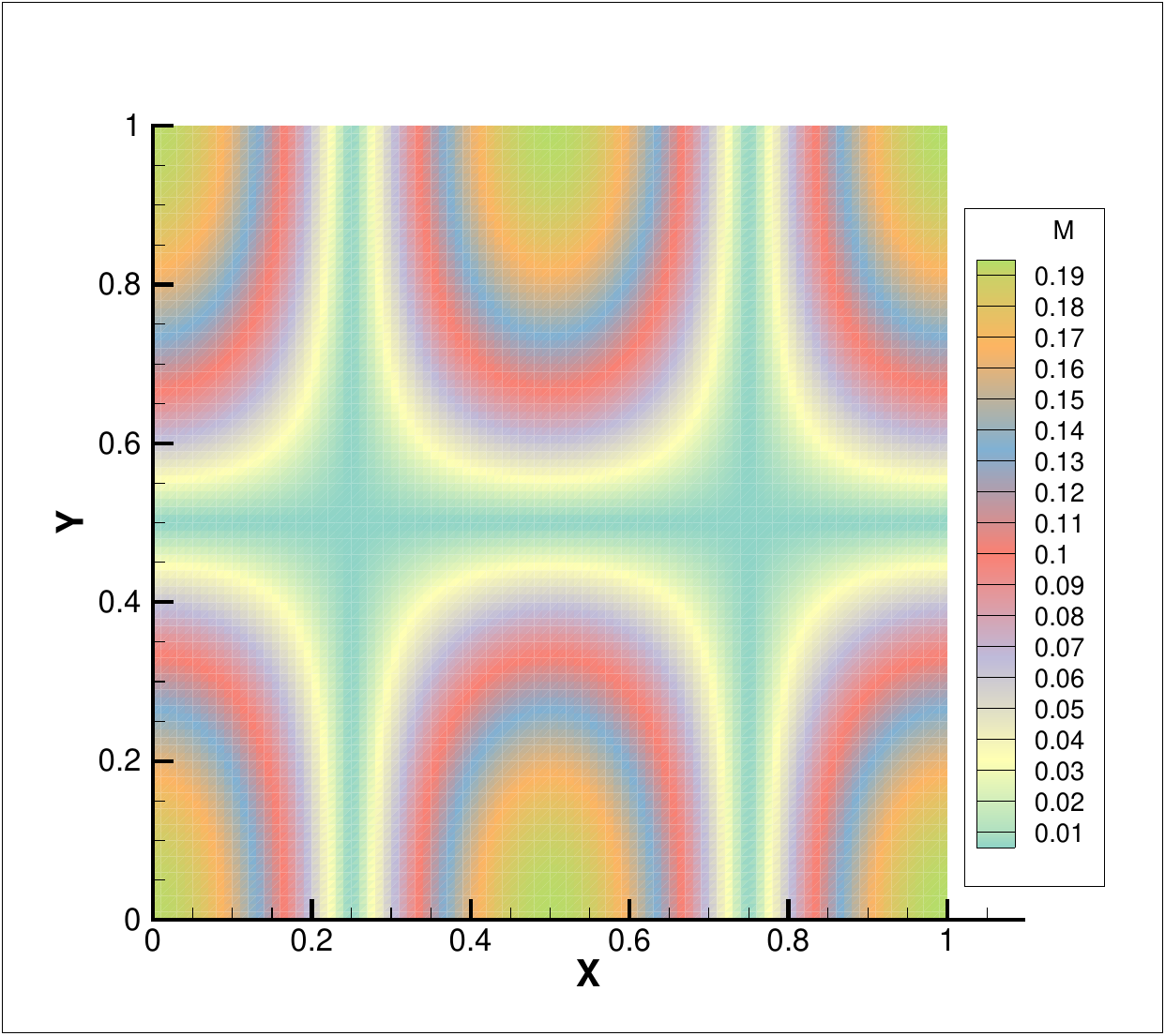} % ?????????
	%  \caption{3}  
	%  \label{fig:image3}  
\end{minipage}  
\begin{minipage}{0.32\textwidth}  
	\centering  
	\includegraphics[width=\textwidth]{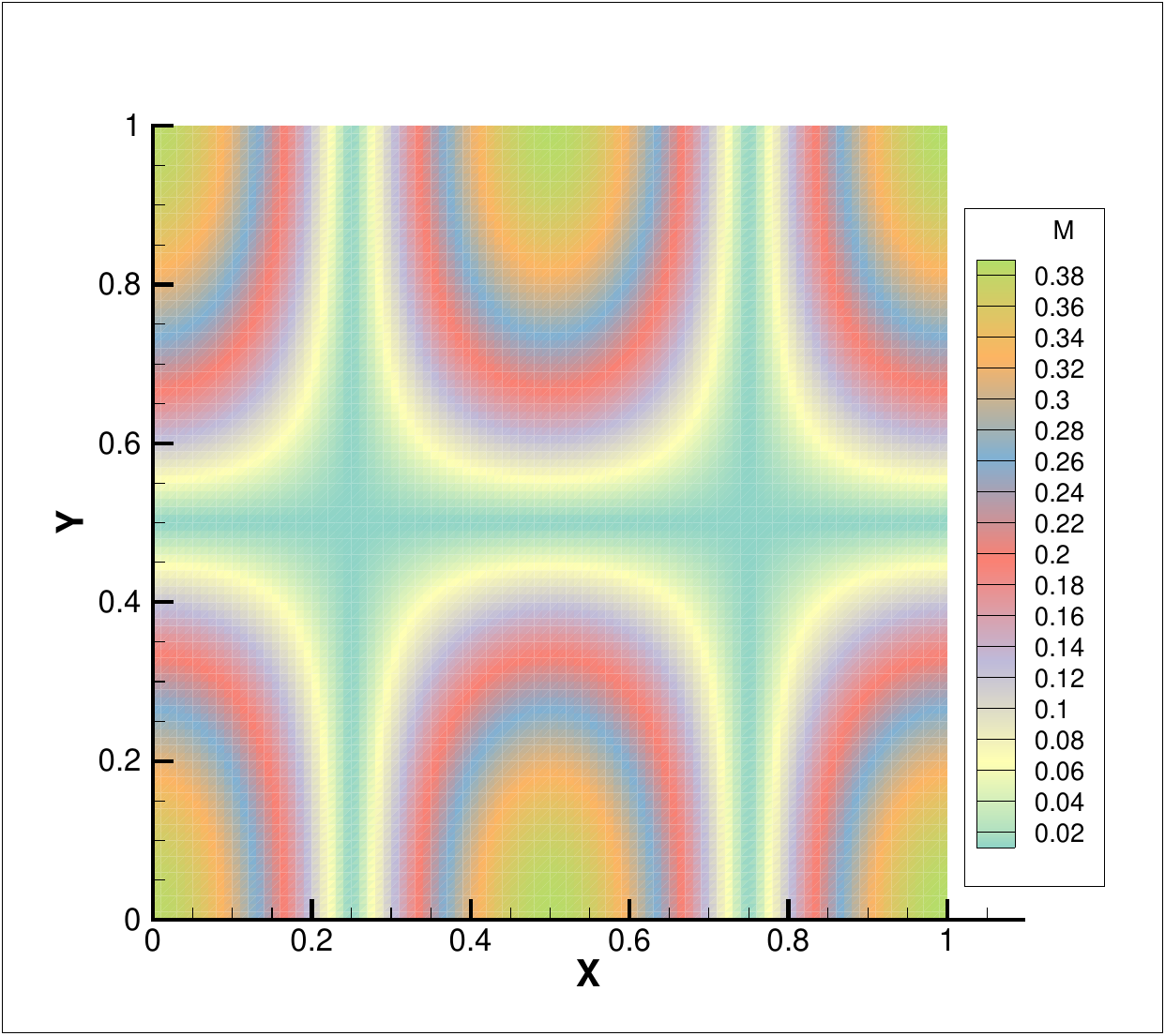}  
	%  \caption{6}  
	%  \label{fig:image6}  
\end{minipage}  

\begin{minipage}{0.32\textwidth} % ????????  
	\centering  
	\includegraphics[width=\textwidth]{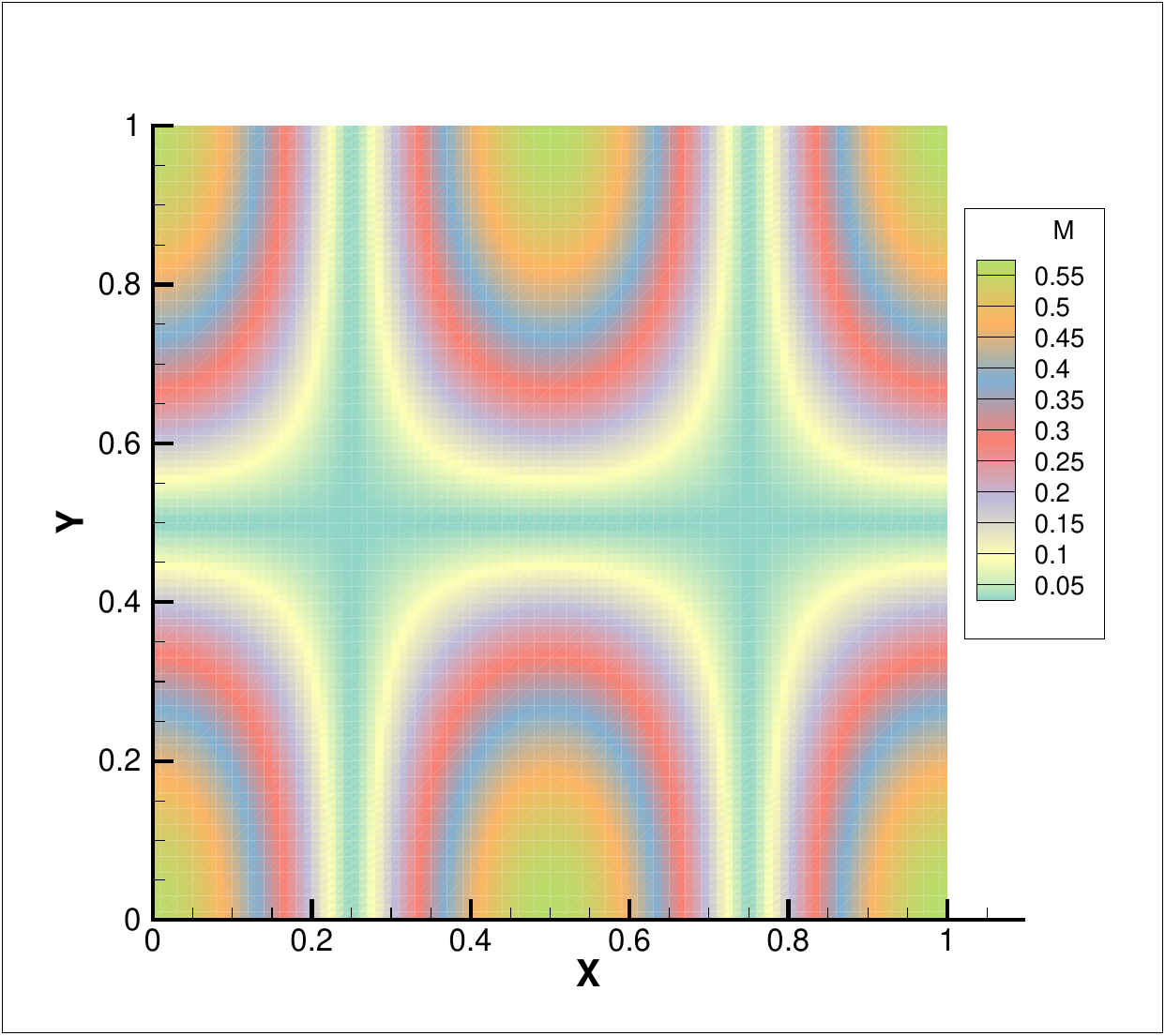}  
	% \caption{t=0}  
	%\label{fig:image1}  
\end{minipage}  
\begin{minipage}{0.32\textwidth} % ????????  
	\centering  
	\includegraphics[width=\textwidth]{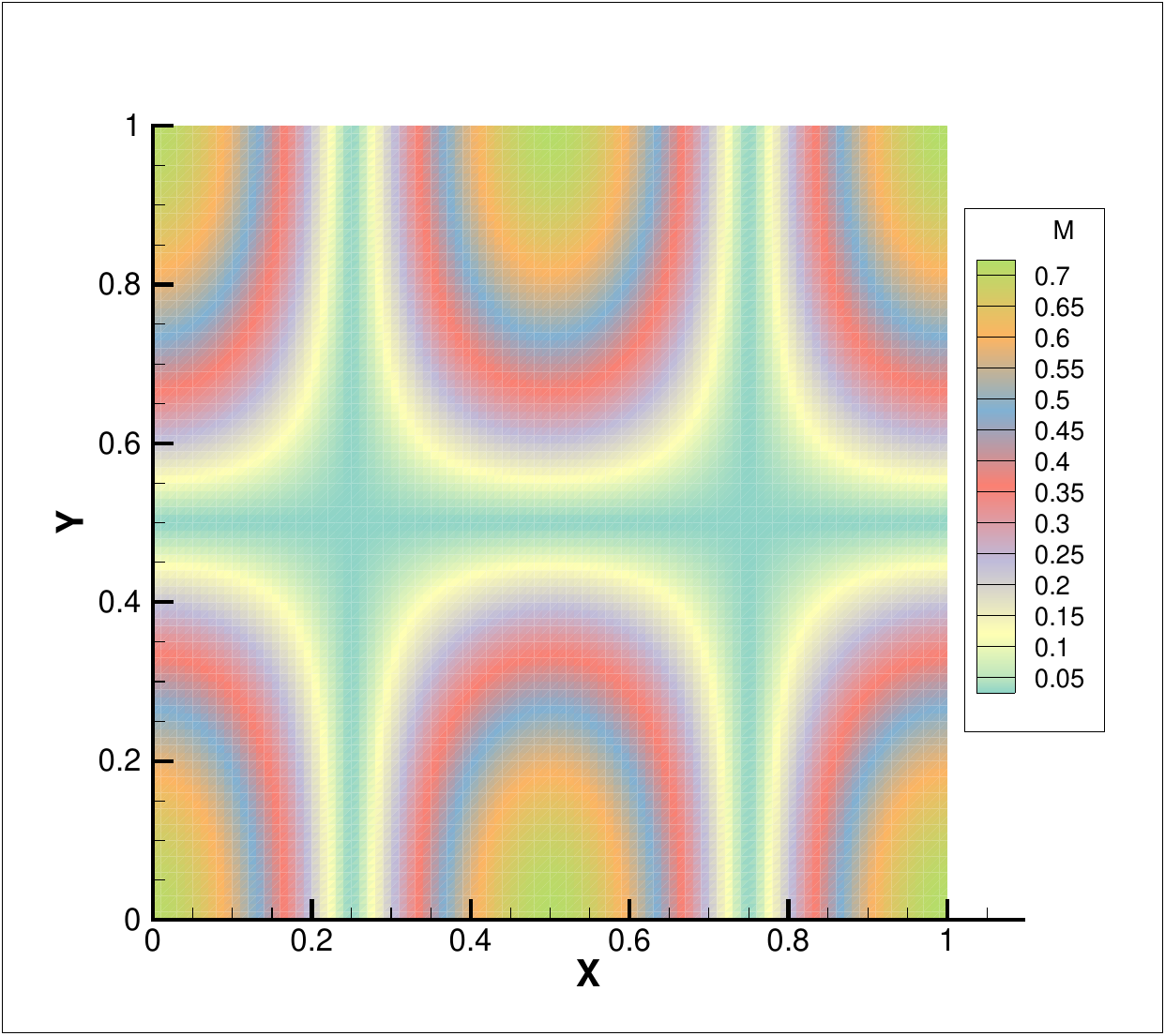} % ?????????
	%  \caption{3}  
	%  \label{fig:image3}  
\end{minipage}  
\begin{minipage}{0.32\textwidth}  
	\centering  
	\includegraphics[width=\textwidth]{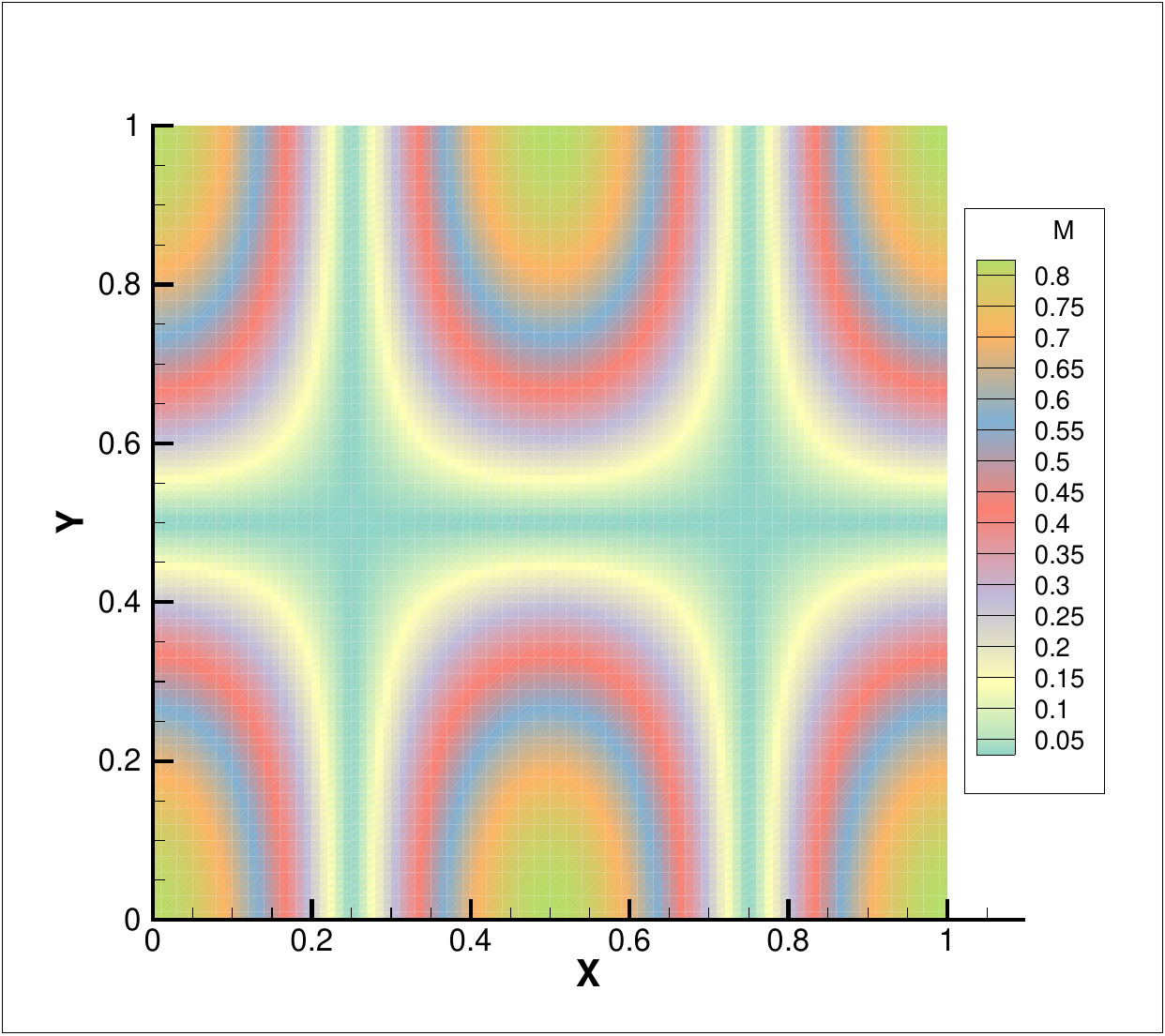}  
	%  \caption{6}  
	%  \label{fig:image6}  
\end{minipage}  
% \vspace{1.0em}
\caption{Numerical solutions of  the cell density at times t = 0, 0.2, 0.4, 0.6, 0.8, 1.0.}  
\label{celldensity}  
\end{figure}

\begin{figure}
	\begin{minipage}{0.32\textwidth} % ????????  
		\centering  
		\includegraphics[width=\textwidth]{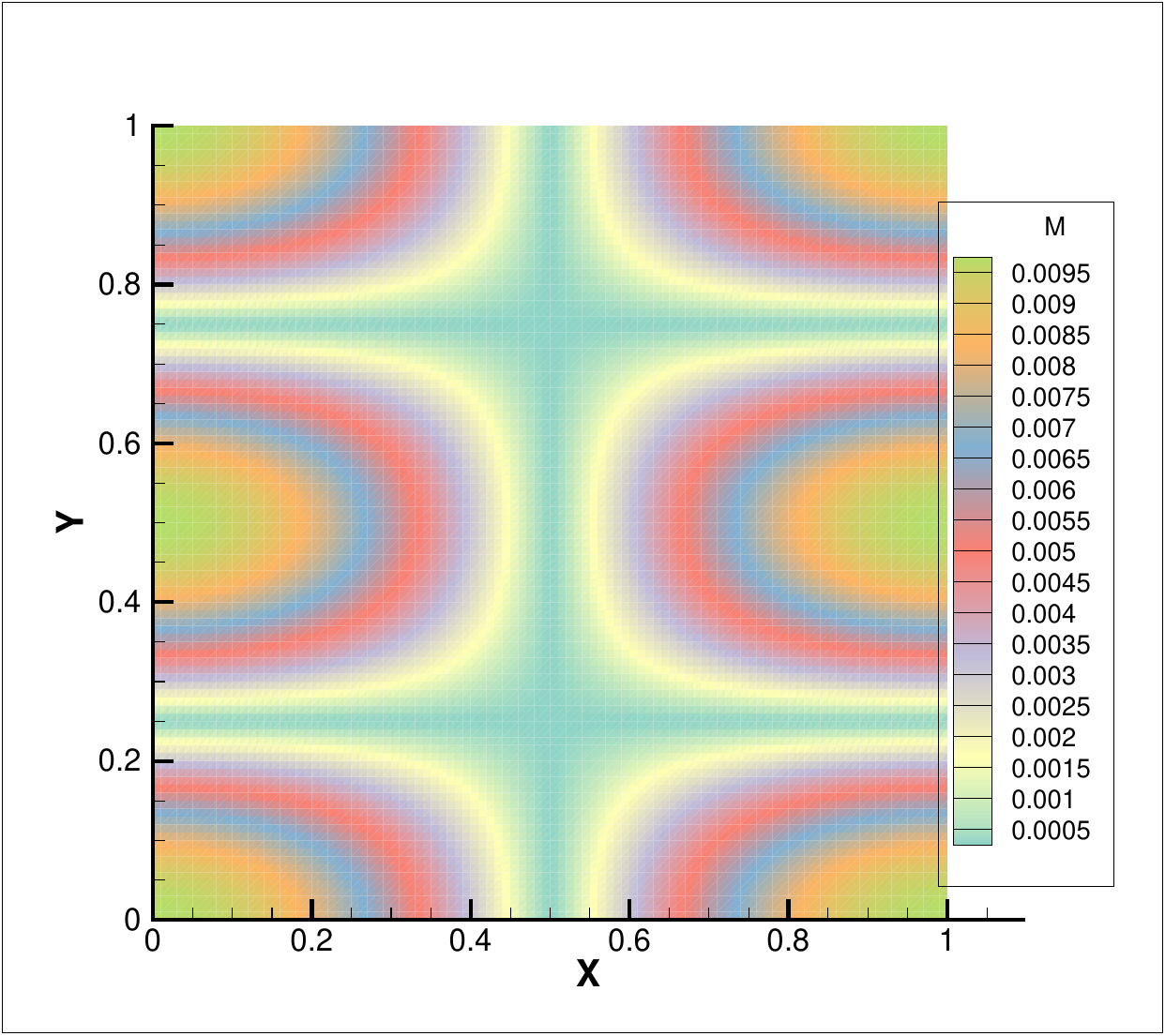}  
		% \caption{t=0}  
		%\label{fig:image1}  
	\end{minipage}  
	\begin{minipage}{0.32\textwidth} % ????????  
		\centering  
		\includegraphics[width=\textwidth]{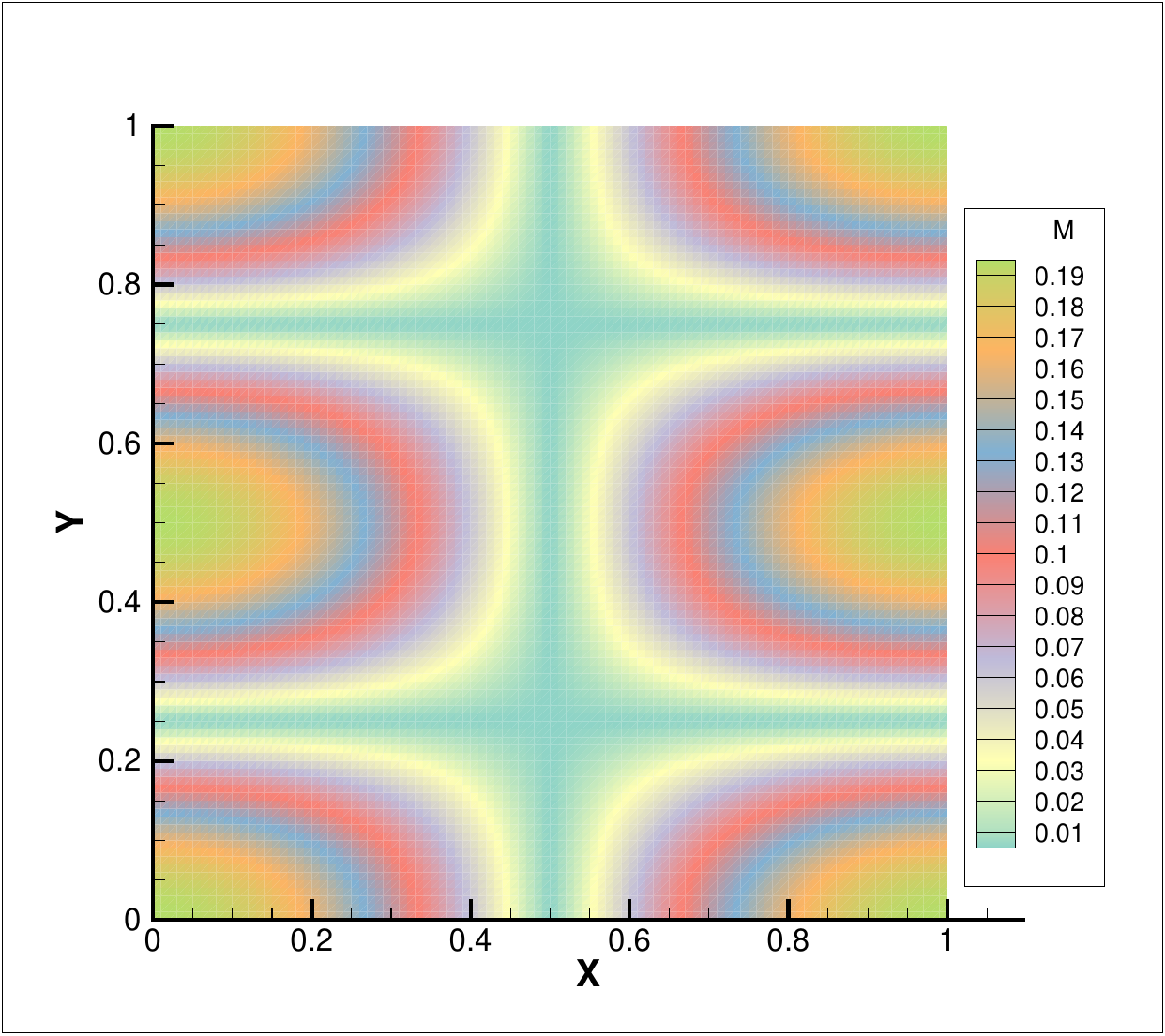} % ?????????
		%  \caption{3}  
		%  \label{fig:image3}  
	\end{minipage}  
	\begin{minipage}{0.32\textwidth}  
		\centering  
		\includegraphics[width=\textwidth]{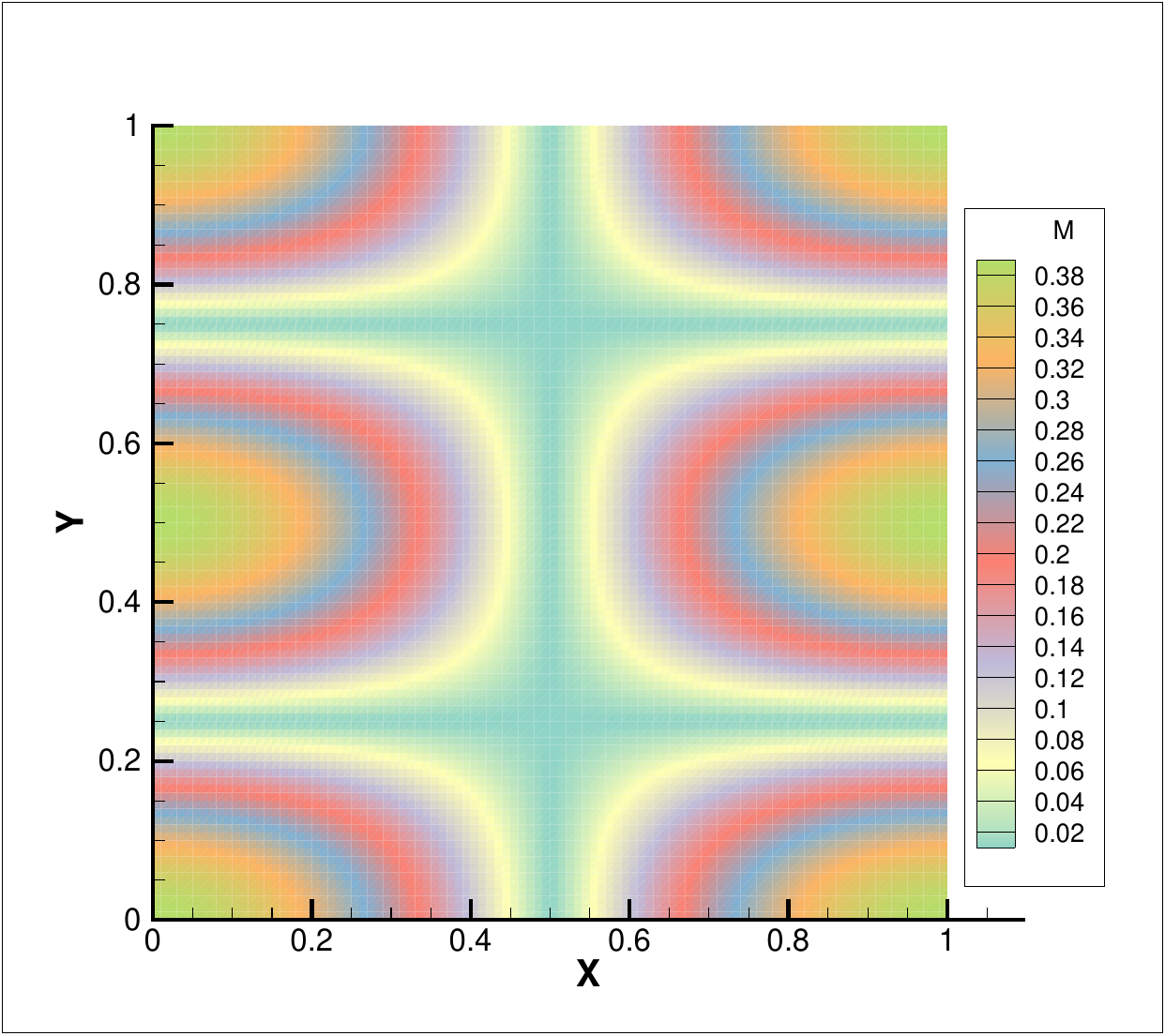}  
		%  \caption{6}  
		%  \label{fig:image6}  
	\end{minipage}  
	
	\begin{minipage}{0.32\textwidth} % ????????  
		\centering  
		\includegraphics[width=\textwidth]{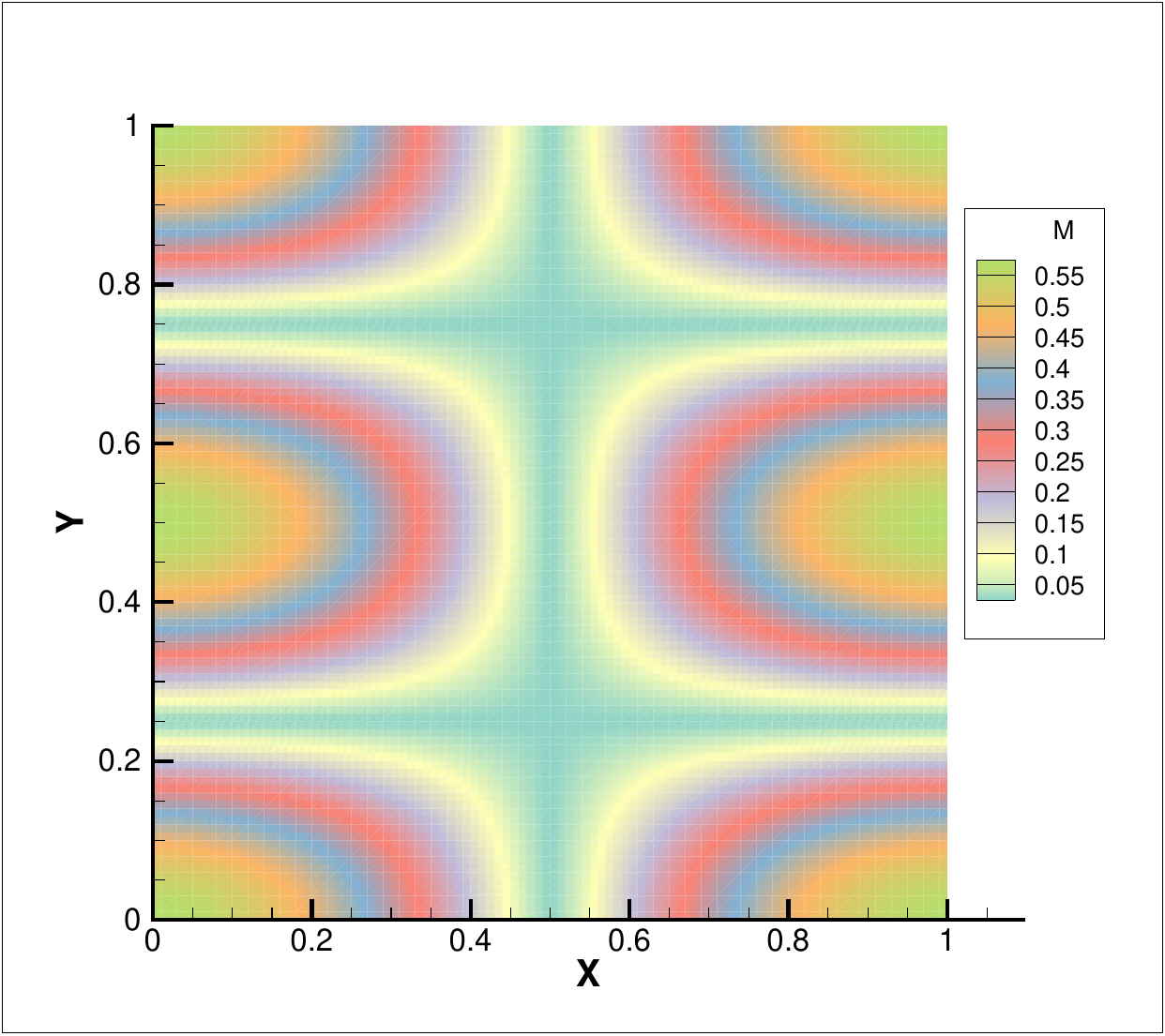}  
		% \caption{t=0}  
		%\label{fig:image1}  
	\end{minipage}  
	\begin{minipage}{0.32\textwidth} % ????????  
		\centering  
		\includegraphics[width=\textwidth]{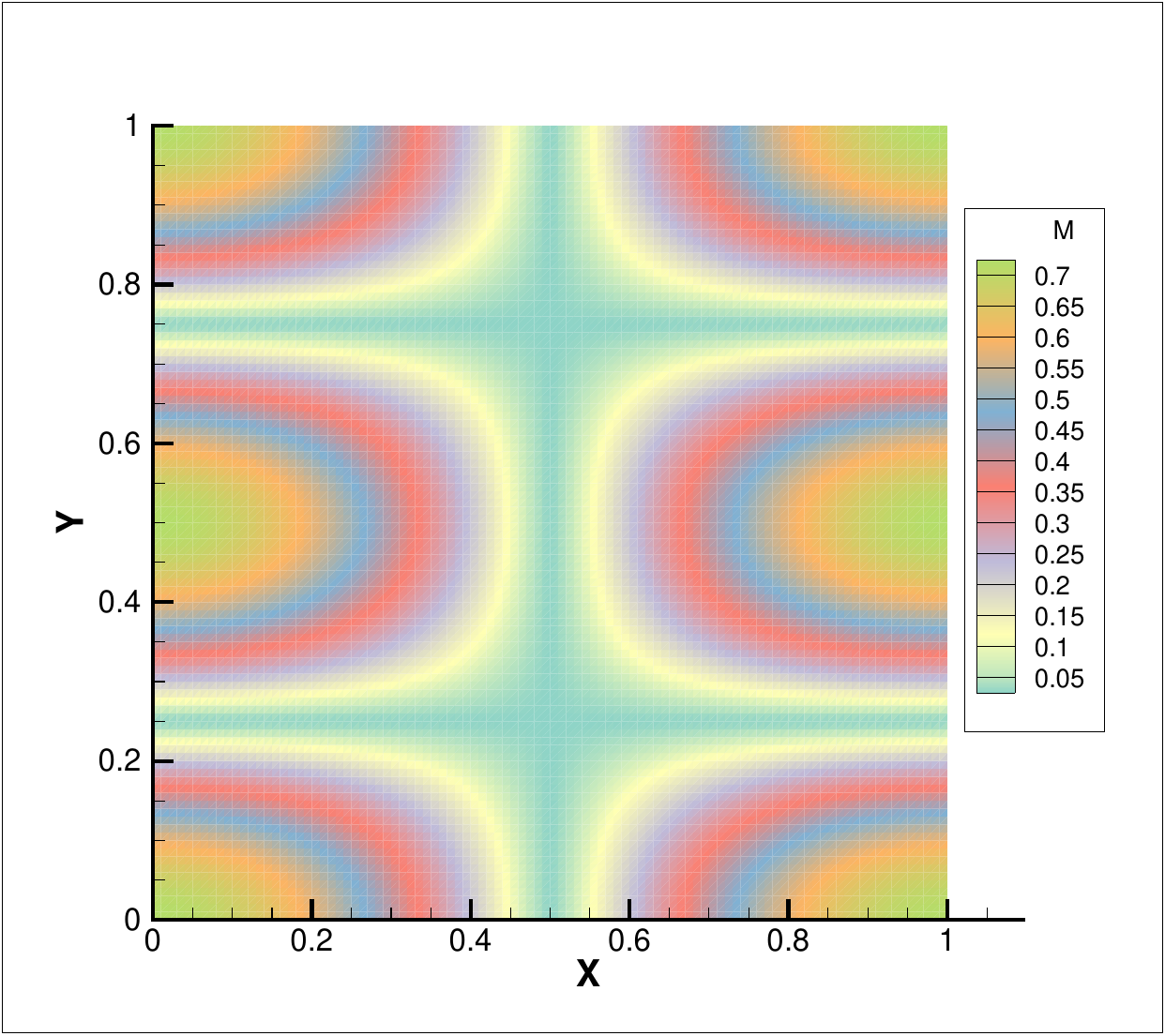} % ?????????
		%  \caption{3}  
		%  \label{fig:image3}  
	\end{minipage}  
	\begin{minipage}{0.32\textwidth}  
		\centering  
		\includegraphics[width=\textwidth]{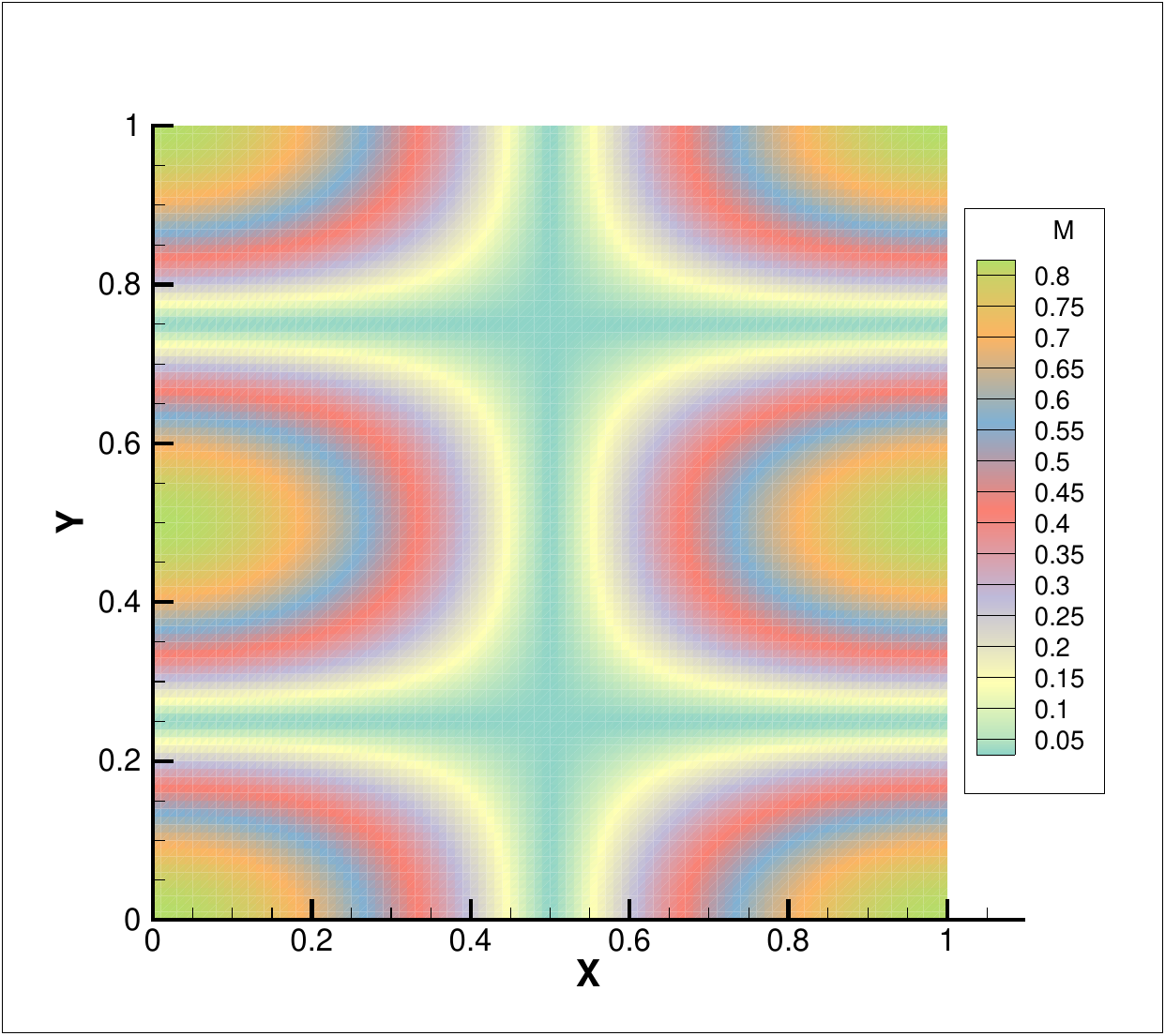}  
		%  \caption{6}  
		%  \label{fig:image6}  
	\end{minipage}  
	% \vspace{1.0em}
	\caption{Numerical solutions of the concentration of chemical substances at times t = 0, 0.2, 0.4, 0.6, 0.8, 1.0.}  
	\label{concentration}  
\end{figure}

%%%%%%%%%%%%%%%%%%%%%%%%%%%%%%%%%%%
\subsection{Chemotaxis phenomena in a liquid environment}
%%%%%%%%%%%%%%%%%%%%%%%%%%%%%%%%%%%%%%%%%%%%%%%%%%%%%%%%%%%%%
In this numerical experiment, we consider a rectangular domain $\Omega = [0,2] \times [0,1]$, with the following initial conditions:
\begin{align*}
	\eta_0(x,y) &= \sum_{i=1}^{3} \left( 70 \exp \left( -8(x - s_i)^2 - 10(y - 1)^2 \right) \right),\\
	c_0(x,y) &= 30 \exp(-5(x - 1)^2 - 5(y - 0.5)^2),\\
	\u_0(x,y) &= \mathbf{0},
\end{align*}
where $s_1=0.2$, $s_2=0.5$, and $s_3=1.2$.

The numerical solution is computed using a mesh size parameter of $h = \frac{1}{100}$ and a time step of $\tau = 10^{-5}$. The model parameters are chosen as follows: $\beta = 8$, $a_\eta = 4$, $\gamma = 6$, $\nu = 10$, $a_c = 1$, and an external potential given by $\phi(x,y) = -1000y$.

We present simulation results at various times $t= 10^{-5}, 2 \times 10^{-4}, 5 \times10^{-4},10^{-3},2 \times10^{-3},5 \times10^{-3}$, highlighting the dynamic evolution of the system. The distributions of cell density and chemical concentration are shown in Figure \ref{chemo-celldensity}, while The velocity field evolution is illustrated in Figure \ref{chemo-velocity}.

Initially, the cells migrate towards the center of the domain, where the chemical attractant exhibits the highest concentration. As time progresses, the chemical is gradually consumed, leading to a reduction in its concentration. Consequently, the cells begin to diffuse across the domain. Eventually, influenced by the force of gravity, the cells are observed to accumulate near the bottom boundary of the domain. This behavior is qualitatively consistent with the experimental and numerical observations reported in \cite{duarte2021}.

\begin{figure}[htbp] % htbp???????????????  
	\centering % ???????  
	\begin{minipage}{0.32\textwidth} % ????????  
		\centering  
		\includegraphics[width=\textwidth]{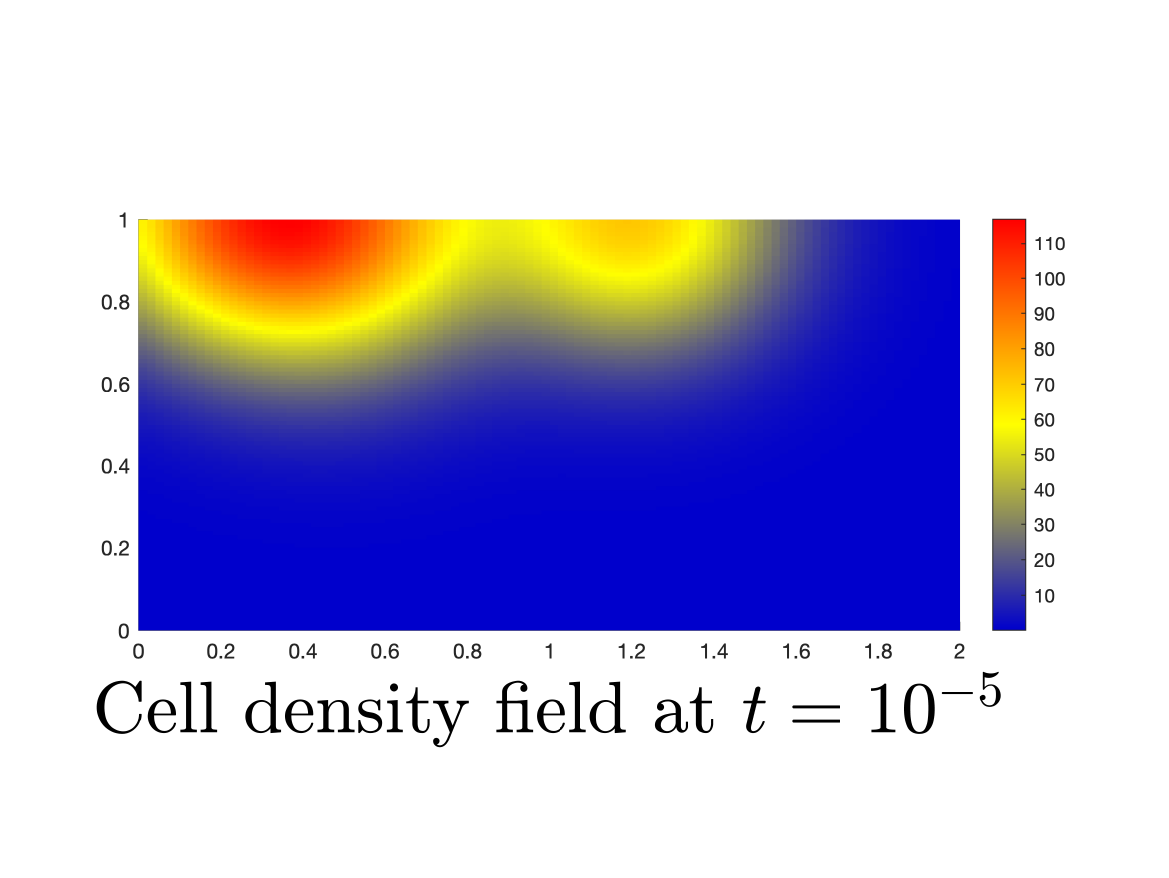}  
		% \caption{t=0}  
		%\label{fig:image1}  
	\end{minipage}  
	\begin{minipage}{0.32\textwidth} % ????????  
		\centering  
		\includegraphics[width=\textwidth]{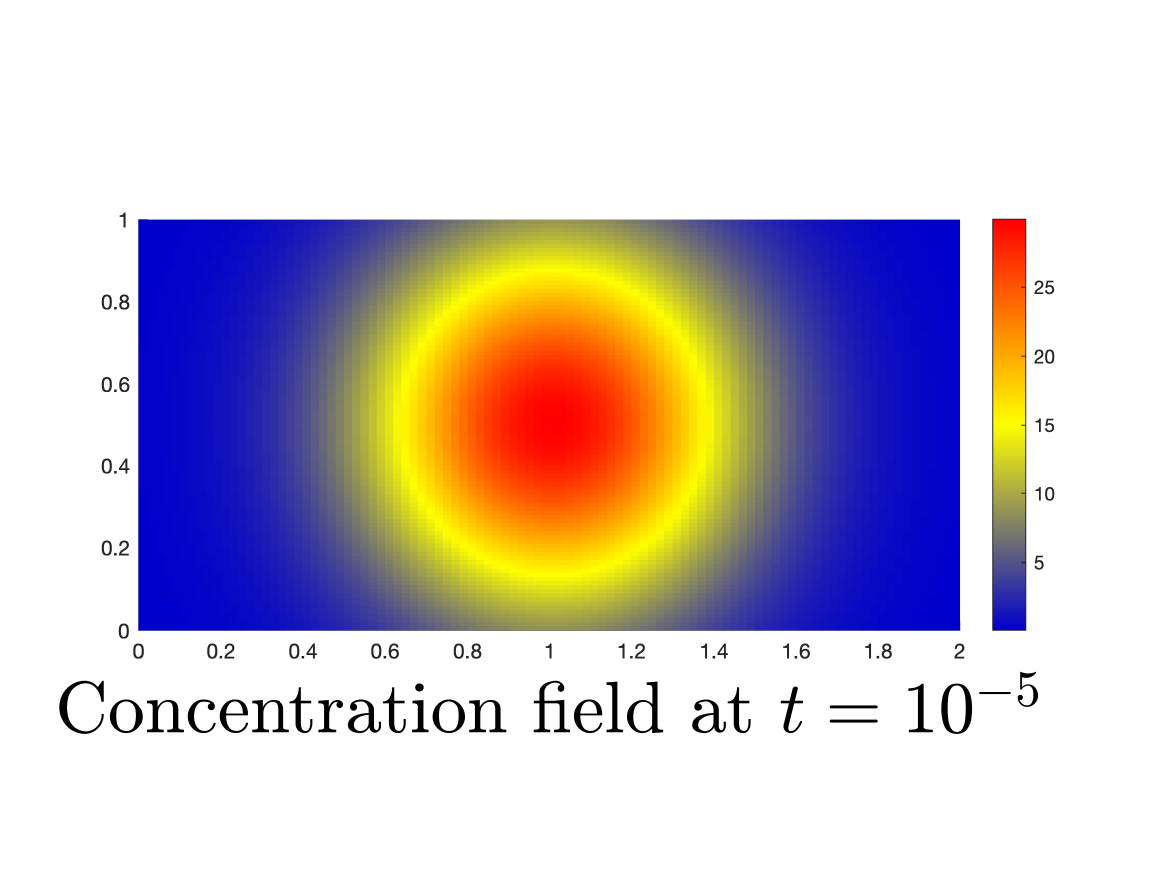} % ?????????
		%  \caption{3}  
		%  \label{fig:image3}  
	\end{minipage}  
	\qquad
	\begin{minipage}{0.32\textwidth}  
		\centering  
		\includegraphics[width=\textwidth]{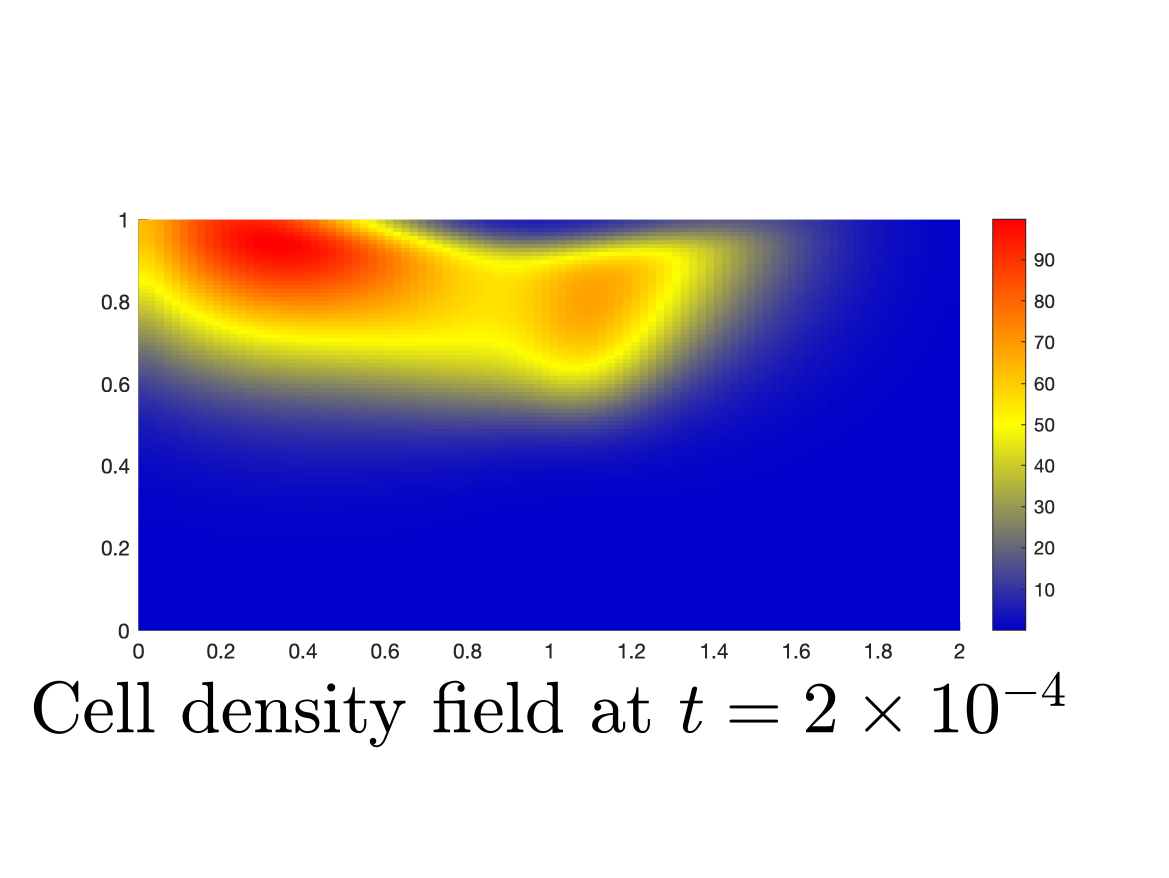}  
		%  \caption{6}  
		%  \label{fig:image6}  
	\end{minipage}  	\begin{minipage}{0.32\textwidth} % ????????  
		\centering  
		\includegraphics[width=\textwidth]{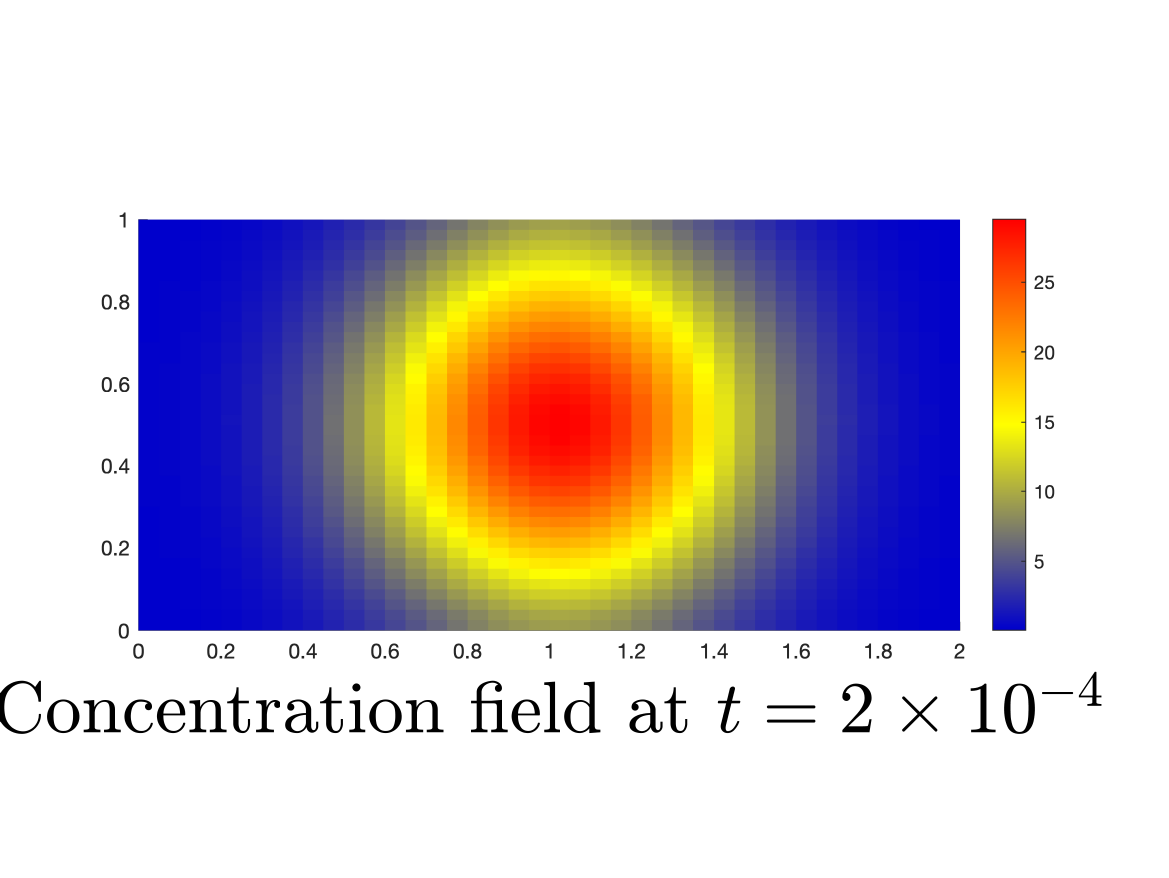}  
		% \caption{t=0}  
		%\label{fig:image1}  
	\end{minipage}  
	\qquad
	\begin{minipage}{0.32\textwidth} % ????????  
		\centering  
		\includegraphics[width=\textwidth]{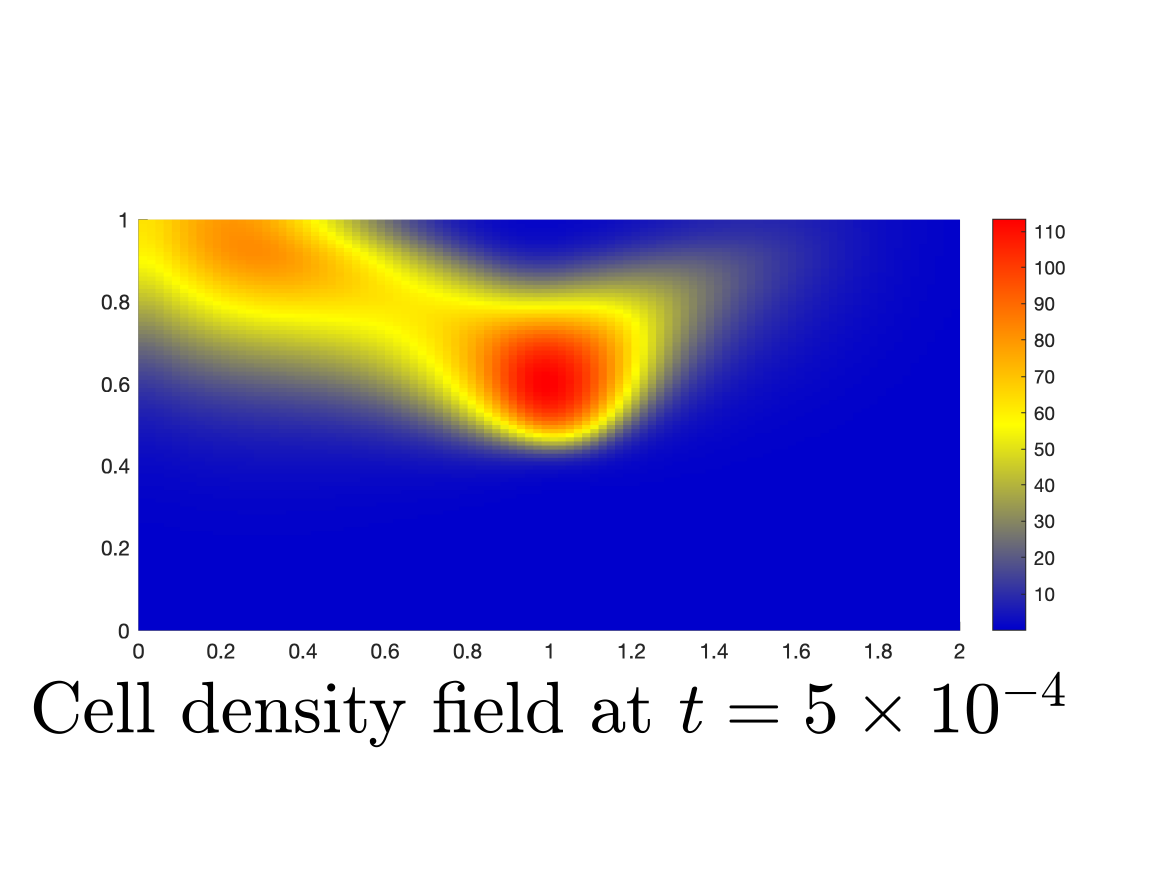} % ?????????
		%  \caption{3}  
		%  \label{fig:image3}  
	\end{minipage}  
	\begin{minipage}{0.32\textwidth}  
		\centering  
		\includegraphics[width=\textwidth]{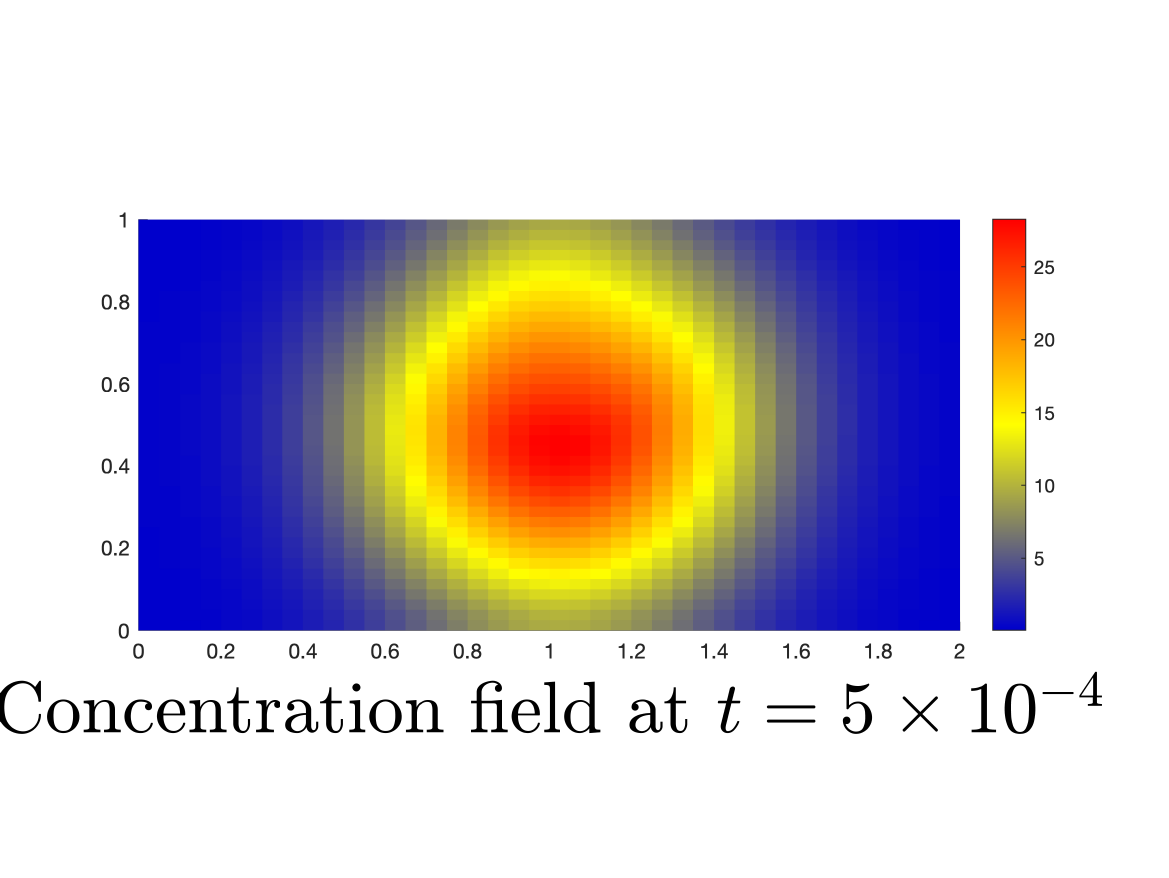}  
		%  \caption{6}  
		%  \label{fig:image6}  
	\end{minipage}  
	\qquad
		\begin{minipage}{0.32\textwidth} % ????????  
		\centering  
		\includegraphics[width=\textwidth]{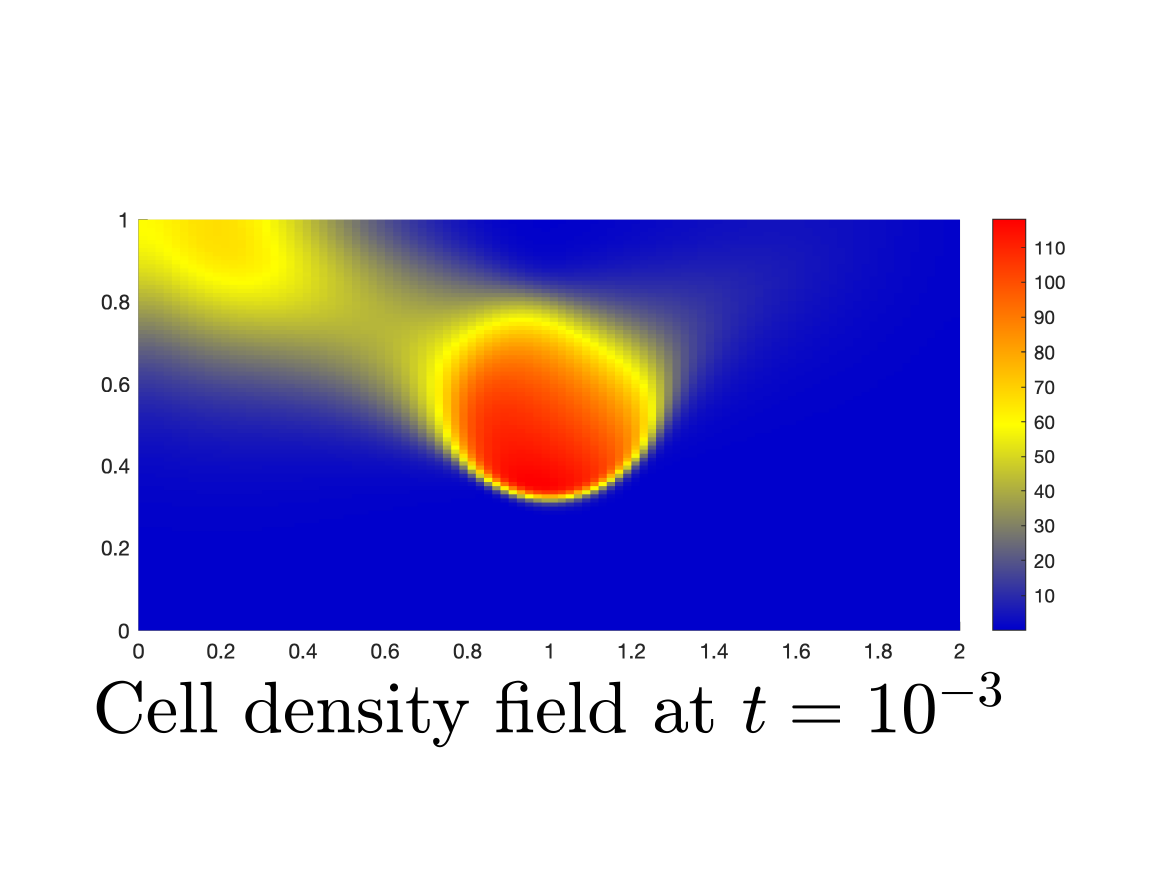} % ?????????
		%  \caption{3}  
		%  \label{fig:image3}  
	\end{minipage}  
	\begin{minipage}{0.32\textwidth}  
		\centering  
		\includegraphics[width=\textwidth]{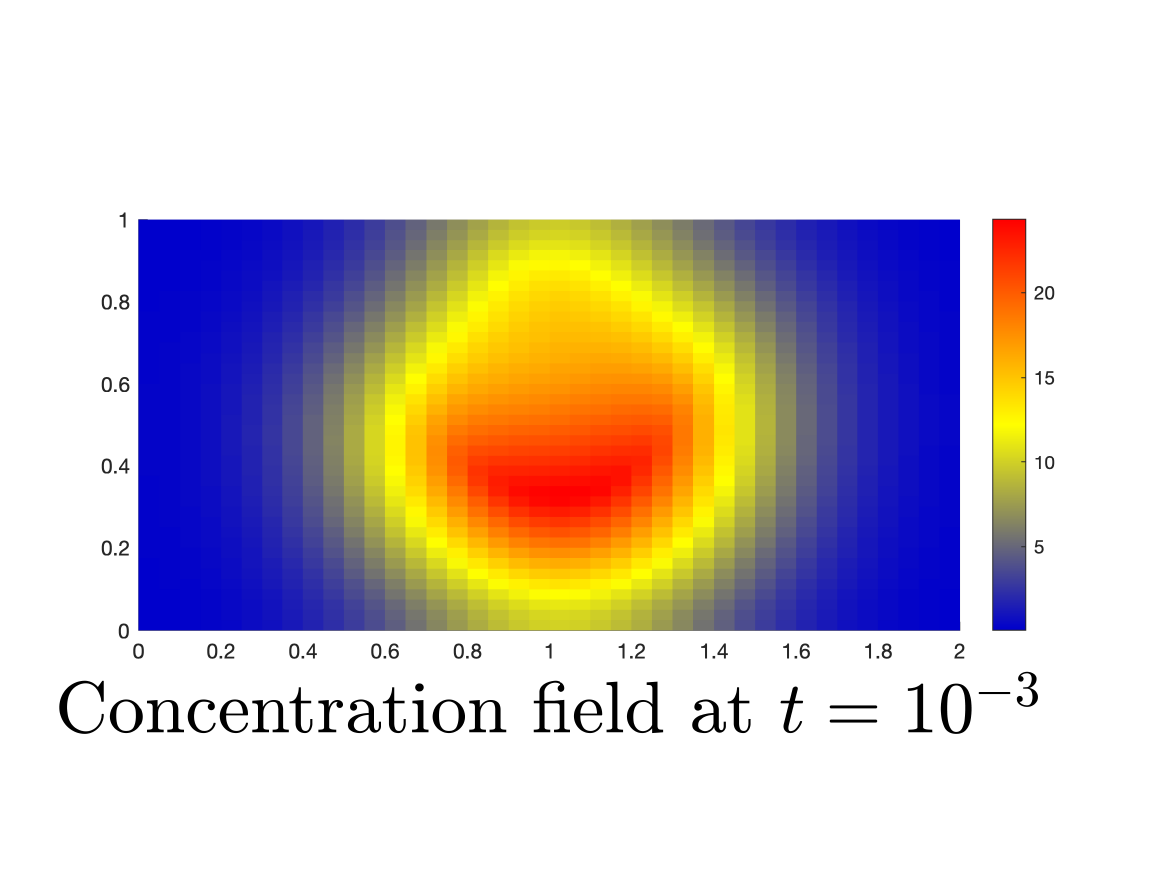}  
		%  \caption{6}  
		%  \label{fig:image6}  
	\end{minipage}  
	\qquad
		\begin{minipage}{0.32\textwidth} % ????????  
		\centering  
		\includegraphics[width=\textwidth]{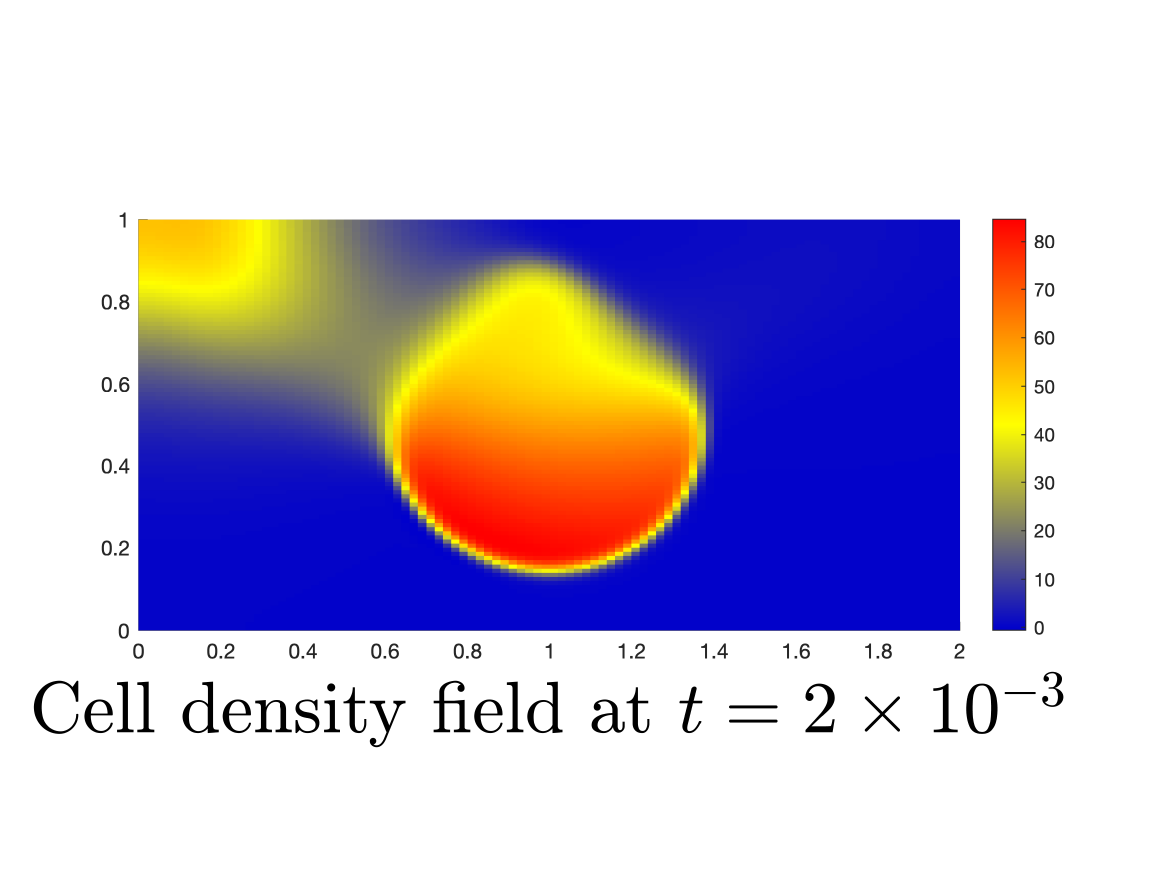} % ?????????
		%  \caption{3}  
		%  \label{fig:image3}  
	\end{minipage}  
	\begin{minipage}{0.32\textwidth}  
		\centering  
		\includegraphics[width=\textwidth]{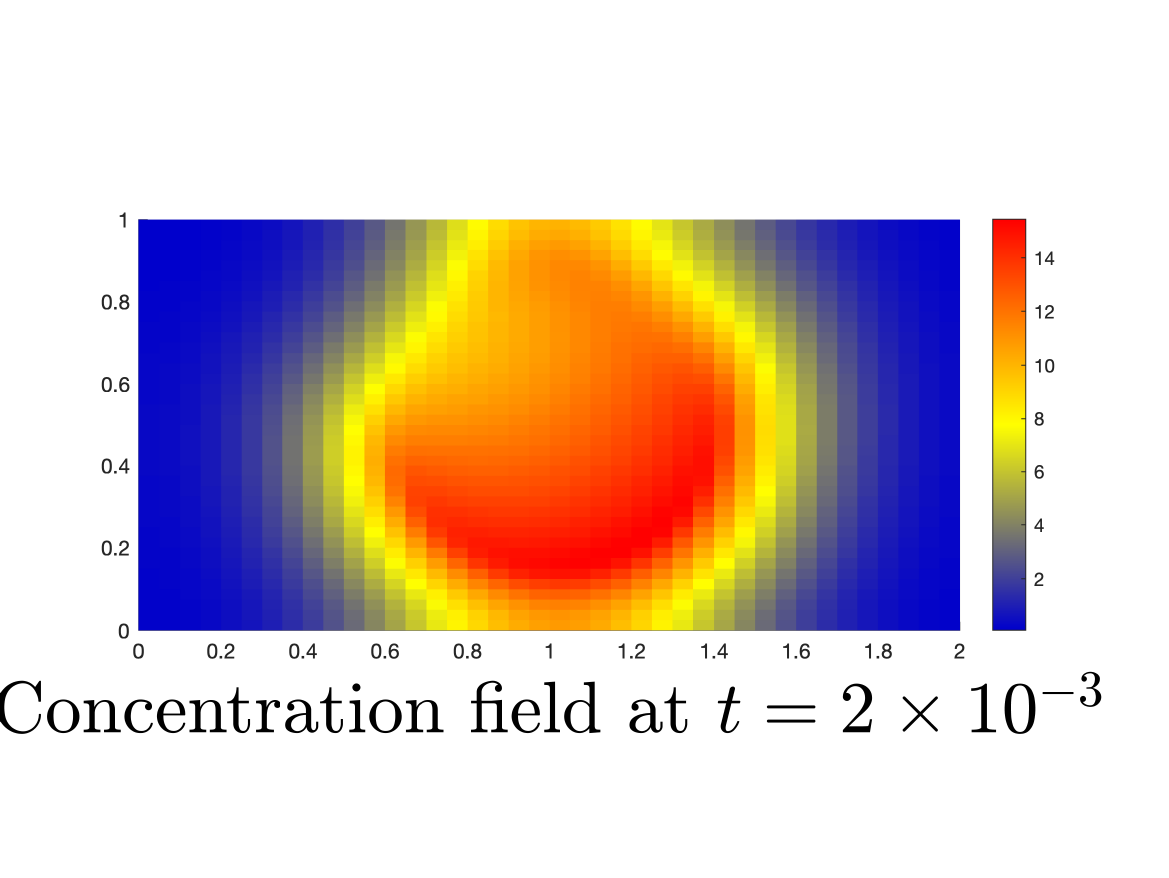}  
		%  \caption{6}  
		%  \label{fig:image6}  
	\end{minipage}  
	\qquad
		\begin{minipage}{0.32\textwidth} % ????????  
		\centering  
		\includegraphics[width=\textwidth]{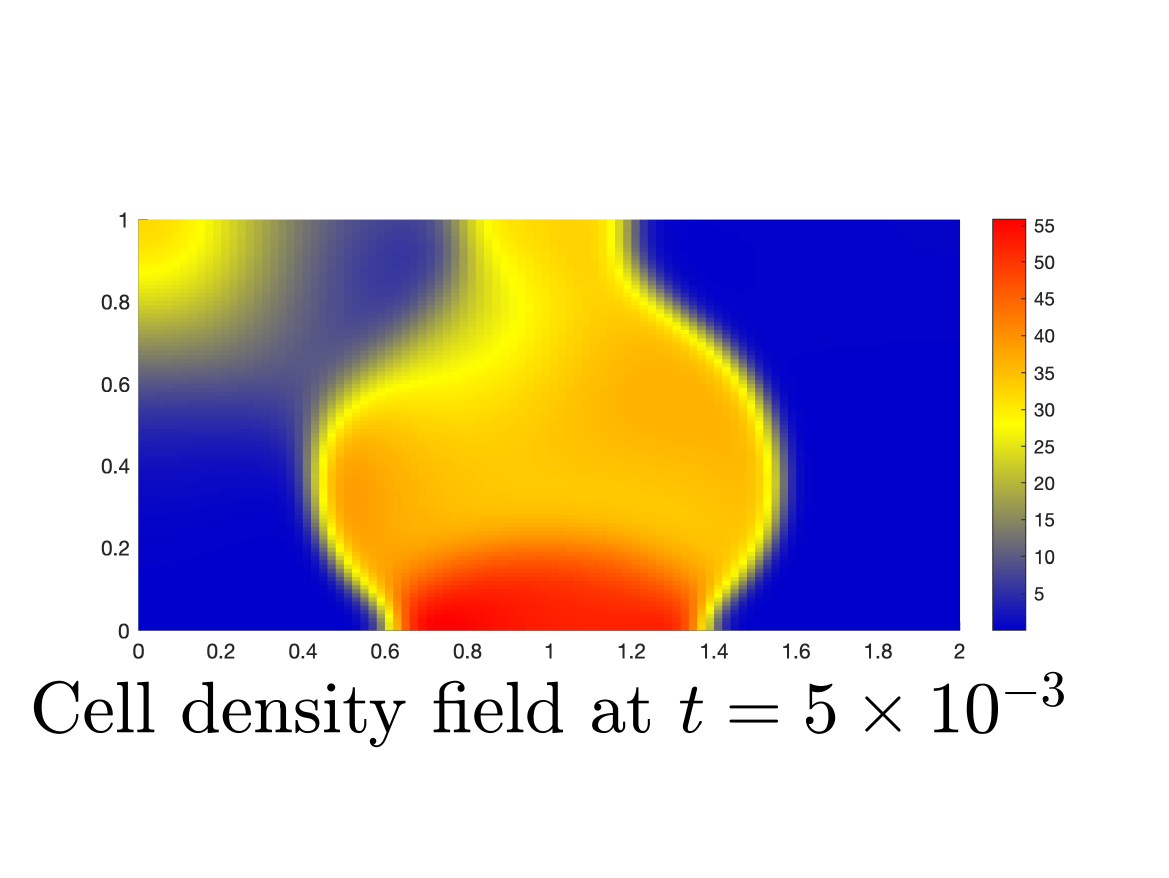} % ?????????
		%  \caption{3}  
		%  \label{fig:image3}  
	\end{minipage}  
	\begin{minipage}{0.32\textwidth}  
		\centering  
		\includegraphics[width=\textwidth]{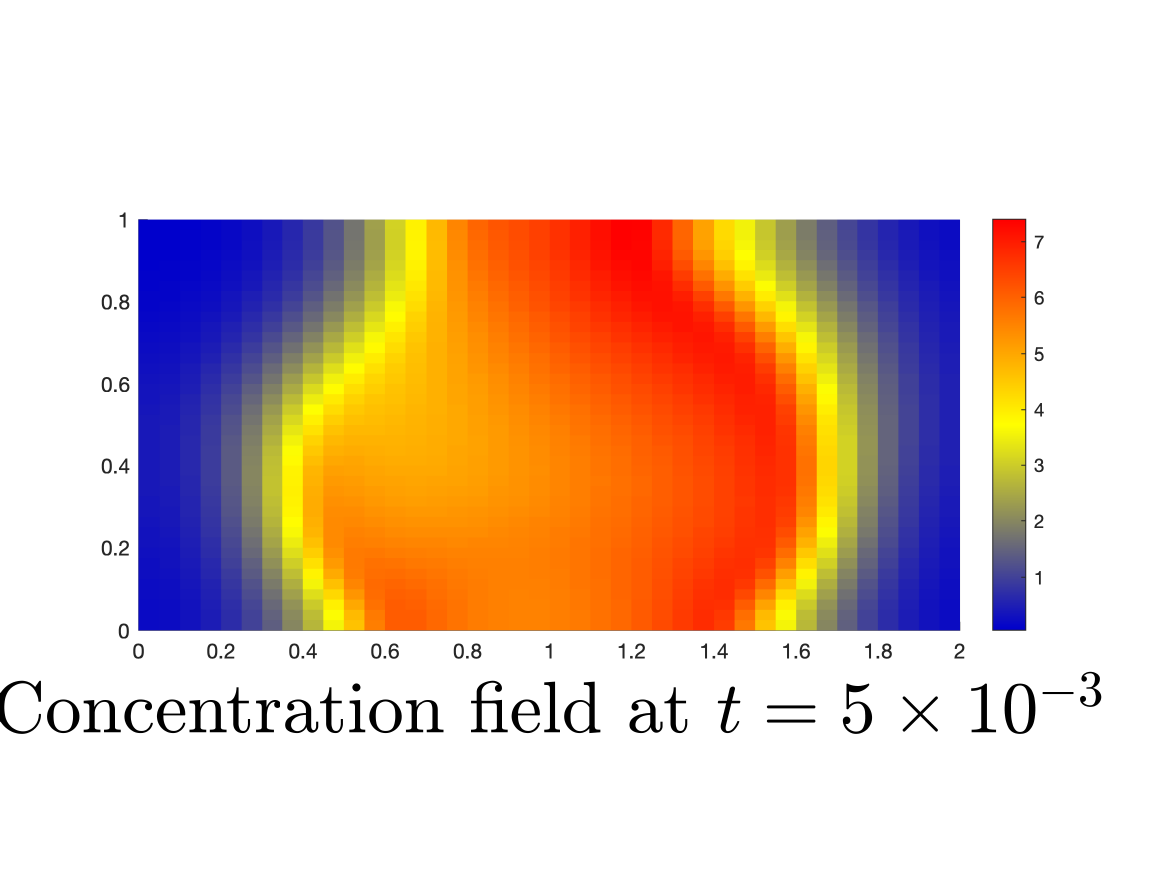}  
		%  \caption{6}  
		%  \label{fig:image6}  
	\end{minipage}  
	% \vspace{1.0em}
	\caption{Cell density vs. chemical concentration at $t= 10^{-5}, 2 \times 10^{-4}, 5 \times10^{-4},10^{-3},2 \times10^{-3},5 \times10^{-3}$.}  
	\label{chemo-celldensity}  
\end{figure}

\begin{figure}[htbp] % htbp???????????????  
	\centering % ???????  
	\begin{minipage}{0.45\textwidth} % ????????  
		\centering  
		\includegraphics[width=\textwidth]{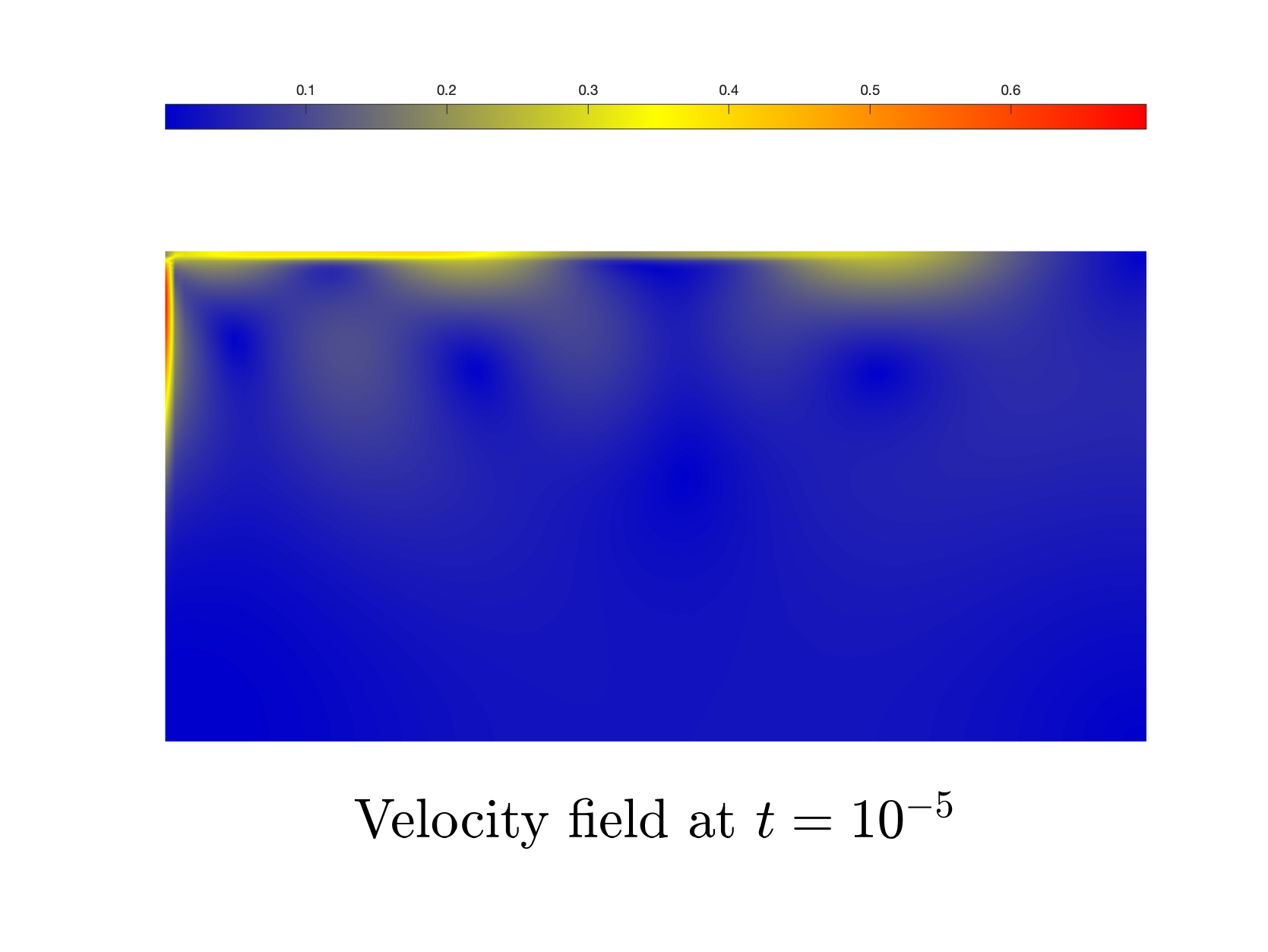}  
		% \caption{t=0}  
		%\label{fig:image1}  
	\end{minipage}  
	\begin{minipage}{0.45\textwidth} % ????????  
		\centering  
		\includegraphics[width=\textwidth]{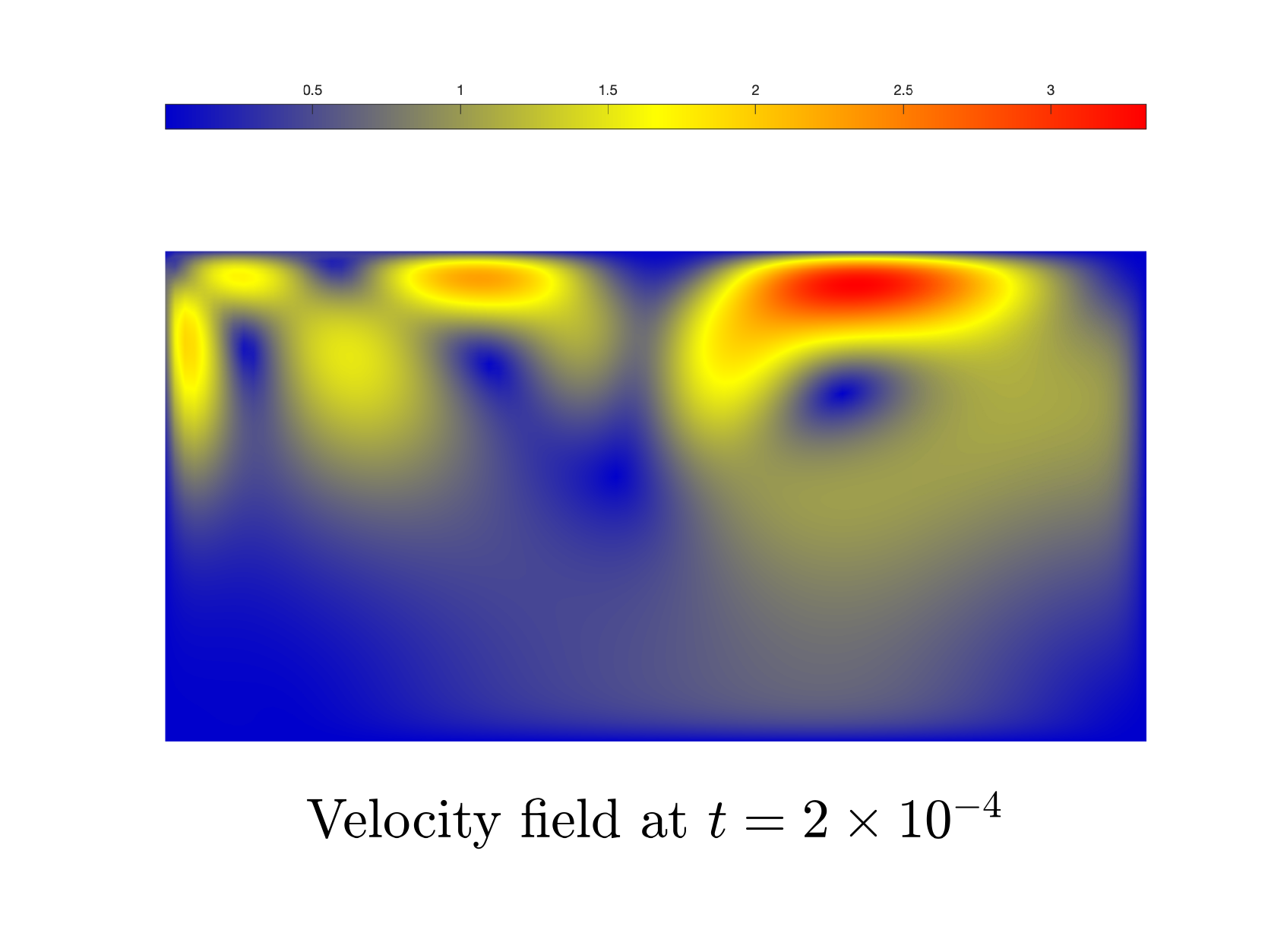} % ?????????
		%  \caption{3}  
		%  \label{fig:image3}  
	\end{minipage}  
	\qquad
	\begin{minipage}{0.45\textwidth}  
		\centering  
		\includegraphics[width=\textwidth]{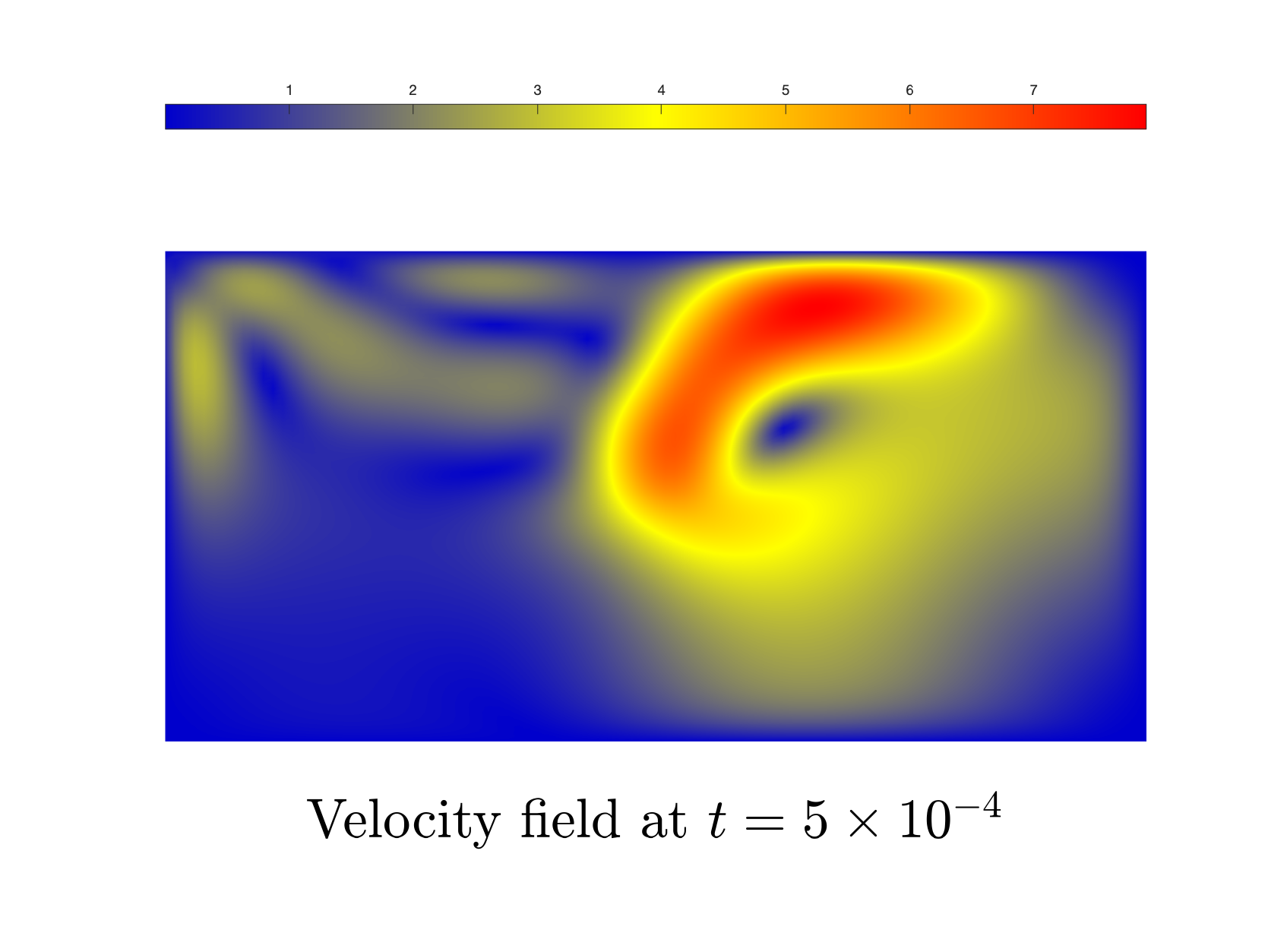}  
		%  \caption{6}  
		%  \label{fig:image6}  
	\end{minipage}  	\begin{minipage}{0.45\textwidth} % ????????  
		\centering  
		\includegraphics[width=\textwidth]{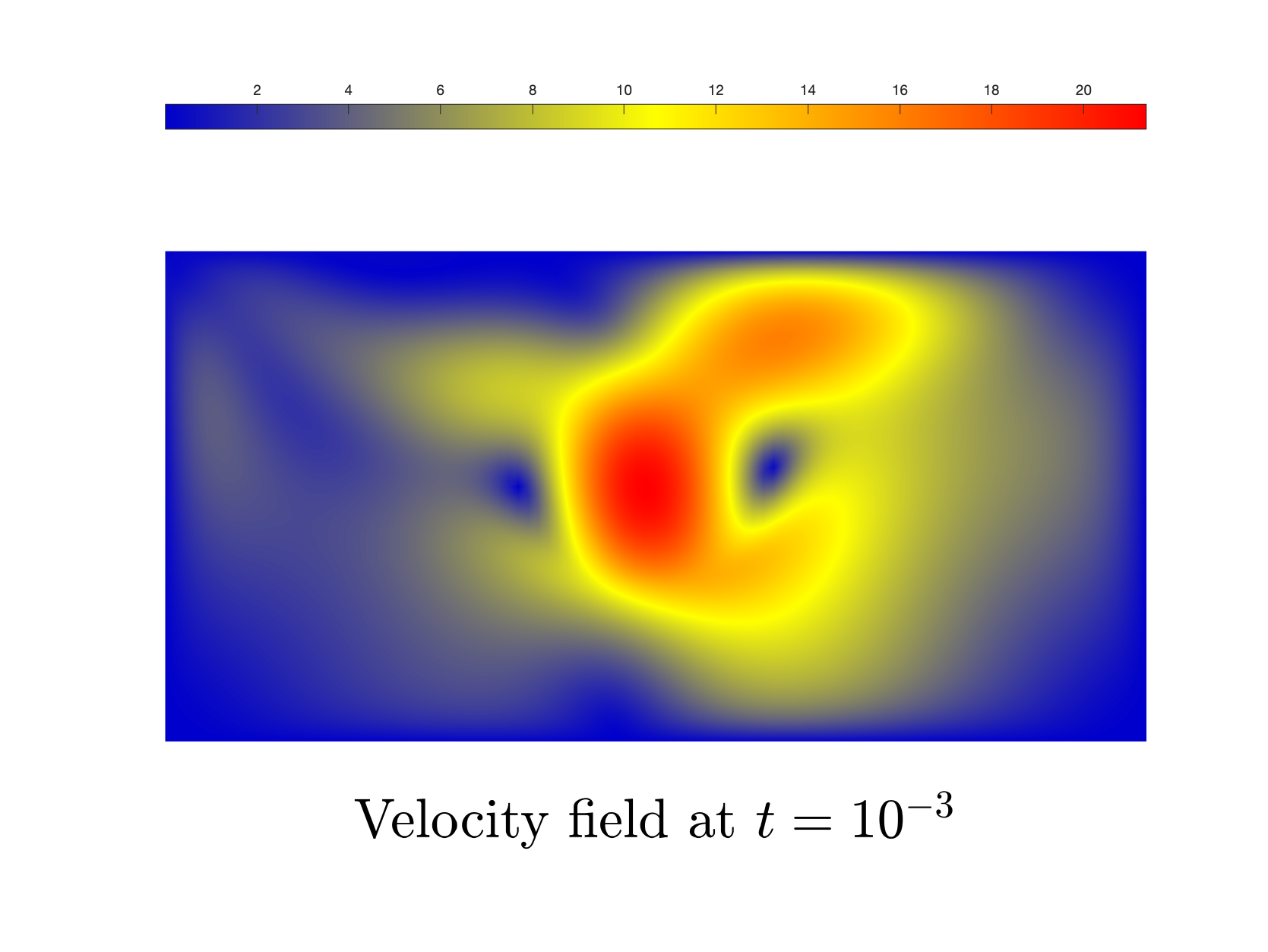}  
		% \caption{t=0}  
		%\label{fig:image1}  
	\end{minipage}  
	\qquad
	\begin{minipage}{0.45\textwidth} % ????????  
		\centering  
		\includegraphics[width=\textwidth]{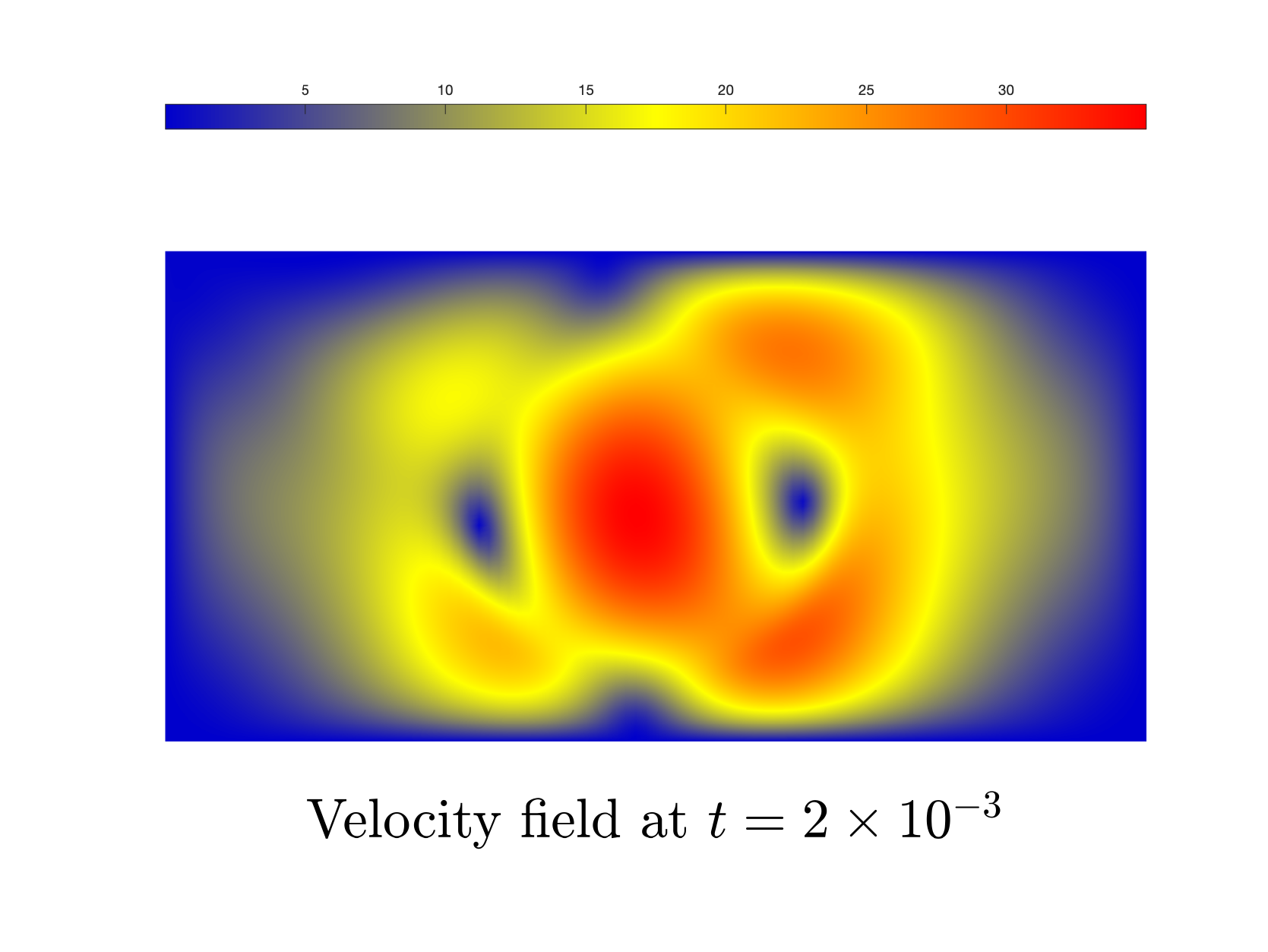} % ?????????
		%  \caption{3}  
		%  \label{fig:image3}  
	\end{minipage}  
	\begin{minipage}{0.45\textwidth}  
		\centering  
		\includegraphics[width=\textwidth]{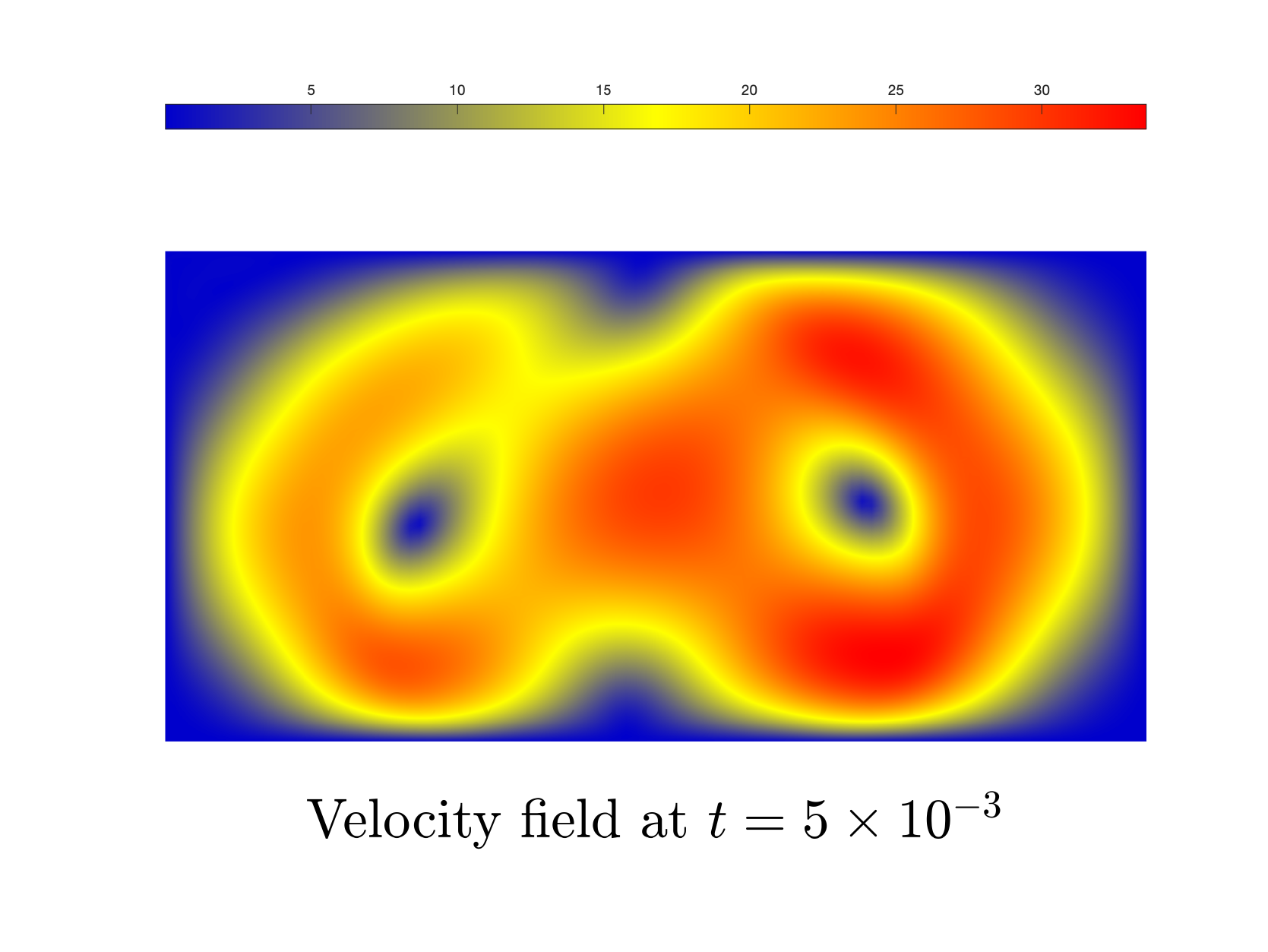}  
		%  \caption{6}  
		%  \label{fig:image6}  
	\end{minipage}  
	\caption{Evolution of the velocity field of the fluid at $t= 10^{-5}, 2 \times 10^{-4}, 5 \times10^{-4},10^{-3},2 \times10^{-3},5 \times10^{-3}$.}  
\label{chemo-velocity}  
\end{figure}
%%%%%%%%%%%%%%%%%%%%%%%%%%%%%%%%%%%%%%%%%%%%%%%%%%%%%%%%%%%%%%%%%%%%%%%%%%%%%%%%%%%%%%%%%%%%%%%%%%%%%%%%%%%%%%%%%%%%%%%%%%%%%%%%%%%%%%%%%%%%%%%%%%%%%%%%%%%%%%%%%%%%%%%%%%%%%%%%%%%%%%%%%%%%%%%%%%%%%%%%%%%%%%%%%%%%%%%%%%%%%%%%%%%%%%%%%%%%%%%%%%%%%%%%%%%%%%%%%%%%%%%%%%%%%%%%%%%%%%%%%%%%%%%%%%%%%%%%%%%%%%%%%%%%%%%%%%%%%%%%%%%%%%%%%%%%%%%%%%%%%%%%%%%%%%%%%%%%%%%%%%%%%%%%%%%%%%%%%%%%%%%%%%%%%%%%%%%
\section{Conclusion}\label{conclusion}

In this paper, we have proposed a fully discrete finite element method based on a first-order  pressure-correction projection scheme for the time-dependent Chemotaxis–Navier–Stokes system. The Mini finite element was used to approximate the velocity-pressure pair, while standard linear Lagrangian elements were adopted for the cell density and chemical concentration. Nonlinear terms were treated semi-implicitly to balance stability and accuracy.
We have rigorously derived  error estimates for the discrete velocity and chemical concentration in both the \(L^2\)-norm under suitable regularity assumptions. 
%Our theoretical analysis confirms the mass conservation and energy stability properties of the proposed method, which are crucial in preserving the physical fidelity of chemotaxis-fluid models. Furthermore, numerical results have been presented to validate the theoretical findings and demonstrate the robustness and accuracy of the scheme.
To the best of our knowledge, this is the first work that provides a complete error analysis for a finite element discretization of the Chemotaxis–Navier–Stokes system using a  projection method. This study fills a gap in the existing literature and lays a foundation for further numerical investigations of chemotaxis-fluid interactions.
Future work will be directed towards extending this framework to more complex models, such as bioconvection phenomena arising in the Patlak–Keller–Segel–Navier–Stokes system and chemo-repulsion–fluid systems. These problems involve additional biological mechanisms and richer nonlinear structures, and their accurate and efficient numerical simulation remains an open and meaningful research direction.

\section*{CRediT authorship contribution statement}
Chenyang Li:
Writing -- original draft, Visualization, Validation, Software, Methodology, Conceptualization;
Ping Lin:
Methodology, Conceptualization;
Haibiao Zheng:
Methodology, Conceptualization;

\section*{Data availability}
Data will be made available on request.

\section*{Declaration of competing interest}
The authors declare that they have no known competing financial interests or personal relationships
that could have appeared to influence the work reported in this paper.
\section*{Acknowledgments}
The authors would like to thank the editor and referees for their valuable comments and suggestions
which helped us to improve the results of this paper. This work was supported by National Natural Science Foundation of China (No. 12471406) and the Science
and Technology Commission of Shanghai Municipality (Grant Nos. 22JC1400900, 22DZ2229014).
%%%%%%%%%%%%%%%%%%%%%%%%%%%%%%%%%%%%%%%%%%%%%%%%%%%%%%%%%%%%%%%%%%%%%%%%%%%%%%%%%%%%%%%%%%%%%%%%%%%
%%%%%%%%%%%%%%%%%%%%%%%%%%%%%%%%%%%%%%%%%%%%%%%%%%%%%%%%%%%%%%%%%%%%%%%%%%%%%%%%%%%%%%%%%%%%%%%%%%%%%%%%%%%%%%%%%%%%%%%%%%%%%%%%%%%%%%%%%%%%%%%%%%%%%%%%%
\section*{Date availability statement}
%%%%%%%%%%%%%%%%%%%%%%%%%%%%%%%%%%%%%%%%%%%%%%%%%%%%%%%%%%%%%%%%%%%%%%%%%%%%%%%%%%%%%%%%%%%%%%%%%%%%%%%%%%%%%%%%%%%%%%%%%%%%%%%%%%%%%%%%%%%%%%%%%%%%%%%%%%%%%%%%%%%%%%%%%%%%%%%%%%%%%%%%%%%%%%%%%%%%%%%%%%%%%%%%%%%%%%%%%%%%%%%%%%%%
The datasets generated and  analyzed during the current study are available from the corresponding author upon reasonable request. 
%%%%%%%%%%%%%%%%%%%%%%%%%%%%%%%%%%%%%%%%%%%%%%%%%%%%%%%%%%%%%%%%%%%%%%%%%%%%%%%%%%%%%%%%%%%%%%%%%%%%%%%%%%%%%%%%%%%%%%%%%%%%%%%%%%%%%%%%%%%%%%%%%%%%%%%%%%%%%%%%%%%%%%%%%%%%%%%%%%%%%%%%%%%%%%%%%%%%%%%%%%%%%%%%%%%%%%%%%%%%%%%%%%%%

%%%%%%%%%%%%%%%%%%%%%%%%%%%%%%%%%%%%%%%%%%%%%%%%%%%%%%%%%%%%%%%%%%%%%%%%%%%%%%%%%%%%%%%%%%%%%%%%%%%%%%%%%%%%%%%%%%%%%%%%%%%%%%%%%%%%%%%%%%%%%%%%%%%%%%%%%%%%%%%%%%%%%%%%%%%%%%%%%%%%%%%%%%%%%%%%%%%%%%%%%%%%%%%%%%%%%%%%%%%%%%%%%%%%
\end{document}